\documentclass[10pt,english,reqno]{amsart}
\usepackage[T1]{fontenc}
\usepackage[latin9]{inputenc}
\usepackage{a4wide}
\usepackage{verbatim} 
\usepackage{amsthm}
\usepackage{amssymb}
\usepackage{setspace}
\usepackage{enumerate}
\usepackage{fancyhdr}
\usepackage{xcolor}
\usepackage{upgreek} 
\usepackage{mathabx}
\usepackage[colorlinks=true]{hyperref}
\hypersetup{linkcolor=red,citecolor=blue,filecolor=dullmagenta,urlcolor=blue}

\numberwithin{equation}{section}
\theoremstyle{plain}
\newtheorem{thm}{Theorem}[section]
\newtheorem{lem}[thm]{Lemma}
\newtheorem{prop}[thm]{Proposition}
\newtheorem{cor}[thm]{Corollary}
\newtheorem{claim}[thm]{Claim}
\newtheorem{rem}{Remark}[thm]

\theoremstyle{definition}

\theoremstyle{remark}

\newcommand{\R}{\mathbb{R}}
\newcommand{\C}{\mathbb{C}}

\newcommand{\I}{\mathcal{I}(t)}
\newcommand{\LL}{\mathcal{L}}
\newcommand{\B}{\mathcal{B}}
\newcommand{\PL}{\partial^2_x\LL}
\newcommand{\J}{\mathcal{J}(t)}
\newcommand{\M}{\mathcal{M}(t)}
\newcommand{\N}{\mathcal{N}(t)}
\newcommand{\Op}{(1-\partial_x^2)}

\newcommand{\vA}{\varphi_A}

\newcommand{\IOpg}{(1-\gamma \partial_x^2)^{-1}}
\newcommand{\Int}[1]{\int_{#1}}
\newcommand{\Jap}[2]{\langle {#1}, {#2} \rangle}

\newcommand{\sgn}{\operatorname{sgn}}
\newcommand{\tv}{\tilde{v}}
\newcommand{\lips}[2]{\widecheck{#1}_{#2}=#1_{#2}-\widetilde{#1}_{#2}}
\newcommand{\ch}[1]{\widecheck{#1}}

\makeatother

\usepackage{babel}
\addto\shorthandsspanish{\spanishdeactivate{~<>}}

\addto\captionsenglish{}
\addto\captionsenglish{}
\addto\captionsenglish{}
\addto\captionsenglish{}
\addto\captionsenglish{}
\addto\captionsenglish{}
\addto\captionsenglish{}
\addto\captionsspanish{}
\addto\captionsspanish{}
\addto\captionsspanish{}
\addto\captionsspanish{}
\addto\captionsspanish{}
\addto\captionsspanish{}
\addto\captionsspanish{}

\newcommand{\sech}{\operatorname{sech}}

\begin{document}
	\title[Asymptotics for Good Boussinesq solitons]{Asymptotic stability manifolds for solitons in the generalized Good Boussinesq equation}
	\author[Christopher Maul\'en]{Christopher Maul\'en}  
	\address{Departamento de Ingenier\'{\i}a Matem\'atica and Centro
de Modelamiento Matem\'atico (UMI 2807 CNRS), Universidad de Chile, Casilla
170 Correo 3, Santiago, Chile.}
	\email{cmaulen@dim.uchile.cl}
	\thanks{Ch.M. was partially funded by Chilean research grants FONDECYT 1191412, and CONICYT PFCHA/DOCTORADO NACIONAL/2016-21160593 and CMM ANID PIA AFB170001.}

\keywords{Generalized Boussinesq Boussinesq, decay, virial}

	\begin{abstract}
		We consider the generalized Good-Boussinesq model in one dimension, with power nonlinearity and data in the energy space $H^1\times L^2$. This model has solitary waves with speeds $-1<c<1$. When $|c|$ approaches 1, Bona and Sachs showed orbital stability of such waves. It is well-known from a work of Liu that for small speeds solitary waves are unstable. In this paper we consider in more detail the long time behavior of zero speed solitary waves, or standing waves. By using virial identities, in the spirit of Kowalczyk, Martel and Mu\~noz, we construct and characterize a manifold of even-odd initial data around the standing wave for which there is asymptotic stability in the energy space. 
	\end{abstract}
	\maketitle
	\tableofcontents

	\section{Introduction}

	\subsection{Setting}

	 In the 1870's,  J. Boussinesq \cite{Bou1} deduced a system of equations to describe two-dimensional irrotational and inviscid fluids in a uniform rectangular channel with flat bottom. He was the first to give a favorable explanation to the traveling-waves, solitons, or solitary waves solutions discovered by Scott Rusell thirty years earlier \cite{SR}, which remained in their form and travelled with constant velocity.

	 \medskip
	 
	 In a first order approximation, Boussinesq's matrix model reduces to a scalar, fourth order model
	 \begin{align}
		  \partial_t^2 \phi - \partial^4_x \phi-\partial^2_x \phi+\partial_{x}^2(f(\phi))=0, \label{eq:BB}
		\end{align}
	However, this model, known as the bad Boussinesq equation, is strongly linearly ill-posed. Consequently, in order to repair this problem,
	  the following equation was proposed \cite{Z,MK}:
		\begin{align}
		  \partial_t^2 \phi+\partial^4_x \phi-\partial^2_x \phi+\partial_{x}^2(f(\phi))=0. \label{eq:GB}
		\end{align}
		Here the physical model considers the nonlinearity as quadratic, i.e. $f(\phi)=\phi^2$ and $\phi(t,x)$ is a real-valued function. This model is called good Boussinesq, and if formally $u=\phi$ and $v=\partial_x^{-1}\partial_t \phi$, this model has the following representation as $2\times2$ system:
		\begin{align}\label{eq:gGB}
		(g\mbox{GB)}\ \ \ \ 
		\begin{cases}
		\partial_t u=\partial_x v \\
		\partial_t v=\partial_x (-\partial_x^2 u+u-f(u)).
		\end{cases}
		\end{align}
This will be the exact model worked in this paper, which is Hamiltonian, and has the following associated conserved quantities:	
\begin{equation}\label{eq:energy}
\begin{aligned}
E[u,v] =&~{}  \frac12 \int \big[v^2 +u^2 +(\partial_x u)^2-2F(u)\big]\ \ \  &(\mbox{Energy}) ,\quad \\
 P[u,v] =&~{}\int uv \quad \quad  &(\mbox{Momentum}) .
\end{aligned}
\end{equation}	
(Here $\int$ means $\int_{\mathbb R} dx$.) These laws define a standard energy space $(u,v)\in H^1\times L^2$. As well as the Korteweg-de Vries (KdV) equation, ($g\mbox{GB}$) is considered as a canonical model of shallow water waves, see \cite{whitman}. In addition, ($g\mbox{GB}$) arises in the so-called "nonlinear string equation" describing small nonlinear oscillations in an elastic beam (see \cite{FST}). 

\medskip
The study of the Boussinesq-type equations has increased recently, mainly due to the versatility of these models when describing nonlinear phenomena. There are several authors that focus on the good Boussinesq equation. The fundamental works
Bona and Sachs \cite{Bona-Sachs}, using abstract techniques of Kato, proved that the Cauchy problem is locally and globally well-posed for small data, and showed the existence of solitary waves for velocities $c^2<1$.
Linares \cite{Linares,Notes_linares}, using Stricharz estimates, proved that the Cauchy problem is globally well-posed in the energy space in the case of small data.
Kishimoto \cite{kishimoto}, in the case of a quadratic nonlinearity, proved that the Cauchy problem is  globally well-posed in $H^{s}(\R)$,  for $s\leq -1/2$, and ill-posed for $s<-1/2$. In \cite{MPP}, it was proved that small solutions in the energy space must decay to zero as time tends to infinity in proper subsets of space. Recently, Charlier and Lenells \cite{CL} developed the inverse scattering transform and a Riemann-Hilbert approach for the quadratic ($g\mbox{GB}$), which is integrable. In general, solitons (solitary waves in integrable equations) are stable objects. However, this is not the case of good Boussinesq (similar to Klein-Gordon). Indeed, small perturbations of solitons may decay or form singularities in finite time, see \cite{FST,Liu,BZ,Z}. 
\medskip
 
In this paper, we are motivated by the long time behavior problem for solitary waves of the gGB \eqref{eq:GB} in the case where $f(s)=|s|^{p-1}s$ for $p>1$. A solitary wave is a solution to \eqref{eq:GB} of the form
\[
(u,v)=(Q_c,-cQ_c)(x-ct-x_0), \quad |c|<1, \quad x_0\in\R,
\]
with $Q_c$ solving $ (c^2-1) Q_c + Q_c'' + f(Q_c)=0$ in $H^1(\mathbb R)$. This interesting question has attracted the attention of several authors before us, showing that the behavior of solitary waves in the standard energy space $H^{1}\times L^2$ is not an easy problem. 
Bona and Sachs \cite{Bona-Sachs}, applying the theory developed by Grillakis, Shatath and Strauss (see \cite{GSS_stability}),
 proved that solitary waves are stable if the speed $c$ obeys the condition $(p-1)/4<c^2<1$ and $p>4$.
Li, Ohta, Wu and Xue \cite{LOWX2020} proved the orbital instability in the degenerate case $1<p<5$ and speed $c=(p-1)/4$. Additionally, 
Kalantarov and Ladyzhenkaya in \cite{KZ} proved that solutions associated to 
 initial data with nonpositive energy may blow up in some sense. Inspired by this work, Liu \cite{Liu} showed that there are solutions with initial data arbitrarily near the ground state ($c=0$) that blow up in  finite time.



		\subsection{Standing waves}
In the case that $f$ is a pure power nonlinearity of the form $f(s)=|s|^{p-1}s$ for $p>1$, it is well-known that (up to shifts) standing solitary waves have the form
	\begin{equation}
	u(t,x)=Q(x)=\left(\dfrac{p+1}{2\cosh^{2}\left(\frac{p-1}{2}x\right)}\right)^{1/(p-1)}, \quad v(t,x)=0.
	\end{equation}
	Here, $Q$ satisfies the equation
	\begin{equation}
	Q''(x)-Q(x)+f(Q(x))=0 \label{eq:soliton_eq}.
	\end{equation}
	Let us consider a perturbation in  \eqref{eq:gGB} of $Q$ of the form
	\[
		u(t,x)=Q(x)+w(t,x),\ \ \ v(t,x)=z(t,x).
	\]
	Then one can see that this perturbation satisfies the following linear system at first order:
		\begin{align} \label{eq:eq_lin}
		\begin{cases}
		\partial_t w= \partial_x z \\
		\partial_t z
				= \partial_x \LL w, 
		\end{cases}
		\end{align}
where
\begin{align}\label{eq:LL}
\LL(w)=-\partial_x^2w+V_0(x)w,
\quad \mbox{with} \quad V_0(x)=1-f'(Q).
\end{align}
$\LL$ is the classical Schr\"odinger operator associated to the soliton $Q$.  This operator has been extensively studied in  \cite{S_NLS} for instance.

\medskip

Therefore, from \eqref{eq:eq_lin} one has $\partial_t^2 w= \PL w$. Consequently, for the well-understanding of the problem we require to  study the fourth order operator $-\PL$, much in the spirit of the fundamental works by Pego and Weinstein results \cite{PW,PW2}.
In Appendix \ref{A}, we will prove the following: for any $p>1$, the linear operator 
\begin{equation}\label{eq:PL}
-\PL(u)=\partial^4_x u-\partial_{x}^2u+\partial_{x}^2(pQ^{p-1}u),
\end{equation}
has a unique eigenfunction $\phi_0(x)$ associated to a negative first eigenvalue $-\nu_0^2<0$, satisfying
	\begin{equation}\label{eq:phi0}
		\langle \partial_x^{-1} \phi_0, \partial_x^{-1}\phi_0 \rangle=1,\ \ \ -\PL(\phi_0)=-\nu_0^2\phi_0, \quad |\phi_0(x)|\lesssim e^{-1^{-}|x|}.
	\end{equation}
Note that we also have $\partial_x^{-1} \phi_0$ well-defined, exponentially decreasing and part of $L^2$. Here $\langle \cdot ,\cdot \rangle$ is the inner product in $L^2(\R)$, and $1^-$ is a number slightly below 1. The second eigenvalue of $-\PL$ is 0 but it is also a resonance in the classical sense (in $L^\infty\backslash L^2$), but the unique $L^2$ eigenvalue is $\phi_1(x)=c_1 Q'(x).$ Therefore, by the Spectral Theorem, orthogonal to $\phi_0$ the operator $-\PL$ is nonnegative. See Appendix \ref{A} for more details and full proofs of all the previous statements.

Let
\begin{equation}\label{eq:Y_Z}
	\boldsymbol{Y}_{\pm}=\left(
	\begin{array}{c}
	\phi_0\\
	\pm\nu_0 \partial_x^{-1}\phi_0
	\end{array}\right)
	, \ \ \ \
	\boldsymbol{Z}_{\pm}=\left(
		\begin{array}{c}
		\partial_x^{-2}\phi_0\\
		\pm\nu_0^{-1} \partial_x^{-1}\phi_0
		\end{array}\right).
\end{equation}
These are even-odd functions, i.e. the first coordinate is even and the second odd
(see Appendix \ref{cor:even_eigenfunction}). The functions $\textbf{u}_{\pm}(t,x)=e^{\pm \nu_0 t }\textbf{Y}_{\pm}(x)$ are solutions of the linearized problem \eqref{eq:eq_lin}, showing the presence of exponentially stable and unstable linear manifolds relevant for the dynamics of nonlinear solutions in a neighborhood of the soliton.

\medskip

In that follows, we refers to \emph{global solution} of \eqref{eq:gGB} to a function $C([0,\infty), H^{1}\times L^2)$ that satisfies \eqref{eq:gGB} for all $t\geq 0$.

\subsection{Main results}

It is not difficult to realize that \eqref{eq:gGB} preserves the even-odd parity in its variables $(u,v)$. In this paper, we will prove that any even-odd small perturbation  
of the static soliton ($c=0$) in the energy space, under certain orthogonality condition,
 is orbitally stable and in fact, it is 
 (locally) asymptotically stable. Furthermore, we will construct a manifold of initial data such that the associated solutions are orbitally stable in $H^{1}\times L^{2}$, and locally asymptotically stable in the space $L^2\cap L^{\infty}$. Our first result is:

\begin{thm}\label{thm1}
Let $p\geq 2$. There exists $\delta>0$ such that if a global even-odd solution $(\phi,\partial_t\partial_{x}^{-1}\phi)$ of \eqref{eq:gGB} satisfies for all $t\geq 0$,
\begin{equation}
\|(\phi,\partial_t\partial_{x}^{-1} \phi)(t)-(Q,0) \|_{H^1(\R)\times L^2(\R)}<\delta ,\label{eq:condition_globalsolution}
\end{equation}
then, for any $\gamma>0$ small enough and any compact interval $I$ of $\R$,
\begin{equation}\label{eq:local_stability}
\lim_{t\to +\infty} \left( \| \phi(t)-Q \|_{L^2(I)\cap L^\infty(I)}+\| (1-\gamma \partial_x^2)^{-1} \partial_t \phi(t) \|_{L^2(I)} \right)=0.
\end{equation}
\end{thm}

This is, as far as we understand, the first description of the standing wave dynamics in the Good Boussinesq model, which is unstable by nature. Clearly the data under which \eqref{eq:condition_globalsolution} is satisfied is not empty, the soliton $(Q,0)$ being its most important representative. However, \eqref{eq:condition_globalsolution} cannot define an open set in the energy space as simple as in some stable, subcritical dynamics, such as KdV. Our second result will describe the manifold of initial data leading to \eqref{eq:condition_globalsolution}, but first we need to clarify some remarks.

\begin{rem}[On the lack of decay of derivatives] Estimate \eqref{eq:local_stability} provides a clean and clear description of the local decay of $\phi(t)$ in the Lebesgue spaces $L^2\cap L^\infty$. However, no clear description of the derivative $\partial_x\phi(t)$ has been found, which remains an interesting open problem. 
\end{rem}

\begin{rem}[On the $\partial_t\partial_{x}^{-1}\phi$ term] We have been unable to provide a clean description of decay for the second component of the Good Boussinesq system. This is due to some deep problems present at the level of the dynamics. However, \eqref{eq:local_stability} provides additional information on the decay of a suitable modification of the second variable. The constant $\gamma$ depends on $\delta$, but it can be taken arbitrarily small if needed. 
\end{rem}

\begin{rem}[About general data]
The construction performed in this paper uses in several steps the parity of the data. Extending our results to general data is a challenging problem, mainly because one needs to introduce shifts that may affect in a strong fashion the dynamics. We hope to consider this problem in a forthcoming publication. 
\end{rem}

\begin{rem}[About the condition $p\geq 2$]
The condition $p\geq 2$ is of technical type, and it is needed to ensure a control on the unstable direction, sufficiently good for our purposes. We believe that the situation for $p$ close to 1 may be very complicated because of the weak decay of the amplitude associated to the unstable direction.  
\end{rem}

The following result provides a description of the manifold of initial data leading to global solutions for which \eqref{eq:condition_globalsolution}  holds.

\medskip

Let $\delta_0>0$, and let $\mathcal{A}_0$ be  the manifold given by
\begin{equation}\label{eq:A0_variedad}
\mathcal{A}_0=
\left\{ 
\boldsymbol{\epsilon} \in H^1(\R)\times L^2(\R) \vert \ \boldsymbol{\epsilon} \mbox{ is even-odd }, \Vert \boldsymbol{\epsilon} \Vert_{H^1\times L^2}<\delta_0\mbox{ and }\langle \boldsymbol{\epsilon}, \boldsymbol{Z}_{+}\rangle=0  \right\}.
\end{equation}


\begin{thm}\label{thm2}
Let $p>1$. There exist $C, \delta_0>0$ and a Lipschitz function $h:\mathcal{A}_0\to \R$ with $h(0)=0$ and $|h(\epsilon )|\leq C \|\epsilon \|_{H^1\times L^2 }^{3/2}$ such that, denoting
\begin{equation}\label{eq:M_variedad}
\mathcal{M}=\left\{ (Q,0)+\epsilon +h(\epsilon ) Y_{+}  \mbox{ with } \epsilon\in  \mathcal{A}_0\right\},
\end{equation}
the following holds:
\begin{enumerate}
	\item If $\boldsymbol{\phi}_0\in \mathcal{M}$ then the solution of \eqref{eq:gGB} with initial data $\boldsymbol{\phi}_0$ is global and satisfies, for all $t\geq 0$,
	\begin{equation}\label{eq:condition_close}
	\| \boldsymbol{\phi}(t)-(Q,0) \|_{H^1(\R)\times L^2 (\R)}\leq C \|\boldsymbol{\phi}_0-(Q,0) \|_{H^1(\R)\times L^2(\R)}.
	\end{equation}
	\item If a global even-odd solution $\boldsymbol{\phi}$ of \eqref{eq:gGB} satisfies, for all $t\geq 0$,
	\begin{equation}
	\| \boldsymbol{\phi}(t)-(Q,0)\|_{H^1(\R)\times L^2(\R)}\leq \frac12 \delta_0,
	\end{equation}
	then for all $t\geq 0$, $\phi(t)\in \mathcal{M}$.
\end{enumerate}
\end{thm}

\begin{rem}[About blow-up] Liu \cite{Liu} showed that initial data $(u_0,v_0)$ for which $E(u_0,v_0)<0$, or $E(u_0,v_0)\geq 0$ and less than a particular function of $\hbox{Im} \int \partial_x^{-1}u_0v_0$ (which is zero in our case), lead to blow up solutions in finite time. In our case, we work with perturbation of the soliton $(Q,0)$. One can easily check that $E(Q,0)= \frac{p-1}{2(p+1)}\int Q^{p+1}>0$, therefore we are not in the blow-up regime determined by Liu. 
\end{rem}

\begin{rem}[Extension to other models]
We believe that our results open the door to the understanding of long time solitary wave dynamics in several other Boussinesq models. We mention for instance the asymptotic stability of abcd solitary waves, at least in the zero speed even data case \cite{BCS1,BCS2}, and the more involved case of the Improved Boussinesq solitary wave; see \cite{MaMu} for further details on this challenging problem.
\end{rem}

\subsection{Idea of the proof}
%
The proofs in this paper follow the lines of the ideas used recently by Kowalczyk, Martel and Mu\~noz in \cite{KMM} to understand the unstable soliton dynamics in the nonlinear Klein-Gordon equation, and  by Kowalczyk, Martel, Mu\~noz and Van Den Bosch \cite{kink} to study the stability properties of kinks for (1+1)-dimensional nonlinear scalar field theories. 

\medskip

More precisely, the proofs are based in a series of localized virial type arguments, similar to the ones used in \cite{ACKM,AM,KMM,kink,KMM2017,Martel-Merle1,MM_solitonsKdV}.
In our case, we will use a combination of virials to obtain the integrability in time of the $L^{2}\times L^{2}$-norm of $(  \phi(t)-Q , (1-\gamma \partial_x^2)^{-1} \partial_t \phi(t) )$, for any $\gamma>0$ small enough, and in any  compact interval $I$, i.e., 
\[
\int_{0}^{\infty} \left( \| \phi(t)-Q \|_{L^2(I)}^2 +\| (1-\gamma \partial_x^2)^{-1} \partial_t \phi(t) \|_{L^2(I)}^2 \right)dt <\infty.
\]
However, some important issues, not present in the previously mentioned works \cite{KMM,kink} will appear along the proofs. The beginning of the proof is similar to \cite{KMM}: The first step is to decompose the solution close to the solitary waves in an adequate way.  We will consider  $(u_1,u_2)\in H^1\times L^2$ be an even-odd perturbation of the solitary waves, which are in some sense orthogonal to $\boldsymbol{Y}_+$ and $\boldsymbol{Y}_-$, and the flow on these directions: for $a_1,a_2$ unique,
\begin{equation*}
\begin{cases}
u(t,x)= Q(x)+a_1(t)\phi_{0}(x)+u_{1}(t,x),\\
v(t,x)= a_2(t)\nu_0 \partial_{x}^{-1}\phi_{0}(x)+u_{2}(t,x).
\end{cases}
\end{equation*}
Then, we will focus on $(u_1,u_2)\in H^1\times L^2$, which satisfy the linearized equation \eqref{eq:eq_lin}. Following \cite{MPP}, for an adequate weight function $\varphi_A$ placed at scale $A$ large, we obtain the virial estimate
\begin{equation}\label{eq:I_intro}
\begin{aligned}
\dfrac{d}{dt}\int \varphi_{A}(x) u_1 u_2
\leq&-\dfrac{1}{2} \int \left[ w_2^2 +2(\partial_x w_1)^2
+\left(1- C_1A^{-1} \right)w_1^2 \right] \\& +C_1 a_1^4 +  C_1\int \sech\left(x \right) u_1^2,
\end{aligned}
\end{equation}
where $(w_1,w_2)$ is localized version of $(u_1,u_2)$ at $A$ scale, and $C_1$ denotes a fixed constant.  This virial estimate has no good sign because of the term $C_1\int \sech\left(x \right) u_1^2$. Then we require to transform the system to a new one which has better virial estimates, in the spirit of  Martel \cite{Martel_linearKDV}. For any $\gamma>0$ small enough, we  define new variables $(v_1,v_2)\in H^{1}\times H^{2}$ by
\begin{equation*}
	\begin{cases}
	v_1= (1-\gamma \partial_x^2)^{-1}\LL  u_1,\\
	v_2= (1-\gamma \partial_x^2)^{-1}  u_2.
	\end{cases}
\end{equation*}
(see \eqref{eq:change_variable}). Note that $(v_1,v_2)\in H^1\times H^2$, which is bad news because of the lack of a correct regularity order in the variables. This will cause problems later on. However, the new system for $(v_1,v_2)$ (see \eqref{eq:syst_v}) satisfies, for an adequate weight function $\psi_{A,B}$, $B\ll A$, the virial estimate
	\begin{equation}\label{eq:J_intro}
	\begin{aligned}
	\dfrac{d}{dt}\int \psi_{A,B} v_1 v_2 \leq & -\frac{C_2}{2} \int \left[z_1^2+(V_0(x)-B^{-1}) z_2^2 +2(\partial_x z_2)^2 \right] \\
	& +B^{-1} \bigg(\| w_1\|_{L^2}^2
	+\| w_2\|_{L^2}^2 \bigg)
	+ |a_1|^3,
	\end{aligned}
	\end{equation}
where $(z_1,z_2)$ is a lozalized version of $(v_1,v_2)$,  at the smaller scale $B$, $V_0$ given by \eqref{eq:LL}, and $C_2$ denotes a fixed constant.

\medskip

Following \cite{KMM}, in order to combine estimates \eqref{eq:I_intro} and \eqref{eq:J_intro} we need an estimate for the last term in \eqref{eq:I_intro}. However, unlike previous works, here we have the following coercivity estimate in terms of the variables $(w_1,w_2)$ and $(z_1,z_2)$:
	\begin{equation}\label{eq:sech_u1_intro}
	\begin{aligned}
	\int \sech ( x)  u_1^2
	\lesssim & ~{}  B^{-1/2} \left( \| w_1\|_{L^2} +\|\partial_x w_1\|_{L^2}^2 \right)
		+B^{1/2} \|z_1\|_{L^2}^2 +B^{-4}\| \partial_x z_1\|_{L^2}^2.
	\end{aligned}
	\end{equation}
We can directly observe that the term $\partial_x z_1$ does not appears in \eqref{eq:J_intro}, leading to the main obstruction present in this paper. This problem is deeply related to the fact that $(v_1,v_2)\in H^1\times H^2$, i.e., the new variables are in opposed order of regularity. 

\medskip

In order to overcome this problem, we introduce a series of modifications that will allow us to close estimates \eqref{eq:I_intro} and \eqref{eq:J_intro} properly. First, we must gain derivates. In a new virial estimate for the system of $(\partial_x v_1,\partial_x v_2)$ (see \eqref{eq:syst_vx}), we obtain the third virial estimate
		\begin{equation}\label{eq:M_intro}	
		\begin{aligned}
	 \frac{d}{dt} \int \psi_{A,B} \partial_x v_1 \partial_x v_2 
	 \leq & ~{}  
	-\dfrac{1}{2}\int \left( (\partial_x z_1)^2+\left(V_0(x) -  C_3B^{-1}\right) (\partial_x z_2)^2+2(\partial_x^2 z_2)^2\right)	\\
	&+ C_3\| z_2\|_{L^2}^2
	+ C_3 B^{-1}\| z_1\|_{L^2}^2\\
	&+C_3 B^{-1} \left( \| \partial_x w_1\|_{L^2}^2+\| w_1\|_{L^2}^2
	+ \| w_2\|_{L^2}^2 \right)+
	C_3|a_1|^3,
	\end{aligned}
	\end{equation}
with $C_3>0$ fixed.	This new estimate give us local $L^2$ control on $\partial_x z_1$ and $\partial_x^2 z_2$, which was not present before. Finally, our last contribution is a  transfer virial estimate that exchanges information between  $\partial_x z_1$, $\partial_x z_2$ and $\partial_x^2 z_2$, in the form of
		\begin{equation}\label{eq:Dintro}
		\begin{aligned}
		\frac12	\int (\partial_x z_1)^2 \le &~{} \frac{d}{dt} \int \rho_{A,B} \partial_x v_1 v_2
		+\frac32 \int \left[ (\partial_x^2 z_2)^2  + (\partial_x z_2)^2 + z_2^2+z_1^2\right]\\
		&+ C_4 B^{-3} \bigg(\| w_1\|_{L^2}^2+\| w_2\|_{L^2}^2\bigg) +C_4 |a_1|^3.
		\end{aligned}
		\end{equation}
Here $C_4>0$ is fixed and $\rho_{A,B}$ is a suitable weight function. Finally, we consider a functional $\mathcal{H}$ being a well-chosen linear combination of \eqref{eq:I_intro}, \eqref{eq:J_intro}, \eqref{eq:M_intro}, \eqref{eq:sech_u1_intro} and \eqref{eq:Dintro}. We get
	\begin{equation*}
	\begin{aligned}
	\dfrac{d}{dt}\mathcal{H}(t)
			\leq &		
	-C_2 B^{-1} \left(\|w_1\|_{L^2}^2+\|\partial_x w_1\|_{L^2}^2 
	+\| w_2\|_{L^2}^2  \right)
	+C_5 |a_1|^3 , \ \ \mbox{for all } t\geq 0.
	\end{aligned}
	\end{equation*}
This final estimate allows us to close estimates, and prove local decay for $u_1$ after some standard change of variables from $w_j$ to $u_j$.

\subsection*{Organization of this paper} This paper is organized as follows. Section \ref{sec:2} deals with a first virial estimate for a decomposition system, namely \eqref{eq:decomposition}. In Section \ref{sec:3} we introduce the transformed problem and prove first virial estimates on that system. In Section \ref{sec:4} we obtain virial estimates for higher order derivatives of the transformed problem. Section \ref{sec:5} is devoted to a technical transfer estimate dealing with higher order transformed variables. Finally, in Section \ref{sec:6} we prove Theorem \ref{thm1}, and in Section \ref{proof_TH2} we prove Theorem \ref{thm2}.

\subsection*{Acknowledgments} I deeply thank professors Didier Pilod (U. Bergen), Juan Soler (U. Granada), Francisco Gancedo (U. Sevilla) and Miguel A. Alejo (U. C\'ordoba) for the funding and their hospitality during the research stays where this work was completed.

\medskip

\section{A virial identity for the ($g$GB) system}\label{sec:2}

Recall the ($g$GB) system \eqref{eq:gGB}. The first step in our proof is to consider a small {\bf even-odd} perturbation of soliton $(Q,0)$. In what follows we will describe this decomposition, introduce some notation, and develop a virial estimate for the good Boussinesq system.

\subsection{Decomposition of the solution in a vicinity of the soliton}\label{sub:2.1}

Let $(u,v)=(\phi,\partial_t\partial_{x}^{-1}\phi)$ be a solution of \eqref{eq:gGB} satisfying \eqref{eq:condition_globalsolution} for some small $\delta>0$. Using $\boldsymbol{Y_+}$ as in \eqref{eq:Y_Z}, we decompose $(u,v)$
as follows
\begin{equation}\label{eq:decomposition}
\begin{cases}
u(t,x)= Q(x)+a_1(t)\phi_{0}(x)+u_{1}(t,x),\\
v(t,x)= a_2(t)\nu_0 \partial_{x}^{-1}\phi_{0}(x)+u_{2}(t,x),
\end{cases}
\end{equation}
where (see \eqref{eq:phi0})
\[
\begin{aligned}
a_1(t)&=\Jap{u(t)-Q}{\nu_0^{-2}\LL\phi_0}=\Jap{u(t)-Q}{\partial_x^{-2}\phi_0},\\
a_2(t)&=\frac{1}{\nu_0}\Jap{\partial_x v }{\nu_0^{-2}\partial_x\LL\phi_0}=\frac{1}{\nu_0}\Jap{\partial_xv }{\partial_x^{-1}\phi_0},
\end{aligned}
\]
such that 
\begin{equation}\label{eq:orthogonal_condition}
\Jap{u_1(t)}{\partial_x^{-2}\phi_0}
=0=\Jap{u_2(t)}{\partial_x^{-1}\phi_0},
\end{equation}
or equivalently,
\begin{equation}\label{eq:orthogonal_condition2}
\Jap{u_1(t)}{\LL\phi_0}
=0=\Jap{u_2(t)}{\partial_x\LL\phi_0}.
\end{equation}
Orthogonalities \eqref{eq:orthogonal_condition} are nonstandard particular choices motivated by key cancelation properties. See Appendix \ref{A} for a detailed construction of $\partial_x^{-1}\phi_0$ and $\partial_x^{-2}\phi_0$. Setting
\begin{equation}\label{eq:b}
b_{+}=\frac{1}{2}(a_1+a_2),\ \ b_{-}=\frac{1}{2}(a_1-a_2),
\end{equation}
from \eqref{eq:condition_globalsolution}, we have for all $t\in\R_{+}$
\begin{equation}\label{eq:ineq_hip}
\Vert u(t)\Vert_{H^1}+\Vert v(t)\Vert_{L^2}+|a_1(t)|+|a_2(t)|+|b_+(t)|+|b_-(t)|\leq C_0 \delta.
\end{equation}
Moreover, using \eqref{eq:soliton_eq}, \eqref{eq:phi0} and \eqref{eq:orthogonal_condition},  $(a_1,a_2)$ satisfies the following differential system
\begin{equation}\label{eq:motion}
\begin{cases}
\dot{a}_1= \nu_0 a_2 \\
\dot{a}_2=  \nu_0 a_1+\dfrac{N_0}{\nu_0},
\end{cases} \ \ \mbox{or equivalently} \ \ \
\begin{cases}
\dot{b}_{+}= \nu_0 b_{+}+\dfrac{N_0}{2\nu_0}\vspace{0.05in} \\ 
\dot{b}_{-}=-\nu_0 b_{-}-\dfrac{N_0}{2\nu_0}.
\end{cases}
\end{equation}
where
\begin{equation}\label{eq:Nperp}
\begin{aligned}
N=&\ \partial_x\left(f(Q)+f'(Q)(a_1\phi_0+u_1)-f(Q+a_1\phi_0+u_1)\right),\\
N^{\perp}=&\ \  N-N_0\partial_x^{-1}\phi_0, \quad \ \mbox{and} \  \quad N_0=\Jap{N}{\partial_x^{-1} \phi_0}.
\end{aligned}
\end{equation}
Then, $(u_1,u_2)$ satisfies the system
\begin{equation}\label{eq:u_linear}
	\begin{cases} 
	\dot{u}_1= \partial_x u_2 \\
	\dot{u}_2=\partial_x \LL(u_1)+N^{\perp},
	\end{cases}
\end{equation}
with $u_1$ even and $u_2$ odd.

\subsection{Notation for virial argument}
We consider a smooth even function $\chi:\R\to \R$ satisfying
\begin{equation}\label{chichi}
\chi=1 \mbox{ on } [-1,1], \quad \chi=0 \mbox{ on } (-\infty,2]\cup [2,\infty), \quad \chi'\leq 0 \mbox{ on } [0,\infty).
\end{equation}
For $A>0$, we define the functions $\zeta_A$ and $\varphi_A$ as follows
\begin{equation}\label{eq:bound_phiA}
\zeta_A(x)=\exp\left( -\dfrac{1}{A}(1-\chi(x))|x|\right),
\quad \varphi_A(x)=\int_{0}^{x} \zeta_A^2(y)dy, \ \ x\in \R.
\end{equation}
For $B>0$, we also define
\begin{equation}\label{eq:zetaB}
\zeta_B(x)=\exp\left( -\dfrac{1}{B}(1-\chi(x))|x|\right),
\quad \varphi_B(x)=\int_{0}^{x} \zeta_B^2(y)dy, \ \ x\in \R.
\end{equation}
We consider the function $\psi_{A,B}$ defined as
\begin{equation}\label{eq:psi_chiA}
\psi_{A,B}(x)=\chi^2_A(x)\varphi_B(x) \mbox{ where }\ \ \chi_A(x)=\chi\left(\dfrac{x}{A}\right), \ \  x \in \R.
\end{equation}
These functions will be used in two distinct virial arguments with different scales
\begin{equation}\label{eq:scales}
	1\ll B \ll B^{10}  \ll A.
\end{equation}

The following remark will be essential for the well-boundedness of some nonlinear terms in what follow. 
\begin{rem}\label{rem:chiA_zetaA4}
	One can see that for each function $v$
	\begin{equation*}
		\int \chi_A^2 v^2\leq\int_{|x|\leq 2A} v^2 \leq C\int_{|x|\leq 2A} e^{-4|x|/A}v^2 \lesssim 	\int v^2 \zeta_A^4 \leq \| \zeta_A^2 v\|^2_{L^2} .
	\end{equation*}
This estimate will be useful later on (see Subsections \ref{control_tJ1} and \ref{control_J4}).	
\end{rem}

\subsection{Virial estimate}
Set
\begin{equation}\label{eq:I}
\I=\Int{\R} \varphi_A(x) u_1 u_2  ,
\end{equation}
and
\begin{equation} \label{eq:wi}
w_i=\zeta_A u_i, \quad  i=1,2.
\end{equation}

Here, $(w_1,w_2)$ represents a localized version of $(u_1,u_2)$ at scale $A$. The following virial argument has been used in \cite{KMM,kink} in a similar context.
 \begin{prop}\label{prop:virial_I}
There exist $C_1>0$ and $\delta_1>0$ such that for any $0<\delta\leq\delta_1$, the following holds. Fix $A=\delta^{-1}$. Assume that for all $t\geq 0$,  \eqref{eq:ineq_hip} holds. Then for all $t\geq 0$, 
\begin{equation}\label{eq:dI_w}
\begin{aligned}
\dfrac{d}{dt}\I
\leq&-\dfrac{1}{2} \int \left[ w_2^2 +2(\partial_x w_1)^2
+\left(1-C_1A^{-1} \right)w_1^2 \right]  +C_1 a_1^4 +  C_1\int \sech\left(\dfrac{x}{2}\right) w_1^2  .
\end{aligned}
\end{equation}
\end{prop}

Some remarks are in order.

\begin{rem}
This virial has several similarities with the developed in \cite{KMM} for nonlinear Klein-Gordon equation. 
In that paper, the main part of the virial is composed by the $\dot{H}^1$-norm  of $w_1$. In our case, this main part is similar to the $H^{1}\times L^2$-norm of $(w_1,w_2)$, and the rest of the terms are the same. Unlike \cite{KMM}, we did not use a correction term since the momentum of the equation \eqref{eq:energy} works well in this case. This virial was already used in \cite{MPP} in a different context (small solutions around zero).
\end{rem}

The proof of Proposition \ref{prop:virial_I} follows after the next intermediate lemma.

\begin{lem}
Let $(u_1,u_2)\in H^1(\R)\times L^2(\R)$ a solution of \eqref{eq:u_linear}. Consider $\varphi_A =\varphi_A (x)$ a smooth bounded function to be chosen later. Then
\begin{equation}
	\begin{aligned}
   \frac{d}{dt}\I=&- \dfrac{1}{2}\int \vA' \left( u_2^2 + u_1^2
   +3(\partial_x u_1)^2\right) 
   +\dfrac{1}{2}\int  \vA''' u_1^2   \\
   &+\int ( \vA' u_1+\vA \partial_x u_1)  \left(f(Q)+f'(Q)a_1\phi_0-f(Q+a_1\phi_0+u_1)-N_0\partial_x^{-2}\phi_0 \right).
   \end{aligned}
\end{equation}
\end{lem}
\begin{proof}
Taking derivative in \eqref{eq:I} and using \eqref{eq:u_linear},
	\begin{equation}\label{eq:I'_RHS}
	\begin{aligned}
	\dfrac{d}{dt}\I
	=&\int \vA(\dot{u_1}u_2+u_1\dot{u_2}) 
	= \int \vA(\partial_x u_2 u_2+u_1 (\partial_x\LL(u_1)+N^{\perp})) \\
	=& -\frac12 \int \vA' u_2^2  
	+\int \vA u_1 \partial_x\LL(u_1)  
	+\int \vA u_1 N^{\perp}. 
	\end{aligned}
	\end{equation}
For the second integral in the RHS of the above equation, we have
	\begin{equation*}
	\begin{aligned}
	\int \vA u_1 \partial_x\LL(u_1)  
	=& \int \vA u_1 (-\partial_x^3 u_1+\partial_x u_1-\partial_x (f'(Q) u_1))  \\
	=& -\int \vA u_1 \partial_x^3 u_1   +\frac12 \int \vA \partial_x(u_1^2) -\int \vA u_1 \partial_x (f'(Q) u_1)  .
	\end{aligned}
	\end{equation*}
Integrating by parts
	\begin{equation}\label{eq:uPLu}
	\begin{aligned}
	\int \vA u_1 \partial_x\LL(u_1)  =& -\frac12 \int \vA' u_1^2 
	+\int (\vA' u_1+\vA \partial_x u_1) \partial_x^2 u_1  
	-\int \vA  u_1 \partial_x (f'(Q) u_1)  \\
	=& -\frac12 \int  \vA'  \left[ u_1^2+(\partial_x u_1)^2\right]
	+\int \vA' u_1 \partial_x^2 u_1 
		-\int \vA  u_1 \partial_x (f'(Q) u_1)  .
	\end{aligned}
	\end{equation}
Integrating by parts in the second integral in the RHS of the above equation, we get
	\begin{equation*}
	\begin{aligned}
	\int \vA' u_1 \partial_x^2 u_1  
	=& -\int ( \vA'' u_1+ \vA' \partial_x u_1 )\partial_x u_1   \\
	=& -\int \left( \vA'' \dfrac{\partial_x (u_1^2)}{2}+\vA' (\partial_x u_1)^2 \right)   =	-\int \vA' (\partial_x u_1)^2   +\int  \vA''' \dfrac{u_1^2}{2}. 
	\end{aligned}
	\end{equation*}
For the last integral in the RHS of \eqref{eq:I'_RHS}, separating terms and integrating by parts we obtain
	\begin{equation*}
	\begin{aligned}
	\int \vA u_1 N^{\perp}  
	=& \int \vA u_1 \left( \partial_x\left(f(Q)+f'(Q)(a_1\phi_0+u_1)-f(Q+a_1\phi_0+u_1)\right) -N_0 \partial_x^{-1}\phi_0 \right)  \\
	=&-\int (\vA' u_1+\vA \partial_x u_1) \left( f(Q)+f'(Q)a_1\phi_0-f(Q+a_1\phi_0+u_1)\right)\\
	&+\int \vA u_1  \partial_x (f'(Q)u_1)
	 -N_0\int \vA u_1  \partial_x^{-1}\phi_0   .
		\end{aligned}
	\end{equation*}
Cancelling terms, we finally obtain
\begin{equation}
	\begin{aligned}
	\dfrac{d}{dt}\I=& 
	- \int \vA' \left( \dfrac{u_2^2}{2} + \dfrac{u_1^2}{2}
	+(\partial_x u_1)^2 +\dfrac{(\partial_x u_1)^2}{2}\right)  \\
	& +\int  \vA''' \dfrac{u_1^2}{2} 
	-\int \vA u_1  \partial_x (f'(Q) u_1)  \\
	&+\int (\vA' u_1+\vA \partial_x u_1)  \left(f(Q)+f'(Q)a_1\phi_0-f(Q+a_1\phi_0+u_1)\right) \\
	& 	+\int \vA u_1  \partial_x (f'(Q) u_1)  -N_0\int \vA u_1  \partial_x^{-1}\phi_0 \\
	=&- \dfrac{1}{2}\int \vA' \left( u_2^2 + u_1^2
	+3(\partial_x u_1)^2\right) 
	+\dfrac{1}{2}\int  \vA''' u_1^2   -N_0\int \vA u_1  \partial_x^{-1}\phi_0 . \\
	&+\int (\vA' u_1+\vA \partial_x u_1)  \left(f(Q)+f'(Q)a_1\phi_0-f(Q+a_1\phi_0+u_1)\right).
	\end{aligned}
\end{equation}
This concludes the proof.
\end{proof}

Now we rewrite the main part of the virial identity using the new variables $(w_1,w_2)$. 
\begin{lem}\label{lem:2p3}
	It holds
	\begin{equation*}
		\begin{aligned}
		\int \vA' \left( u_2^2 + u_1^2
		+3(\partial_x u_1)^2\right) 
		-\int  \vA''' u_1^2
		=& \int  \left( w_2^2  + 3(\partial_x w_1)^2  
		+\left(1+\dfrac{ \zeta_A''}{\zeta_A} -2\dfrac{(\zeta_A')^2}{\zeta_A^2}\right)w_1^2 \right),
		\end{aligned}
	\end{equation*}
	with
	\begin{equation}\label{eq:zA''-zA2}
	\left|\dfrac{ \zeta_A''}{\zeta_A} -2\dfrac{(\zeta_A')^2}{\zeta_A^2} \right|\lesssim \frac{1}{A}.
	\end{equation}
\end{lem}
\begin{proof}
	Considering $w_i=\zeta_A u_i$, $i=1,2$, and $\vA' =\zeta_A^2$, we have
	\begin{equation}
	\int \vA' \left( u_2^2 + u_1^2\right) 
	=\int \zeta_A^2 \left( u_2^2 + u_1^2\right) 
	=\int \left( w_2^2 + w_1^2\right) .
	\end{equation}
	Also,
	\begin{equation}
	\int \vA' (\partial_x u_1)^2 
	= \int (\partial_x w_1)^2  
	+\int w_1^2\dfrac{ \zeta_A''}{\zeta_A} .
	\end{equation}
	In the case of the last terms we have,
	\begin{equation*}
	\int  \vA''' u_1^2  =\int \dfrac{(\zeta_A^2)''}{\zeta_A^2}w_1^2 
	= 2\int \left(\dfrac{\zeta_A''}{\zeta_A}+\dfrac{(\zeta_A')^2}{\zeta_A^2}\right)w_1^2 .
	\end{equation*}
	By \eqref{eq:bound_phiA}, we have
	\begin{equation}
	\begin{aligned}
	\dfrac{\zeta_A'}{\zeta_A}&=-\dfrac{1}{A}\left[-\chi'(x)|x|+(1-\chi(x)\sgn(x)) \right],\label{eq:zetaA'/zetaA}\\
	\dfrac{\zeta_A''}{\zeta_A}&=\left(\dfrac{\zeta_A'}{\zeta_A}\right)^2+\dfrac{1}{A}\left[\chi''(x)|x|+2\chi'(x)\sgn(x) \right].
	\end{aligned}
\end{equation}
Then, substracting
	\[
	\begin{aligned}
	\dfrac{ \zeta_A''}{\zeta_A}-2\left(\dfrac{\zeta_A'}{\zeta_A}\right)^2
	=&
	-\frac{1}{A^2}\left[-\chi'(x)|x|+(1-\chi(x))\sgn(x) \right]^2+\frac1A \left[\chi''(x)|x|+2\chi'(x)\sgn(x) \right].
	\end{aligned}
	\]
	For $1\leq|x|\leq 2$, one can see that 
	\[
	\left|\dfrac{ \zeta_A''}{\zeta_A} -2\left(\dfrac{\zeta_A'}{\zeta_A}\right)^2\right|
	\lesssim \frac 1A+\frac{1}{A^2}.
	\]
	 For $|x|\geq2$, we have that
	\[
	\left|\dfrac{ \zeta_A'}{\zeta_A} -2\left(\dfrac{\zeta_A'}{\zeta_A}\right)^2\right|
	= \frac{1}{A^2}.
	\]
	Then,
	\[
	\left| \dfrac{ \zeta_A''}{\zeta_A}-2\left(\dfrac{\zeta_A'}{\zeta_A}\right)^2\right|
	\lesssim \left(\frac1A+\frac{1}{A^2}\right)\textbf{1}_{\{|x|\geq 1\}} 
	\lesssim  \frac{1}{A}
	.
	\]
 This ends the proof.
\end{proof}

Next, we deal with the nonlinear terms.

\begin{lem}\label{lem:2p4}
\begin{equation}\label{eq:nonlinearterm_I}
\begin{aligned}
&\left| \int ( \vA' u_1+\vA \partial_x u_1)  \left(f(Q+a_1\phi_0+u_1)-f(Q)-f'(Q)a_1\phi_0 \right) -N_0\int \vA u_1  \partial_x^{-1}\phi_0\right|\\
& \quad\lesssim
a_1^4 +  \int \sech\left(\dfrac{x}{2}\right) w_1^2
+A^2 \Vert u_1 \Vert_{L^{\infty}}^{p-1} \int |\partial_x w_1|^2.
\end{aligned}
\end{equation}
\end{lem}
\begin{proof}
	First, we treat the term $N_0\int \vA u_1  \partial_x^{-1}\phi_0$. Noticing that
	\begin{equation}
	  N_0
	  =\Jap{N}{\partial_x^{-1}\phi_0}
	  =-\Jap{f(Q)+f'(Q)(a_1\phi_0+u_1)-f(Q+a_1\phi_0+u_1)}{\phi_0},
	\end{equation}
	and by Taylor's expansion, one has
	\begin{equation}
	\label{eq:taylorFull}
	|f(Q+a_1\phi_0+u_1)-f(Q)-f'(Q)(a_1\phi_0+u_1)|\lesssim a_1^2 f''(Q) \phi_0^2+f''(Q) u_1^2+|a_1|^{p} \phi_0^p+|u_1|^{p}.
	\end{equation}
	Thus, by exponential decay estimates on $Q$ and  $\phi_0$ (see Appendix \ref{A}), and by \eqref{eq:ineq_hip}, $|a_1|\lesssim1, \Vert u_1\Vert_{L^\infty}\leq \Vert u_1\Vert_{H^1}\lesssim 1 $, it holds
	\begin{equation}
		\begin{aligned}
		|N_0|
		\lesssim  a_1^2 + \int \sech(9x/10) u_1^2,
		\end{aligned}
	\end{equation}
	taking $A\geq 4$, we have
    \begin{equation}
    |N_0|	\lesssim  a_1^2 + \int \sech\left( \dfrac{x}{2}\right) w_1^2.\label{eq:N_0_bounded}
    \end{equation}
	 Noticing that for all $x\in \R$, $|\vA|\leq |x|$,
	 \[
	 |\varphi_A \partial_x^{-1} \phi_0 |\lesssim |x \sech(9x/10)|\leq \left|\sech\left(\frac34 x\right)\right|,
	 \]
	  and using H\"older inequality, we have
	\begin{equation}
	\begin{aligned}\label{eq:bound_N_0_part}
	\left| \int  u_1 \vA \partial_x^{-1} \phi_0 \right|
	\lesssim & ~{}  \left|\int  w_1 \sech \left(\frac x2\right) \right|\\
		\lesssim & ~{}  \left|\int  w_1^2 \sech \left(\frac x2\right) \right|^{1/2} 	 \left|\int  \sech \left(\frac x2\right) \right|^{1/2} \lesssim \left|\int  w_1^2 \sech \left(\frac x2\right) \right|^{1/2} .
	\end{aligned}
	\end{equation}
	We conclude using Cauchy-Schwarz inequality
	\begin{equation*}
		\begin{aligned}
		\left|N_0 \int  \vA  u_1  \partial_x^{-1} \phi_0 \right|
		\lesssim & ~{} |N_0|^2+\left|\int  u_1 \vA \partial_x^{-1} \phi_0 \right|^2
		\lesssim  a_1^4 +  \int \sech\left(\dfrac{x}{2}\right) w_1^2.
		\end{aligned}
	\end{equation*}
	For the remaining terms, we consider the following decomposition
	\[
	\begin{aligned}
		\int  &\left( \varphi_A \partial_x u_1+\varphi_A' u_1\right)
		\left[f(Q+a_1\phi_0+u_1)-f(Q)-f'(Q)a_1\phi_0
		\right]\\
		=& \int  \varphi_A \partial_x\left[F(Q+a_1\phi_0+u_1)-F(Q+a_1\phi_0)-(f(Q)+f'(Q)a_1\phi_0)u_1 \right]\\
		&-\int  \varphi_A Q' \left[ f(Q+a_1\phi_0+u_1)-f(Q+a_1\phi_0)- (f'(Q)+f''(Q)a_1\phi_0)u_1 \right]\\
		&-a_1\int  \varphi_A \partial_x \phi_0 \left[ f(Q+a_1\phi_0+u_1)-f(Q+a_1\phi_0)- f'(Q)u_1\right]\\
		&+\int  \varphi_A' u_1[f(Q+a_1\phi_0+u_1)-f(Q)- f'(Q)a_1\phi_0]\\
		=: &~{} I_1+I_2+I_3+I_4,
	\end{aligned}
	\]
and rewriting as
	\[
	\begin{aligned}
	I_1
	=& -\int  \varphi_A' \left[F(Q+a_1\phi_0+u_1)-F(Q+a_1\phi_0)-F'(Q+a_1\phi_0)u_1-F(u_1) \right]\\
	&-\int  \varphi_A' \left[f(Q+a_1\phi_0)-f(Q)+f'(Q)a_1\phi_0)\right]u_1-\int \varphi_A' F(u_1),
	\end{aligned}
	\]
	\[
	\begin{aligned}
	I_2
	=&-\int  \varphi_A Q' \left[ f(Q+a_1\phi_0+u_1)-f(Q+a_1\phi_0)-f'(Q+a_1\phi_0)u_1  \right]]\\
	&-\int  \varphi_A Q'\left(f'(Q+a_1\phi_0)-f'(Q)+f''(Q)a_1\phi_0\right)u_1,
	\end{aligned}
	\]
			\[
	\begin{aligned}
	I_3
	=&-a_1\int  \varphi_A \partial_x \phi_0 \left[ f(Q+a_1\phi_0+u_1)-f(Q+a_1\phi_0)-f'(Q+a_1\phi_0)u_1\right]\\
	&-a_1\int  \varphi_A \partial_x \phi_0 \left( f'(Q+a_1\phi_0)-f'(Q)\right)u_1,
	\end{aligned}
	\]
	and
	\[
	\begin{aligned}
	I_4
	=&\int  \varphi_A' u_1[f(Q+a_1\phi_0+u_1)-f(Q+a_1\phi_0)-f(u_1)]\\
	&+\int  \varphi_A' u_1[f(Q+a_1\phi_0)-f(Q)- f'(Q)a_1\phi_0]+\int  \varphi_A' u_1 f(u_1).
	\end{aligned}
	\]
	By Taylor expansion, $p\geq 1$, $|a_1|, \Vert u_1 \Vert_{L^{\infty}}\lesssim 1$, we have
	\begin{equation*}
		\begin{aligned}
	\vert F(Q+&a_1\phi_0+u_1)-F(Q+a_1\phi_0)-F'(Q+a_1\phi_0)u_1-F(u_1)\vert, \\
	&\lesssim |Q+a_1\phi_{0}|^{p-1}u_1^2+|Q+a_1\phi_0||u_1|^{p},\\
	&\lesssim |Q+a_1\phi_{0}|^{p-1}u_1^2+|Q+a_1\phi_0||u_1|^{2} \lesssim  \sech(9x/10) u_1^2 \lesssim \sech\left(\dfrac{x}{2}\right) w_1^2.
		\end{aligned}
	\end{equation*}
	Similarly, using \eqref{eq:bound_phiA} and $A\geq 4$, we find the following estimates
	\begin{equation*}
	\begin{aligned}
	\vert \vA Q' \left[f(Q+a_1\phi_0+u_1)-f(Q+a_1\phi_0)-f'(Q+a_1\phi_0)u_1\right]\vert
		&\lesssim
		\sech\left(\dfrac{x}{2} \right)w_1^2, \\
	\vert a_1 \vA \partial_x \phi_0 \left[f(Q+a_1\phi_0+u_1)-f(Q+a_1\phi_0)-f'(Q+a_1\phi_0)u_1\right]\vert
	&\lesssim
	 \sech\left(\dfrac{x}{2} \right) w_1^2, \\
	\vert\vA' u_1 \left[f(Q+a_1\phi_0+u_1)-f(Q+a_1\phi_0)-f(u_1)\right]\vert
		&\lesssim \sech\left(\dfrac{x}{2} \right)w_1^2.
	\end{aligned}
	\end{equation*}
	Furthermore, once again by Taylor expansion, we have
	\begin{equation}
	\begin{aligned}\displaystyle
	&\vert \varphi_A' \left[f(Q+a_1\phi_0)-f(Q)+f'(Q)a_1\phi_0)\right]u_1 \vert \\
	&\quad +\vert \varphi_A Q'\left[f'(Q+a_1\phi_0)-f'(Q)+f''(Q)a_1\phi_0\right]u_1 \vert \\
	&\quad +\vert a_1\varphi_A \partial_x \phi_0 \left[ f'(Q+a_1\phi_0)-f'(Q)\right]u_1\vert \\
	&\quad +\vert \varphi_A' u_1\left[f(Q+a_1\phi_0)-f(Q)- f'(Q)a_1\phi_0\right] \vert\\
	&\lesssim 	\sech\left(\dfrac{x}{2}\right)|a_1|^2|u_1|
	 \lesssim \sech\left(\dfrac{x}{2}\right)|w_1|^2+	\sech\left(\dfrac{x}{4}\right)|a_1|^4.
	\end{aligned}
	\end{equation}
	For the last step, we need the following claim proved in \cite{KMM}.
	\begin{claim}
		It holds
		\[
		\int  \zeta_A^2 |u_1|^{p+1}= \int  \zeta_A^{-p+1} |w_1|^{p+1}\lesssim A^2 \Vert u_1 \Vert_{L^{\infty}}^{p-1} \int  |\partial_x w_1|^2.
		\]
	\end{claim}
	Using this claim, we have
	\begin{equation*}
	\left| \int \varphi_A' F(u_1) \right|  +\left| \int \varphi_A' u_1f(u_1) \right| \lesssim \int \zeta_A^2 |u_1|^{p+1}\lesssim A^2 \Vert u_1 \Vert_{L^{\infty}}^{p-1} \int  |\partial_x w_1|^2.
	\end{equation*}
	Finally, we get
\begin{equation}
\begin{aligned}
&\left| \int  (\partial_x \vA u_1+\vA \partial_x u_1)  \left(f(Q+a_1\phi_0+u_1)-f(Q)-f'(Q)a_1\phi_0+N_0\nu_0^{-2}\LL\phi_0 \right) \right|\\
& \quad\lesssim
a_1^4 +  \int \sech\left(\dfrac{x}{2}\right) w_1^2
+A^2 \Vert u_1 \Vert_{L^{\infty}}^{p-1} \int  |\partial_x w_1|^2.
\end{aligned}
\end{equation}
This ends the proof of Lemma \ref{lem:2p4}.
\end{proof}

\subsection{End of Proposition \ref{prop:virial_I}}

Applying Lemmas \ref{lem:2p3} and \ref{lem:2p4}, and taking $\| u_1\|_{L^{\infty}}\leq \delta_A,$ for $\delta_A$ small enough, we have proved
\begin{eqnarray}
\begin{aligned}
\dfrac{d}{dt}\I
\leq&-\dfrac{1}{2} \int  \left[ w_2^2 +3(\partial_x w_1)^2
+\left(1-\left(\frac1A+\frac{1}{A^2}\right)\textbf{1}_{\{|x|\geq 1\}}\right)w_1^2 \right]   \\
&+C_1 a_1^4 +  C_1 \int \sech\left(\dfrac{x}{2}\right) w_1^2  
+A^2 \Vert u_1 \Vert_{L^{\infty}}^{p-1} \int  |\partial_x w_1|^2 \\
\leq&-\dfrac{1}{2} \int  \left[ w_2^2 +2(\partial_x w_1)^2
+\left(1-\frac{C_1}{A}\right)w_1^2 \right]   
+C_1a_1^4 +  C_1 \int \sech\left(\dfrac{x}{2}\right) w_1^2  .
\end{aligned}
\end{eqnarray}
This concludes the proof.

\section{Transformed problem and second virial estimates}\label{sec:3}

Following the idea of Martel \cite{Martel_linearKDV}, we will consider the function $v_1=\LL u_1$ instead $u_1$ to obtain a transformed problem with better virial properties. However, we must be careful since our original variables $(u_1,u_2)$ belong to $H^1(\R)\times L^2(\R)$, and by using $\LL$, the new variables are not well-defined. Therefore, we need a regularization procedure, as in \cite{KMM}.

\subsection{The transformed problem}
Let $\gamma >0$ small, to be determined later, set
\begin{equation}\label{eq:change_variable}
	\begin{cases}
	v_1= (1-\gamma \partial_x^2)^{-1}\LL  u_1,\\
	v_2= (1-\gamma \partial_x^2)^{-1}  u_2.
	\end{cases}
\end{equation}
From the system \eqref{eq:u_linear}, follows that  $(v_1,v_2)\in H^{1}(\R)\times H^{2}(\R)$, and  satisfies the system
\begin{equation}\label{eq:syst_v}
	\begin{cases}
	\dot{v}_1=\LL(\partial_x v_2)
	+G(x),\\
	\dot{v}_2= \partial_x v_1
	+H(x),
	\end{cases}
\end{equation}
where
\begin{equation}\label{eq:G_H}
\begin{aligned}
 H(x)=&~{}(1-\gamma\partial_x ^2)^{-1}  N^{\perp}, \\
 G(x)=&~{} \gamma (1-\gamma\partial_x ^2)^{-1}\left[  \partial_x^2(f'(Q))\partial_x v_2+2 \partial_x(f'(Q))\partial_x^2 v_2 \right].
 \end{aligned}
\end{equation}
Now we compute a second virial estimate, this time on $(v_1,v_2)$.

\subsection{Virial functional for the transformed problem}

Set now
\begin{equation}\label{eq:virial_J}
\J=\int   \psi_{A,B} v_1 v_2  ,
\end{equation}
with
\begin{equation}
\psi_{A,B}=\chi_A^2\varphi_B,\ \  \ z_i=\chi_A \zeta_B v_i,\ \  i= 1,2. \label{eq:def_psiB}
\end{equation}
Here, $(z_1,z_2)$ represents a localized version of the variables $(v_1,v_2)$ at the scale $B$. This scale is intermediate, and $\J$ involves a cut-off at scale $A$, which is needed to bound some bad  error and nonlinear terms; see \cite{kink} for a similar procedure. 

\begin{prop}\label{prop:virial_J}
There exist $C_2>0$ and $\delta_2>0$ such that for $\gamma=B^{-4}$ and for any $0<\delta\leq\delta_2$, the following holds. Fix $B=\delta^{-1/8}$. Assume that for all $t\geq 0$,  \eqref{eq:ineq_hip} holds. Then for all $t\geq 0$, 
	\begin{equation}\label{eq:dJ}
	\begin{aligned}
	\dfrac{d}{dt}\J\leq & -\frac12 \int  \left[z_1^2+(V_0(x)-C_2 B^{-1}) z_2^2 +2(\partial_x z_2)^2 \right] \\
&+C_2 B^{-1} \bigg(\| w_1\|_{L^2}^2
+\| w_2\|_{L^2}^2 \bigg)
+ C_2 |a_1|^3, 
	\end{aligned}
	\end{equation}
	where $V_0(x)$ is given by \eqref{eq:LL}.
\end{prop}

The rest of this section is devoted to the proof of this proposition, which has been divided in several subsections.  

\subsection{Proof of Proposition \ref{prop:virial_J}: first computations}

We have from \eqref{eq:virial_J} and \eqref{eq:syst_v},
\begin{equation}
\begin{aligned}
	\dfrac{d}{dt}\J=&\int  \psi_{A,B}\left[(\LL\partial_{x} v_2)v_2+\frac12\partial_{x}(v_1^2)+H(x)v_1+G(x)v_2 \right]\\
	=&\int  \psi_{A,B} (\LL\partial_{x} v_2)v_2  -\frac12 \int  \psi_{A,B}' v_1^2 +\int  \psi_{A,B} \left[G(x)v_2+H(x)v_1 \right] .
\end{aligned}
\end{equation}
In a similar way to the computation in \eqref{eq:uPLu}, we have
\[
\begin{aligned}
\int  \psi_{A,B} (\LL\partial_{x} v_2)v_2 
=&-\dfrac{1}{2}\int   \psi_{A,B}'  \left(v_2^2+3(\partial_x v_2)^2\right) +\dfrac{1}{2}\int  \psi_{A,B}''' v_2^2 -\dfrac{1}{2}\int \psi_{A,B} f'(Q)\partial_x (v_2^2) .
\end{aligned}
\]
We consider now the following decomposition
\begin{equation}\label{eq:J'_v}
\begin{aligned}
	\dfrac{d}{dt}\J
	=&-\dfrac{1}{2}\int   \psi_{A,B}'  \left(v_1^2+v_2^2+3(\partial_x v_2)^2\right) 
	+\dfrac{1}{2}\int  \psi_{A,B}''' v_2^2 \\
	 &-\dfrac{1}{2}\int \psi_{A,B} f'(Q)\partial_x (v_2^2) 
	+\int  \psi_{A,B} G(x)v_2 +\int  \psi_{A,B} H(x)v_1  \\
	=: & ~{} (J_1+J_2)+ (J_3+J_4+J_5).
\end{aligned}
\end{equation}
By definition of $\psi_{A,B}$ (see \eqref{eq:def_psiB}), it follows that
\begin{equation}\label{eq:psi_deriv}
\begin{aligned}
\psi_{A,B}' =&~{}\chi_A^2 \zeta_B^2+(\chi_A^2)'\varphi_B,\\
 \psi_{A,B}'' =&~{} \chi_A^2 (\zeta_B^2)'+2(\chi_A^2)' \zeta_B^2+(\chi_A^2)''\varphi_B,\\
\psi_{A,B}''' =& ~{}\chi_A^2 (\zeta_B^2)''+3(\chi_A^2)' (\zeta_B^2)'+3(\chi_A^2)'' \zeta_B^2+(\chi_A^2)'''\varphi_B.
\end{aligned}
\end{equation}
Also, by the definition of $z_i$ in \eqref{eq:def_psiB}, we have:
\begin{equation}
\begin{aligned}
-2J_1 =& \int   \psi_{A,B}'  \left(v_1^2+v_2^2+3(\partial_x v_2)^2\right)\\
=& \int   (\chi_A^2 \zeta_B^2+(\chi_A^2)'\varphi_B)\left(v_1^2+v_2^2+3(\partial_x v_2)^2\right) \\
=& \int  (z_1^2+z_2^2)  +\int   (\chi_A^2)'\varphi_B\left(v_1^2+v_2^2+3(\partial_x v_2)^2\right) +3\int   \chi_A^2 \zeta_B^2(\partial_x v_2)^2.
\end{aligned}
\end{equation}
Derivating  $z_2=\zeta_{B}\chi_{A} v_2$, replacing and integrating by parts, we obtain
\begin{equation}\label{eq:chiAzetaB_v_ix}
\begin{aligned}
\int \chi_A^2\zeta_B^2(\partial_x v_2)^2 =& \int  \dfrac{\zeta_B''}{\zeta_B}z_2^2+\int  (\partial_x z_2)^2
+ \int   \left[ \chi_A'' +2\chi_A' \frac{\zeta_B'}{\zeta_B}\right]\chi_A \zeta_B^2 v_2^2.
\end{aligned}
\end{equation}
Then, for $J_1$ we obtain
\begin{equation}
\begin{aligned}
-2J_1
=& \int  (z_1^2+z_2^2) 
+3\int  \dfrac{\zeta_B''}{\zeta_B}z_2^2+3\int  (\partial_x z_2)^2\\
&+ 3\int   \left[ \chi_A'' +2 \chi_A' \frac{\zeta_B'}{\zeta_B}\right]\chi_A \zeta_B^2 v_2^2
+\int   (\chi_A^2)'\varphi_B\left(v_1^2+v_2^2 +3(\partial_x v_2)^2\right).
\end{aligned}
\end{equation}
Now we turn into $J_2$. By \eqref{eq:psi_deriv}, $J_2$ satisfies the following decomposition
\[
\begin{aligned}
2J_2
=&~{} \int (\chi_A^2 (\zeta_B^2)''+3(\chi_A^2)' (\zeta_B^2)'+3(\chi_A^2)'' \zeta_B^2+(\chi_A^2)'''\varphi_B) v_2^2 \\
=& ~{}2\int \left[ \left(\dfrac{\zeta_B'}{\zeta_B}\right)^2+\dfrac{\zeta_B''}{\zeta_B} \right]z_2^2
+ \int (3(\chi_A^2)' (\zeta_B^2)'+3(\chi_A^2)'' \zeta_B^2+(\chi_A^2)'''\varphi_B) v_2^2 .
\end{aligned}
\]
For $J_3$, integrating by parts and using the definition of $z_2$, we obtain
\[
\begin{aligned}
-2J_3
=& -\int \partial_x(\psi_{A,B} f'(Q)) v_2^2  
= -\int \left[(\chi_A^2\zeta_B^2+(\chi_A^2)'\varphi_B)f'(Q)+\chi_A^2 \varphi_B \partial_x (f'(Q))\right] v_2^2  \\
=& -\int  \left[f'(Q)+\partial_x (f'(Q))\dfrac{\varphi_B}{\zeta_B^2}\right]z_2^2  -\int (\chi_A^2)'\varphi_Bf'(Q) v_2^2. 
\end{aligned}
\]
Finally,  we obtain that the main part of the virial can be write as
\[
J_1+J_2+J_3= -\frac12 \int  \left[z_1^2+V(x)z_2^2 +3(\partial_x z_2)^2 \right] 
+\tilde{J_1},
\]
where
\[
V(x)=1+\frac{\zeta_B''}{\zeta_B}-2\dfrac{(\zeta_B')^2}{\zeta_B^2}-f'(Q)-\partial_x(f'(Q))\dfrac{\varphi_B}{\zeta_B^2
},
\]
and the error term is given by
\begin{equation}\label{eq:tildeJ1}
\begin{aligned}
\tilde{J_1}=&
-\frac12 \int   \left[ 3\big( \chi''_A \chi_A-(\chi_A^2)''\big)  \zeta_B^2-\frac32(\chi_A^2)'( \zeta_B^2)'+\big((\chi_A^2)'-(\chi_A^2)'''\big)\varphi_B\right] v_2^2\\
&+\frac12\int (\chi_A^2)'\varphi_B f'(Q) v_2^2  
-\frac12\int   (\chi_A^2)'\varphi_B v_1^2   -\frac32\int   (\chi_A^2)'\varphi_B(\partial_x v_2)^2.
\end{aligned}
\end{equation}
To control the main part of the virial is necessary a lower bound for the potential $V(x)$. We have the following result:\begin{lem}\label{lem:bound_V}
	There are  $C>0$ and $B_0>0$ such that for all $B\geq B_0$, one has
	\begin{equation}\label{eq:bound_V}
	V(x)\geq V_0(x)- CB^{-1}, \ \ \mbox{where} \ \ V_0(x)=1-f'(Q).
	\end{equation}
\end{lem}
\begin{proof}
	First, recalling  \eqref{eq:zA''-zA2} and changing the scale, we have
	\begin{equation}
	\begin{aligned}
	\left\vert \frac{\zeta_B''}{\zeta_B}-2\left(\dfrac{\zeta_B'}{\zeta_B}\right)^2\right\vert
	\lesssim \frac1B. 
	\end{aligned}
	\end{equation}
	Using that for $x\in [0,\infty)\mapsto \zeta_B(x)$ is a non-increasing function, we have for $x>0,$
	\[
	\dfrac{\varphi_B}{\zeta_B^2}=\dfrac{\int_{0}^x \zeta_B^2 }{\zeta_B^2} >0,
	\]
	and  $\partial_x (f'(Q))<0$ for $x>0$. Then,
	\begin{equation}
		\begin{aligned}
		V(x)
		\geq &~{} 1- C B^{-1}-f'(Q)+|\partial_x (f'(Q))x|\geq 1- CB^{-1} -f'(Q)=V_0(x)-CB^{-1}.
		\end{aligned}
	\end{equation}
The case $x\leq 0$ is similar. These estimates hold for any $x\in\R$. This concludes the proof.
\end{proof}

\noindent
{\bf First conclusion}. Using this lemma, and the above definition of $\tilde{J}_1$, we conclude
\begin{equation}\label{eq:estimates_J}
\frac{d}{dt}\J\leq -\frac12 \int  \left[z_1^2+(V_0-CB^{-1}) z_2^2 +3(\partial_x z_2)^2 \right] 
+\tilde{J_1}+J_4+J_5,
\end{equation}
where $J_4$ and $J_5$ are related with the nonlinear term in \eqref{eq:J'_v}.
To control the terms $\tilde{J_1}, J_4, J_5$, and the terms that will appear in the sections below, some technical estimates will be needed.

\subsection{First technical estimates}

For $\gamma>0$, let $(1-\gamma \partial_x^2)^{-1}$
	be the bounded operator from $L^2$ to $H^2$ defined by its Fourier transform as
	\[
	\widehat{((1-\gamma \partial_x^2)^{-1} g)}(\xi)=\frac{\widehat{g}(\xi)}{1+\gamma \xi^2},\quad \mbox{for any }  g\in L^2. 
	\]
We start with a basic but essential result, in the spirit of \cite{kink}.

\begin{lem}\label{lem:estimates_IOp}
	Let $f\in L^2(\R)$ and $0<\gamma<1$, we have the following estimates
	\begin{enumerate}[$(i)$]
		\item $\| (1-\gamma \partial_x^2)^{-1}f\|_{L^2(\R)}\leq \|f \|_{L^2(\R)}$,
		\item $\| (1-\gamma \partial_x^2)^{-1}\partial_x f\|_{L^2(\R)}\leq \gamma^{-1/2}\|f \|_{L^2(\R)}$,
		\item $\| (1-\gamma \partial_x^2)^{-1}f\|_{H^2(\R)}\leq \gamma^{-1}\|f \|_{L^2(\R)}$.
	\end{enumerate}
\end{lem}
We also enunciate the following result that appears in \cite{kink,KMM}:
\begin{lem}\label{lem:cosh_kink}
	There exist $\gamma_1>0$ and $C>0$ such that for any $\gamma\in (0,\gamma_1)$, $0<K\leq 1$ and $g\in L^2$, the following estimates holds
	\begin{equation}\label{eq:lem_sec_Iop g}
	\left\| \sech\left( K x\right) (1-\gamma\partial_x^2)^{-1} g\right\|_{L^2}
	\leq C \left\| (1-\gamma\partial_x^2)^{-1}\left[ \sech\left( K x\right)  g\right]\right\|_{L^2},
	\end{equation}
	and 
	\[
	\left\| \cosh\left( Kx\right) (1-\gamma\partial_x^2)^{-1} \left[\sech\left( K x\right) g\right]\right\|_{L^2}
	\leq C \left\| (1-\gamma\partial_x^2)^{-1} g\right\|_{L^2}.
	\]
\end{lem}

From this lemma, we obtain the following result.
\begin{cor}\label{cor:estimates_Sech_Iop_partial}
	For any $0<K\leq 1$ and $\gamma >0$ small enough, for any $f\in L^2$,
	\begin{equation}
	\| \sech(Kx)(1-\gamma \partial_x^2)^{-1}\partial_x f\|_{L^2} \lesssim \gamma^{-1/2} \| \sech(Kx) f \|_{L^2},
	\end{equation}
	where the implicit constant is independent of $\gamma$ and $K$.
\end{cor}
\begin{proof} 
	Using  \eqref{eq:lem_sec_Iop g} and rewriting, we have
	\[
	\begin{aligned}
	\| \sech(Kx)(1-\gamma \partial_x^2)^{-1}\partial_x f\|_{L^2} 
	\lesssim &	~{} \| (1-\gamma \partial_x^2)^{-1} \left [\sech(Kx)\partial_x f \right] \|_{L^2} \\
		\lesssim &~{}	\| (1-\gamma \partial_x^2)^{-1} \partial_x \left [\sech(Kx) f \right] \|_{L^2}\\
		&+\| (1-\gamma \partial_x^2)^{-1} (\partial_x \sech(Kx)) f  \|_{L^2}. 
	\end{aligned}
	\]
	The proof concludes applying Lemma \ref{lem:estimates_IOp}.
\end{proof}
Following the spirit of Lemma \ref{lem:cosh_kink}, we obtain

\begin{lem}\label{lem:estimates_Sech_Iop_Op}
	For any $0<K\leq 1$ and $\gamma >0$ small enough, for any $f\in L^2$,
	\begin{equation}
	\| \sech(Kx)(1-\gamma \partial_x^2)^{-1}(1-\partial_x^2)f\|_{L^2} \lesssim \gamma^{-1} \| \sech(Kx) f \|_{L^2},
	\end{equation}
	where the implicit constant is independent of $\gamma$ and $K$.
\end{lem}
\begin{proof}
	Set $h=\sech\left( Kx\right) \IOpg \Op f$ and $k=\sech\left(Kx\right) f$. We have
	\begin{equation}
	\cosh \left( Kx\right) h =\IOpg \Op \left[ \cosh(Kx)k\right].
	\end{equation}
	Thus, we obtain
	\begin{equation*}
	\begin{aligned}
	\cosh \left( Kx\right) h 
	=&~{}\IOpg \left[ \cosh(Kx)k\right]-\partial_x^2\IOpg  \left[ \cosh(Kx)k\right]\\
	=&~{}\IOpg \left[ \cosh(Kx)k\right]+\gamma^{-1}( 1-\gamma\partial_x^2-1)\IOpg  \left[ \cosh(Kx)k\right]\\
	=&~{}\IOpg \left[ \cosh(Kx)k\right]+\gamma^{-1}\cosh(Kx)k-\gamma^{-1}\IOpg  \left[ \cosh(Kx)k\right].
	\end{aligned}
	\end{equation*}
Thus,
	\begin{equation*}
\begin{aligned}
\gamma h 
=&
(\gamma -1)\sech \left( Kx\right) \IOpg \left[ \cosh(Kx)k\right]+k,
\end{aligned}
\end{equation*}
using Lemma \ref{lem:cosh_kink} and dividing by $\gamma$, we obtain
	\begin{equation*}
		 \|h\|_{L^2} 
		\lesssim
		\gamma^{-1}\|k\|_{L^2}.
	\end{equation*}
	This concludes the proof.
\end{proof}

We need some additional auxiliary estimates to related the several variables defined.
\begin{lem}\label{lem:v_u}
One has:
	\begin{enumerate}[(a)]
		
		\item Estimates on $v_1$:
		\begin{equation}\label{eq:estimates_v1}
		\begin{gathered}
		\|v_1\|\lesssim \gamma^{-1} \|u_1 \|_{L^2},\\
		\|\partial_x v_1\|\lesssim \gamma^{-1/2} \| f'(Q)u_1 \|_{L^2}+\gamma^{-1} \|\partial_x u_1\|^2_{L^2} .
		\end{gathered}
				\end{equation} 
				
		\item Estimates on $v_2$:
			\begin{equation}\label{eq:estimates_v2}
		\begin{gathered}
		\|v_2\|_{L^2}\lesssim  \| u_2\|_{L^2}, \quad 
			\|\partial_x v_2\|_{L^2} \lesssim \gamma^{-1/2}\| u_2\|_{L^2},\\
				\|\partial_x^2 v_2\|_{L^2}\lesssim \gamma^{-1} \| u_2\|_{L^2}.
		\end{gathered}
	\end{equation}
	\end{enumerate}
\end{lem}
 The proof of the above results are a direct application of Lemma \ref{lem:estimates_IOp}.

%

\begin{lem}\label{lem:v_w} Let  $1\leq K\leq A$ fixed. Then
	
	\begin{enumerate}[(a)]
		
		\item Estimates on $v_1$:
		\begin{equation}\label{eq:estimates_v1_w1}
		\begin{gathered}
		\|\zeta_K v_1\|\lesssim \gamma^{-1} \|w_1 \|_{L^2},\\
		\|\zeta_K \partial_x v_1\|\lesssim \gamma^{-1} \|\partial_x w_1 \|_{L^2}+  \gamma^{-1} \| w_1 \|_{L^2}.
		\end{gathered}
		\end{equation} 
		
		\item Estimates on $v_2$:
		\begin{equation}\label{eq:estimates_v2_w2}
		\begin{gathered}
		\| \zeta_K v_2\|_{L^2}\lesssim  \| w_2\|_{L^2},\\
		\|\zeta_K \partial_x v_2\|_{L^2} \lesssim \gamma^{-1/2}\| w_2\|_{L^2},\\
		\|\zeta_K \partial_x^2 v_2\|_{L^2}\lesssim \gamma^{-1} \| w_2\|_{L^2}.
		\end{gathered}
		\end{equation}
	\end{enumerate}
\end{lem}

\begin{proof}
	Proof of  \eqref{eq:estimates_v1_w1}. $(i)$ Applying the definition of $v_1$ \eqref{eq:change_variable}, we have
	\[
	\begin{aligned}
	\| \zeta_K v_1\|_{L^2}
	\lesssim & ~{} \|\zeta_K  (1-\gamma\partial_x^2)^{-1} \left[(1-\partial_x^2)( u_1)-f'(Q) u_1\right]\|_{L^2}.
	\end{aligned}
	\]
	Using Lemma \ref{lem:estimates_Sech_Iop_Op} and Lemma \ref{lem:estimates_IOp}, \eqref{eq:bound_phiA} and $1\leq K\leq A$, we conclude
		\[
	\begin{aligned}
	\| \zeta_K v_1\|_{L^2}
	\lesssim &~{} \gamma^{-1}\|  \zeta_K u_1\|_{L^2}+\|\zeta_K f'(Q) u_1\|_{L^2}\\
	\lesssim &~{} \gamma^{-1}\|  \zeta_K\zeta_A^{-1 } w_1\|_{L^2}+\|\zeta_K \zeta_A^{-1 } f'(Q) w_1\|_{L^2}\lesssim  \gamma^{-1}\|  w_1\|_{L^2}.
	\end{aligned}
	\]
	 Proof of $(ii)$. First, by the definition of $w_1$ \eqref{eq:wi}, we get
	\begin{equation}\label{eq:Zk_u1x}
	\zeta_K \partial_x u_1=\zeta_K \zeta_A^{-1}\bigg(\partial_x w_1-\frac{\zeta_A'}{\zeta_A} w_1\bigg).
	\end{equation}
	 Then, by definition of $v_1$ in \eqref{eq:change_variable},
	\[
	\begin{aligned}
	\| \zeta_K \partial_x v_1\|_{L^2}
	\lesssim & ~{} \|\zeta_K  (1-\gamma\partial_x^2)^{-1} (1-\partial_x^2)\partial_x u_1\|_{L^2}+\|\zeta_K (1-\gamma \partial_x^2)^{-1}\partial_x [f'(Q) u_1]\|_{L^2}.
	\end{aligned}
	\]
	and using Lemma \ref{lem:estimates_Sech_Iop_Op}, 
	\[
	\begin{aligned}
	\| \zeta_K \partial_x v_1\|_{L^2}
	\lesssim &~{} \gamma^{-1}\|  \zeta_K \partial_x u_1\|_{L^2}	+\| \zeta_K f'(Q) u_1\|_{L^2}\\
		\lesssim & ~{}\gamma^{-1} (\|  \partial_x w_1\|_{L^2}+A^{-1}\|w_1\|_{L^2})	+\| w_1\|_{L^2}\\
				\lesssim & ~{}\gamma^{-1} (\|  \partial_x w_1\|_{L^2}+\|w_1\|_{L^2})	.
	\end{aligned}
	\]
	This ends the proof of \eqref{eq:estimates_v1_w1}.  Following the preceding steps for $v_2$, the proof concludes.
\end{proof}

Now we perform some technical estimates on the variable $z_i$.

\begin{cor}\label{cor:z_w} 
	One has:
	\begin{enumerate}[(a)]
		\item Estimates on $z_1$:
		\begin{equation}\label{eq:estimates_z1_w1}
		\begin{gathered}
		\| z_1\|\lesssim \gamma^{-1} \|w_1 \|_{L^2},\\
		\|\partial_x z_1\|\lesssim \gamma^{-1} \|\partial_x w_1 \|_{L^2}+  \gamma^{-1} \| w_1 \|_{L^2}.
		\end{gathered}
		\end{equation} 
		
		\item Estimates on $z_2$:
		\begin{equation}\label{eq:estimates_z2_w2}
		\begin{gathered}
		\|  z_2\|_{L^2}\lesssim  \| w_2\|_{L^2},\\
		\|\partial_x z_2\|_{L^2} \lesssim \gamma^{-1/2}\| w_2\|_{L^2},\\
		\| \partial_x^2 z_2\|_{L^2}\lesssim \gamma^{-1} \| w_2\|_{L^2}.
		\end{gathered}
		\end{equation}
	\end{enumerate}
\end{cor}

\begin{proof}
	Proof of  \eqref{eq:estimates_z1_w1}. For $(i)$, from definition of $z_1=\chi_A \zeta_B v_1$, we have
	\[
	\begin{aligned}
	\| z_1\|_{L^2}
	\lesssim\|\zeta_B v_1\|_{L^2},
	\end{aligned}
	\]
	and using Lemma \ref{lem:v_w}, we conclude
	\[
	\begin{aligned}
	\| z_1\|_{L^2} \lesssim  \gamma^{-1}\|  w_1\|_{L^2}.
	\end{aligned}
	\]
	For $(ii)$, derivating $z_1$, we obtain
	\[\begin{aligned}
	\partial_x z_1=&~{} \chi_A' \zeta_B v_1+\chi_A \zeta_B' v_1+\chi_A \zeta_B \partial_x v_1\\
	=&~{}\chi_A' \zeta_B v_1+\frac{\zeta_B'}{\zeta_B} z_1+\chi_A \zeta_B \partial_x v_1.
	\end{aligned}\]
	Then, by Lemma \ref{lem:v_w} we have
	\[
	\begin{aligned}
		\|\partial_x z_1\|
			\leq \gamma^{-1}\big(\left\|w_1\right\|_{L^2}+\|  \partial_x w_1\|_{L^2}\big).
	\end{aligned}
	\]
	For $z_2$ we use the same strategy, and we skip the details. This ends the proof.
\end{proof}

\begin{lem}\label{lem:sech_u_w} One has:
	\begin{enumerate}[(a)]
		\item Estimates on $u_1$:
		\begin{equation}\label{eq:estimates_sech_u1_w1}
		\begin{gathered}
		\left\|\sech^{1/2}(x) u_1\right\|_{L^2} \lesssim \|w_1 \|_{L^2},\\
		\left\|\sech^{1/2}(x) \partial_x u_1\right\|_{L^2}	\lesssim  \|\partial_x w_1 \|_{L^2}+ \| w_1 \|_{L^2}.
		\end{gathered}
		\end{equation} 
		
		\item Estimates on $u_2$:
		\begin{equation}\label{eq:sech_u2_w2}
		\begin{gathered}
		\left\| \sech^{1/2}(x) u_2\right\|_{L^2}\lesssim  \| w_2\|_{L^2}.
		\end{gathered}
		\end{equation}
	\end{enumerate}
\end{lem}
\begin{proof}
	Proof of  \eqref{eq:estimates_sech_u1_w1} . Recalling definition of $w_i= \zeta_A u_i$ for $i=1,2.$. We have
	\[
	\begin{aligned}
	\| \sech^{1/2}(x) u_i\|_{L^2}
	\lesssim\|\sech^{1/2}(x)\zeta_A^{-1} w_i\|_{L^2}\leq  \|w_i\|_{L^2}.
	\end{aligned}
	\]
	Furthermore, derivating $w_1$, we have
	\[
	\zeta_A \partial_x u_1=\partial_x w_1-\frac{ \zeta_A'}{\zeta_A} w_1.
	\]
	Then, 
	\[
	\begin{aligned}
	\|\sech^{1/2}(x) \partial_x u_1\|_{L^2}
	=
	\left\|\sech^{1/2}(x)\zeta_A^{-1} \left(\partial_x w_1-\frac{ \zeta_A'}{\zeta_A} w_1\right)\right\|
	\leq  \|\partial_x w_1\|_{L^2}+A^{-1}\|w_1\|_{L^2}.
	\end{aligned}
	\]
This concludes the proof.
\end{proof}

\subsection{Controlling  error and nonlinear terms}

 By the definition of $\zeta_B$ and $\chi_A$ in \eqref{eq:zetaB} and \eqref{eq:psi_chiA}, it holds
\begin{equation}
\begin{aligned}\label{eq:bound_ZCV_B}
\zeta_B(x)\le e^{-\frac{|x|}{B}}, &\quad |\zeta_B'(x)|\lesssim \frac{1}{B} e^{-\frac{|x|}{B}}, \ \
|\varphi_B| \lesssim B,\\  |(\chi_{A}^2)'|\lesssim A^{-1},\ \ &|(\chi_{A}^2)''|\lesssim A^{-2},\ \ \ |(\chi_{A}^2)'''|\lesssim A^{-3}.
\end{aligned}
\end{equation}

\subsubsection{Control of $\tilde{J_1}$.}\label{control_tJ1}
Considering the following decomposition $\tilde{J_1}$:
\[
\begin{aligned}
\tilde{J_1}=&
-\frac12\int   (\chi_A^2)'\varphi_B [v_1^2+3(\partial_x v_2)^2 ]   
+\frac12\int  \big[(\chi_A^2)'f'(Q)  + \big((\chi_A^2)''' -(\chi_A^2)'\big)\big]  \varphi_Bv_2^2  \\
&-\frac32 \int   \left[ \big( \chi''_A \chi_A-(\chi_A^2)''\big) -(\chi_A^2)' \frac{\zeta_B'}{\zeta_B}\right] \zeta_B^2 v_2^2
=:H_1+H_2+H_3.
\end{aligned}
\]
For $H_1$ and $H_2$, using $|(\chi_A^2)'\varphi_B|\lesssim A^{-1}B$ and Remark \ref{rem:chiA_zetaA4}, we obtain
\begin{equation*}
\begin{aligned}
|H_1|
\lesssim   A^{-1}B (\|v_1 \|_{L^2(|x|\leq 2A)}^2+ \|\partial_x v_2\|_{L^2(|x|\leq 2A)}^2)
\lesssim A^{-1}B (\|\zeta_A^2 v_1 \|_{L^2}^2+ \|\zeta_A^2\partial_x v_2\|_{L^2}^2),
\end{aligned}
\end{equation*}
and
\[
\begin{aligned}
|H_2|\lesssim & ~{}   A^{-1}B \| v_2\|_{L^2(A \leq|x|\leq 2A)}^2
\lesssim  A^{-1}B \| v_2\|_{L^2(|x|\leq 2A)}^2 \lesssim A^{-1}B\|\zeta_A^2 v_2\|_{L^2}^2.
\end{aligned}
\]
For $H_3$, using
\eqref{eq:bound_ZCV_B}, we have
\[
\begin{aligned}
|H_3|
\leq &~{}
\frac32 \int   \left| \big( \chi''_A \chi_A-(\chi_A^2)''\big) -(\chi_A^2)' \frac{\zeta_B'}{\zeta_B}\right| \zeta_B^2 v_2^2\\
\lesssim &~{} 
(AB)^{-1}\|\zeta_B v_2\|_{L^2(|x|\leq 2A)}^2
\lesssim  (AB)^{-1} \|\zeta_B v_2\|_{L^2}^2.
\end{aligned}
\]
Finally,  we get
\begin{equation}\label{eq:estimates_tildeJ1}
\begin{aligned}
|\tilde{J_1}|
\lesssim & ~{} 
%
A^{-1}B\left( \|\zeta_A v_1\|_{L^2}^2
+  \|\zeta_A v_2 \|_{L^2}^2 
+  \| \zeta_A \partial_x v_2\|_{L^2}^2 \right)
,
\end{aligned}
\end{equation}
 since  $\zeta_B\lesssim \zeta_A$.
\subsubsection{Control of $J_4$.}\label{control_J4}
Recall that
\[
\begin{aligned}
J_4=
\gamma  \int  \psi_{A,B} v_2(1-\gamma\partial_x ^2)^{-1}\left[  2\partial_x (\partial_x(f'(Q))\partial_x v_2)- \partial_x^2(f'(Q))\partial_x v_2 \right].
\end{aligned}
\]
Using H\"older's inequality
\[
\begin{aligned}
|J_4|
\lesssim & ~{}  \gamma \|\psi_{A,B} v_2\|_{L^2}\underbrace{
	\|(1-\gamma\partial_x ^2)^{-1}[  2\partial_x (\partial_x(f'(Q)) \partial_x v_2)- \partial_x^2(f'(Q))\partial_x v_2]\|_{L^2}}_{J_4^1}.
\end{aligned}
\]
First we focus on $J_4^1$. Using \eqref{lem:estimates_IOp},
\[
J_4^1\lesssim \gamma^{-1/2}\|(\partial_x(f'(Q))\partial_x v_2\|_{L^2}+
\|\partial_x^2(f'(Q))\partial_x v_2 \|_{L^2}.
\]
Recall that $|\partial_x (f'(Q))|, |\partial_x^2 (f'(Q))| \sim Q^{p-2}|Q'| \lesssim e^{-(p-1)|x|}$.
Therefore, we are led to the estimate of
\[
\| e^{-(p-1)|x|} \partial_x v_2\|_{L^2}.
\]
Differentiating $z_2=\chi_A \zeta_B v_2$, we obtain
\[
\chi_A \zeta_B \partial_x v_2
=
\partial_x z_2
-\frac{\zeta_B'}{\zeta_B} z_2-\chi_A'\zeta_B v_2,
\]
we get
\[
\begin{aligned}
&~{} e^{-2(p-1)|x|}(\partial_x v_2)^2\\
&~{} = e^{-2(p-1)|x|} \chi_A^2(\partial_x v_2)^2+ e^{-2(p-1)|x|} (1-\chi_A^2)(\partial_x v_2)^2\\
&~{} \lesssim  e^{-(p-1)|x|}\left[(\partial_x z_2)^2
+\left(\frac{\zeta_B'}{\zeta_B}\right)^2 z_2^2+(\chi_A'\zeta_B v_2)^2\right]  
+ e^{-(p-1)A} e^{-(p-1)|x|}(\partial_x v_2)^2\\
&~{} \lesssim  e^{-(p-1)|x|}\left[(\partial_x z_2)^2
+ \frac1{B^2}z_2^2\right]+e^{-(p-1)|x|}(\chi_A'\zeta_B v_2)^2  
+ e^{-(p-1)A} e^{-(p-1)|x|} (\partial_x v_2)^2.
\end{aligned}
\]
Hence, 
\[
\begin{aligned}
\| e^{-|x|}\partial_x v_2\|_{L^2}
\lesssim & ~{}
\|\partial_x z_2\|_{L^2}
+ \|z_2\|_{L^2}
+e^{-(p-1)A} (  \| \zeta_B v_2\|_{L^2} + \|\zeta_B \partial_x v_2\|_{L^2}).
\end{aligned}
\]
By the above inequality, we have
\begin{equation}\label{eq:bound_f'x_v2x}
\begin{aligned}
J_4^1\lesssim & ~{} 
\gamma^{-1/2} \bigg(\| \partial_x z_2\|_{L^2}
+ \|z_2\|_{L^2}
   +e^{-(p-1)A}  (  \| \zeta_B v_2\|_{L^2} + \|\zeta_B \partial_x v_2\|_{L^2}) \bigg).
\end{aligned}
\end{equation}
Second, using $\psi_{A,B}=\chi_A^2\varphi_B$ in \eqref{eq:def_psiB}, and Remark \ref{rem:chiA_zetaA4}, one can see that
\[
\|\psi_{A,B}v_2\|_{L^2}
 \lesssim B \|\chi_A v_2\|_{L^2}
 \lesssim  B \| v_2\|_{L^2(|x|<2A)}
  \lesssim  B \| \zeta_A^2 v_2\|_{L^2}.
\]
We conclude
\begin{equation}\label{eq:estimates_J4}
\begin{aligned}
|J_4|
\lesssim & ~{} 
\gamma^{1/2} B \|\zeta_A^2 v_2\|_{L^2}
\bigg(
\| \partial_x z_2\|_{L^2}
+ \|z_2\|_{L^2}
+e^{-(p-1)A}   (  \|\zeta_{B} v_2\|_{L^2} + \|\zeta_B\partial_x v_2\|_{L^2})
\bigg).
\end{aligned}
\end{equation}

\subsubsection{Control of $J_5$.}
Recalling that $\psi_{A,B}=\chi_A^2\varphi_B$, using the H\"older inequality and Remark \ref{rem:chiA_zetaA4}, we get
\[
\begin{aligned}
|J_5|=& \left|\int  \psi_{A,B} H(x)v_1 \right|\lesssim \| \chi_A\varphi_B v_1\|_{L^2}
\|\chi_A (1-\gamma\partial_x ^2)^{-1} N^{\perp} \|_{L^2}\\
\lesssim & ~{} 
\| \chi_A\varphi_B v_1\|_{L^2}
\|\zeta_A^2 (1-\gamma\partial_x ^2)^{-1} N^{\perp} \|_{L^2}.
\end{aligned}
\]
By the definition of $N^{\perp}$ (see \eqref{eq:Nperp}), it follows that
\[
\begin{aligned}
\|\zeta_A^2 (1-\gamma\partial_x ^2)^{-1}( N^{\perp} )\|_{L^2}
\leq & ~{} \| \zeta_A^2 (1-\gamma\partial_x ^2)^{-1} N\|_{L^2}+|N_0|\| \zeta_A^2 (1-\gamma\partial_x ^2)^{-1} \partial_x^{-1} \phi_0 \|_{L^2}\\
\lesssim & ~{}  \|\zeta_A^2  (1-\gamma\partial_x ^2)^{-1} N\|_{L^2}+|N_0|,
\end{aligned}
\]
since $\partial_x^{-1} \phi_0 \in L^2$ and $0\leq \zeta_A \lesssim 1$.\\
Furthermore, by definition of $N$ in \eqref{eq:Nperp}, and using Corollary \ref{cor:estimates_Sech_Iop_partial}, \eqref{eq:taylorFull} and Lemma \ref{lem:estimates_IOp}, we have
\begin{equation}\label{eq:LambaN}
\begin{aligned}
&~{} \|\zeta_A^2 (1-\gamma\partial_x ^2)^{-1} N\|_{L^2}\\
&~{} \leq \gamma^{-1/2}\| \zeta_A^2 [f(Q)+f'(Q)(a_1\phi_0+u_1)-f(Q+a_1\phi_0+u_1)]\|_{L^2}\\
&~{} \leq \gamma^{-1/2}\| \zeta_A^2 [ a_1^2 f''(Q) \phi_0^2+f''(Q) u_1^2+|a_1|^{p} \phi_0^p+|u_1|^{p}] \|_{L^2}\\
&~{}\leq \gamma^{-1/2}\left( a_1^2 \|  f''(Q) \zeta_A^2\phi_0^2  \|_{L^2} + \| f''(Q) \zeta_A w_1^2 \|_{L^2} +|a_1|^{p} \| \zeta_A^2 \phi_0^p \|_{L^2}+ \|\zeta_A |u_1|^{p-1}w_1 \|_{L^2} \right)\\
 & ~{} \lesssim \gamma^{-1/2} \left(a_1^2+\| u_1\|_{L^\infty}\| f''(Q) w_1\|_{L^2}+|a_1|^{p}+  \|u_1\|_{L^\infty}^{p-1}\| w_1 \|_{L^2}\right)\\
 & ~{} \lesssim \gamma^{-1/2} \left(a_1^{2} +\| u_1\|_{L^\infty}\| w_1\|_{L^2}\right).
\end{aligned}
\end{equation}
Note that we have used that $p\geq 2$.  Since $|\chi_A\varphi_{B}|\lesssim B$, we have
\begin{equation}\label{eq:estimates_psiB_v1}
\| \chi_A\varphi_{B} v_1\|_{L^2}\lesssim B \|\chi_A v_1\|_{L^2} \lesssim B \|v_1\|_{L^2(|x|<2A)}\lesssim B \|\zeta_A^2 v_1\|_{L^2}.
\end{equation}
Finally, by \eqref{eq:N_0_bounded},  \eqref{eq:LambaN} and \eqref{eq:estimates_psiB_v1} , we conclude
\begin{equation}\label{eq:estimates_J5}
	|J_5|
		\lesssim
	B\gamma^{-1/2} \| \zeta_A^2 v_1\|_{L^2}
	\big(
	a_1^2+\|u_1\|_{L^\infty}\|w_1\|_{L^2}
	\big).
\end{equation}

%

\subsection{End of proof of Proposition \ref{prop:virial_J}}

From
 \eqref{eq:estimates_tildeJ1}, \eqref{eq:estimates_J4},  \eqref{eq:estimates_J5},
 and choosing  
 \begin{equation}\label{primera}
 0<\gamma = B^{-4},
 \end{equation}
 it follows
 \begin{equation}\label{eq:J1+J4+J5}
 \begin{aligned}
|\tilde{J_1}+J_4+J_5|
\lesssim & ~{}  A^{-1}B\left( \|\zeta_A v_1\|_{L^2}^2
 +  \|\zeta_A v_2 \|_{L^2}^2 
 +  \| \zeta_A \partial_x v_2\|_{L^2}^2 \right)\\
 &+\gamma^{1/2} B \|\zeta_A v_2\|_{L^2}
 \bigg(
 \| \partial_x z_2\|_{L^2}
 + \|z_2\|_{L^2}
 +e^{-(p-1)A}   (  \|\zeta_{B} v_2\|_{L^2} + \|\zeta_B\partial_x v_2\|_{L^2})
 \bigg)\\
& +B\gamma^{-1/2} \| \zeta_A v_1\|_{L^2}
 \bigg(
 a_1^2+\|u_1\|_{L^\infty}\|w_1\|_{L^2}
 \bigg).
  \end{aligned}
 \end{equation}
Applying Lemma \ref{lem:v_w}-\eqref{eq:estimates_v1_w1} and \eqref{eq:scales}, we obtain
  \begin{equation}\label{eq:J1+J4+J5_w}
 \begin{aligned}
 |\tilde{J_1}+J_4+J_5| \lesssim & ~{} 
B^{-1} \bigg(\| w_1\|_{L^2}^2
+\| w_2\|_{L^2}^2
+ \| z_2\|_{L^2}^2 + \| \partial_x z_2\|_{L^2}^2 \bigg)\\
&+ B^8 \bigg(a_1^{4} +\|u_1\|_{L^\infty}^2\|w_1\|_{L^2}^2\bigg).
 \end{aligned}
 \end{equation}
Choosing 
\begin{equation}\label{segunda}
B \leq \delta^{-1/8},
\end{equation}
(to be fixed later) and using \eqref{eq:ineq_hip}, we arrive to
\[
\begin{aligned}
	B^{8}(a_1^4+\|u_1\|_{L^\infty}^2\|w_1\|_{L^2}^2)
	\lesssim & ~{}  \delta^{-1}( a_1^{4}+\|u_1\|^2_{L^\infty}\|w_1\|_{L^2}^2)
	\lesssim |a_1|^3+\delta \|w_1\|_{L^2}^2.
\end{aligned}
\]
Then, using the above estimates, we obtain that the  error term and the associated to the nonlinear part are bounded as follows:
\begin{equation*}
	\begin{aligned}
	|	\tilde{J}_1+J_4+J_5|
		\lesssim & ~{} 
	B^{-1} \left(\| w_1\|_{L^2}^2
	+\| w_2\|_{L^2}^2
	+ \| z_2\|_{L^2}^2 + \| \partial_x z_2\|_{L^2}^2 \right)
	+ |a_1|^3.
	\end{aligned}
\end{equation*}
Finally, the virial estimate is concluded as follows: for some $C_2>0$ independent of $B$ large, 
\begin{equation}\label{eq:dJ_final}
\begin{aligned}
\frac{d}{dt}\J\leq & ~{}  -\frac12 \int  \left[z_1^2+(V_0(x)-C B^{-1}) z_2^2 +3(\partial_x z_2)^2 \right] \\
&+C B^{-1} \bigg(\| w_1\|_{L^2}^2
+\| w_2\|_{L^2}^2
+ \| z_2\|_{L^2}^2 + \| \partial_x z_2\|_{L^2}^2 \bigg)
+C  |a_1|^3.\\
\leq & ~{} -\frac12 \int  \left[z_1^2+(V_0(x)-C_2 B^{-1}) z_2^2 +2(\partial_x z_2)^2 \right] \\
&+C_2 B^{-1} \bigg(\| w_1\|_{L^2}^2
+\| w_2\|_{L^2}^2 \bigg)
+ C_2 |a_1|^3.
\end{aligned}
\end{equation}
This ends the proof of Proposition \ref{prop:virial_J}.

\section{Gain of derivatives via transfer estimates}\label{sec:4}

We must note that in \eqref{eq:dI_w} the last term is a localized one, which in the language of estimate \eqref{eq:dJ_final} will correspond to a term of type $\partial_x z_1$, not appearing in this last estimate. However, this new term will be well-defined by the regularity of the original variables  $(u_1,u_2)$. We think that this problem appears as a product of the lack of balance in the regularity of $(v_1,v_2)$ (see Subsection \ref{eq:change_variable}). Therefore, we need new estimates to control $\partial_x z_1$.  
\\

To solve this new problem, we will focus on a new virial obtained for a new system of equations involving the variables $\tilde{v}_i=\partial_x v_i$, for $i=1,2$. Formally taking derivatives in \eqref{eq:syst_v}, we have
\begin{equation}\label{eq:syst_vx}
\begin{cases}
\dot{\tv}_1=\LL(\partial_x \tv_2)
-\partial_x(f'(Q)) v_2+\tilde{G}(x), \qquad \tilde{G}(x)=\partial_x G(x),\\
\dot{\tv}_2= \partial_x \tv_1
+\tilde{H}(x),\qquad \tilde{H}(x)=\partial_x H(x),
\end{cases}
\end{equation}
where $G$ and $H$ are given in \eqref{eq:G_H}.\\
For this new system, we consider the virial
\begin{equation}\label{eq:new_virials}
\mathcal{M}(t)=\int \phi_{A,B} \tv_1 \tv_2= \int \phi_{A,B} \partial_x v_1 \partial_x  v_2.
\end{equation}
Later we will choose $\phi_{A,B}=\psi_{A,B}=\chi_A^2 \varphi_B $ (see \eqref{eq:def_psiB}). 

\subsection{A virial estimate related to $\M$}
\begin{lem}\label{lem:identity_dtM}
	Let $(v_1,v_2)\in H^1(\R)\times H^2(\R)$ a solution of \eqref{eq:syst_v}. Consider $\phi_{A,B}$ an odd smooth bounded function to be a choose later. Then
	\begin{equation}
	\begin{aligned}\label{eq:M'}
	\frac{d}{dt}\mathcal{M}(t)
	=&-\dfrac{1}{2}\int  \phi_{A,B}'  \left(( \partial_x v_1)^2+(\partial_x v_2)^2+3(\partial_x^2 v_2)^2\right) 
	+\dfrac{1}{2}\int  \phi_{A,B}''' (\partial_x v_2)^2 \\
	&-\dfrac{1}{2}\int \phi_{A,B} f'(Q)\partial_x ((\partial_x v_2)^2) 
	+\int  \phi_{A,B} \tilde{G}(x) \partial_x v_2+\int  \phi_{A,B} \tilde{H}(x) \partial_x v_1   .
	\end{aligned}
	\end{equation}
\end{lem}

The identity \eqref{eq:M'} is interesting because it has exactly the same structure that $\dfrac{d}{dt}\J$ in \eqref{eq:J'_v}. This holds despite the new derivative terms appearing in \eqref{eq:syst_vx}. To obtain this we will benefit from a cancellation given by the parity of the data.

\begin{proof}[Proof of Lemma \ref{lem:identity_dtM}]
	From \eqref{eq:syst_v}, \eqref{eq:syst_vx} and \eqref{eq:J'_v}, we have
	\begin{equation}\label{eq:M'_vx}
	\begin{aligned}
	\frac{d}{dt}\mathcal{M}(t)
	=&-\dfrac{1}{2}\int  \phi_{A,B}'  \left(\tv_1^2+\tv_2^2+3(\partial_x \tv_2)^2\right) 
	+\dfrac{1}{2}\int  \phi_{A,B}''' \tv_2^2 \\
	&-\dfrac{1}{2}\int \phi_{A,B} f'(Q)\partial_x (\tv_2^2) 
	+\int  \phi_{A,B}\big[ -\partial_x(f'(Q)) v_2+\tilde{G}(x)\big] \tv_2
	+\int  \phi_{A,B} \tilde{H}(x) \tv_1  .
	\end{aligned}
	\end{equation}
	Rewriting the above identity in term of the variables $(v_1,v_2)$, we have
	\begin{equation}\label{eq:M'_v}
	\begin{aligned}
	\frac{d}{dt}\mathcal{M}(t)
	=&-\dfrac{1}{2}\int  \phi_{A,B}'  \left(( \partial_x v_1)^2+(\partial_x v_2)^2+3(\partial_x^2 v_2)^2\right) 
	+\dfrac{1}{2}\int  \phi_{A,B}''' (\partial_x v_2)^2 \\
	&-\dfrac{1}{2}\int \phi_{A,B} f'(Q)\partial_x ((\partial_x v_2)^2) 
	- \int  \phi_{A,B} \partial_x (f'(Q))v_2 \partial_x v_2\\
	&+\int  \phi_{A,B} \tilde{G}(x) \partial_x v_2+\int  \phi_{A,B} \tilde{H}(x) \partial_x v_1  .
	\end{aligned}
	\end{equation}
	Noticing that $v_2 \partial_x v_2$, $\partial_x (f'(Q))$ and $\phi_{A,B}$ are odd functions (see \eqref{eq:change_variable} and \eqref{eq:u_linear}), we have 
	\[
	 \int  \phi_{A,B} \partial_x (f'(Q))v_2 \partial_x v_2=0.
	\]
	This ends the proof of Lemma \ref{lem:identity_dtM}.
\end{proof}

The following proposition connects two virial identities in the variable $(z_1,z_2)$. Recall that from \eqref{primera} and \eqref{segunda}, $\gamma=B^{-4}$, $B\leq \delta^{-1/8}$.

\begin{prop}\label{prop:ineq_dtM}
	There exist $C_3>0$ and $\delta_3>0$ such that  for any $0<\delta\leq\delta_3$, the following holds. Fix $B=\delta^{-1/19}\leq \delta^{-1/8}$. Assume that for all $t\geq 0$,  \eqref{eq:ineq_hip} holds. Then for all $t\geq 0$, 
	
	\begin{equation}\label{eq:dM_z}
	\begin{aligned}
	 \frac{d}{dt}\M 
	 \leq & ~{}  
	-\dfrac{1}{2}\int  \left( (\partial_x z_1)^2+\left(V_0(x) -  C_3B^{-1}\right) (\partial_x z_2)^2+2(\partial_x^2 z_2)^2\right)	\\
	&+ C_3\| z_2\|_{L^2}^2
	+ C_3 B^{-1}\| z_1\|_{L^2}^2\\
	&+C_3 B^{-1} \left( \| \partial_x w_1\|_{L^2}^2+\| w_1\|_{L^2}^2
	+ \| w_2\|_{L^2}^2 \right)+
	C_3|a_1|^3.
	\end{aligned}
	\end{equation}
	\end{prop}
	
The proof of the above result requires some technical estimates. We first state them, to then prove Proposition \ref{prop:ineq_dtM} (Subsection \ref{ahorasi}). 	

\subsection{Second set of technical estimates} Now, we recall the following technical estimates on the variables $\zeta_B$ and other related error terms. These estimates are similar to the ones obtained in \eqref{eq:zA''-zA2}, therefore we only prove the new ones.
\begin{lem}
	Let $\zeta_B$ and $\chi$ be defined by \eqref{eq:zetaB} and \eqref{chichi}, respectively. Then
	\begin{equation}\label{eq:Z'B_Z''B}
	\begin{aligned}
	\frac{\zeta_B'}{\zeta_B}= -\frac1B[-\chi'(x)|x|+(1-\chi(x))\sgn(x)],\quad \frac{\zeta_B'' }{\zeta_B}= \left(\frac{\zeta_B'}{\zeta_B}\right)^2+\frac1B[\chi''(x)|x|+2\chi'(x)\sgn(x)],
	\end{aligned}
	\end{equation}
	and
	\begin{equation}\label{eq:ZB_3_4}
		\begin{aligned}
					\frac{\zeta_B'''}{\zeta_B}=&~{} 3\frac{\zeta_B''}{\zeta_B} \frac{\zeta_B'}{\zeta_B}
					-2 \left(\frac{\zeta_B'}{\zeta_B}\right)^3 
					+B^{-1}\left[ \chi'''(x)|x|+3\chi''(x)\sgn(x)\right] ,\\
				\frac{\zeta_B^{(4)}}{\zeta_B}=&~{}
				4\frac{\zeta_B'''}{\zeta_B} \frac{\zeta_B'}{\zeta_B}
				+3\left(\frac{\zeta_B''}{\zeta_B}\right)^2 
				-12 \frac{\zeta_B''}{\zeta_B} \left(\frac{\zeta_B'}{\zeta_B}\right)^2 
				+6\left(\frac{\zeta_B'}{\zeta_B}\right)^4 +\frac1B\left[ \chi^{(4)}(x)|x|+4\chi'''(x)\sgn(x)\right]
				 .
		\end{aligned}
	\end{equation}
\end{lem}
\begin{proof} Direct.
\end{proof}
	\begin{rem}\label{acotados}
		From the previous lemma we observe that
		\begin{equation*}
		\left| \frac{\zeta_B'}{\zeta_B}\right|\lesssim B^{-1} \textbf{1}_{\{|x|>1\}}(x),
		\end{equation*}
		\begin{equation*}
				\left| \frac{\zeta_B''}{\zeta_B}\right|
		\lesssim   B^{-2} \textbf{1}_{\{|x|>1\}}(x)+B^{-1} \textbf{1}_{\{1<|x|<2\}}(x)
				\lesssim   B^{-2}+B^{-1} \sech(x),
						\end{equation*}
		\begin{equation*}
		\begin{aligned}
			\left| \frac{\zeta_B'''}{\zeta_B}\right|
			\lesssim & ~{}  B^{-3}+B^{-1} \sech(x),\quad 
			\left| \frac{\zeta_B^{(4)}}{\zeta_B}\right|
			\lesssim   B^{-4}+B^{-1} \sech(x).
		\end{aligned}
	\end{equation*}
	In particular, for $A$ large enough, the following estimates hold:	 
	\begin{equation}\label{eq:C*z'/z}
	\begin{aligned}
			\left| \frac{\zeta_B'}{\zeta_B}\right|\lesssim  ~{}  B^{-1},	\quad 
			\left| \textbf{1}_{\{A<|x|<2A\}}\frac{\zeta_B''}{\zeta_B}\right|\lesssim   B^{-2},	\\
			\left| \textbf{1}_{\{A<|x|<2A\}} \frac{\zeta_B'''}{\zeta_B}\right|\lesssim    B^{-3},\quad
			\left| \textbf{1}_{\{A<|x|<2A\}} \frac{\zeta_B^{(4)}}{\zeta_B}\right|\lesssim   B^{-4}.
	\end{aligned}
	\end{equation}
	Finally,
		\begin{equation}\label{eq:z'/z}
	\begin{aligned}
	\left| \frac{\zeta_B''}{\zeta_B}\right| +
	\left| \frac{\zeta_B'''}{\zeta_B}\right|+
	\left|  \frac{\zeta_B^{(4)}}{\zeta_B}\right|\lesssim & ~{}  B^{-1}.
	\end{aligned}
	\end{equation}
	\end{rem}
	
These estimates will be useful in Claim \ref{claim:R_CZ_vi_xx}. Now we prove a formula for changing variables.

\begin{claim} \label{claim:P_CZ_vi_x}
	Let $P\in W^{1,\infty}(\R)$, $v_i$ be as in \eqref{eq:change_variable}, and $z_i$ be as in \eqref{eq:def_psiB}. Then 
		\begin{equation}\label{eq:claimP_CZ_B_vix_final}
	\begin{aligned}
	\int P(x)\chi_A^2\zeta_B^2(\partial_x v_i)^2 
	=&	\int P(x) (\partial_x z_i)^2 +
	\int    \left[ P'(x) \frac{\zeta_B'}{\zeta_B}+P(x) 
	\frac{\zeta_B''}{\zeta_B}\right] z_i^2\\
	&+\int  \mathcal{E}_1(P(x),x)\zeta_B^2 v_i^2,
	\end{aligned}
	\end{equation}
	where
	\begin{equation}\label{eq:claimE1}
	\begin{aligned}
	\mathcal{E}_1(P(x),x)
	= P(x)   \left[ \chi_A''\chi_A+(\chi_A^2)' \frac{\zeta_B'}{\zeta_B}\right] 
+ \frac12 P'(x)(\chi_A^2)' ,
	\end{aligned}
	\end{equation}
	and
	\begin{equation}\label{cota_final}
	|\mathcal{E}_1(P(x),x) |\lesssim 
	A^{-1}\|P' \|_{L^\infty} 
	+ (AB)^{-1} \|P\|_{L^\infty}.
	\end{equation}
\end{claim}

For the proof of these results, see Appendix \ref{ap:vi_x}.

\begin{rem}\label{rem:CZ_v2x}
	For $P=1$, we get
	\begin{equation}\label{eq:1CZ_B_vix}
	\begin{aligned}
	\int  \chi_A^2\zeta_B^2(\partial_x v_i)^2 
	=&	\int  (\partial_x z_i)^2 +
	\int     
	\frac{\zeta_B''}{\zeta_B} z_i^2
	+\int  \mathcal{E}_1(1,x)\zeta_B^2 v_i^2,
	\end{aligned}
	\end{equation}
	where
	\begin{equation}\label{eq:E1_1}
	\begin{aligned}
	\mathcal{E}_1(1,x)
	=\frac12  \chi_A''\chi_A+(\chi_A^2)' \frac{\zeta_B'}{\zeta_B}.
	\end{aligned}
	 \end{equation}
	Finally, one has the following estimate:
	\[
	\begin{aligned}
	\| \chi_A\zeta_B \partial_x v_i\| ^2
	\lesssim & ~{}  \| \partial_x z_i\|_{L^2}^2+B^{-1}\| z_i\|_{L^2}^2
	+ (AB)^{-1}\| w_i\|_{L^2}^2.
	\end{aligned}
	\]
\end{rem}

We need a second claim on the second derivative of $v_i$.

\begin{claim}\label{claim:R_CZ_vi_xx}
	Let  $R$ be a $W^{2,\infty}(\R)$ function, $v_i$ be as in \eqref{eq:change_variable}, and $z_i$ be as in \eqref{eq:def_psiB}. Then
	\begin{equation*}
	\begin{aligned}
	\int R(x) \chi_A^2\zeta_B^2 (\partial_x^2 v_i)^2
	=&
	\int R(x)(\partial_x^2 z_i)^2+\int \tilde{R}(x)z_i^2
	+\int P_R(x) (\partial_x z_i)^2\\ 
	&+
	\int    \left[ P_R'(x) \frac{\zeta_B'}{\zeta_B}+P_R(x) 
	\frac{\zeta_B''}{\zeta_B}\right] z_i^2 +\int \mathcal{E}_2(R(x),x) \zeta_B^2v_i^2\\
	&+\int  \mathcal{E}_1(P_R(x),x)\zeta_B^2 v_i^2
	+\int \mathcal{E}_3(R(x),x) \zeta_B^2(\partial_x v_i)^2,
	\end{aligned}
	\end{equation*}
	where
	\begin{equation}\label{eq:claimtildeR}	
	\begin{aligned}
	\tilde{R}(x)=&	  
-2R(x)\left[\frac{\zeta_B^{(4)}}{\zeta_B}+\frac{\zeta_B'''}{\zeta_B}\frac{\zeta_B'}{\zeta_B}\right]-2	R'(x)\frac{\zeta_B'''}{\zeta_B}
- R''(x)\frac{\zeta_B''}{\zeta_B},
	\end{aligned}
	\end{equation}
	\begin{equation}\label{eq:claimPR}
	P_R(x)=
	R(x) 
	\bigg[
	4  \frac{\zeta_B''}{ \zeta_B}  
	-2 \bigg(\frac{\zeta_B'}{\zeta_B} \bigg)^2\bigg]
	+2R'(x) \frac{\zeta_B'}{\zeta_B},
	\end{equation}
	$\mathcal{E}_1$ is defined in \eqref{eq:claimE1}, 
		\begin{equation}\label{eq:claimE2}	
	\begin{aligned}
	\mathcal{E}_2(R(x),x)
=&
-R(x)
\bigg(
\chi_A^{(4)} \chi_A 
+4\chi_A''' \chi_A \frac{\zeta_B'}{\zeta_B^2}
+6 \chi_A''\chi_A\frac{\zeta_B''}{\zeta_B}
+2(\chi_A^2)' \frac{\zeta_B'''}{\zeta_B}  
\bigg)\\
&
-R'(x)
\bigg(
2\chi_A''' \chi_A 
+6\chi_A'' \chi_A  \frac{\zeta_B'}{\zeta_B}
+6
\chi_A' \chi_A \frac{\zeta_B''}{\zeta_B}
\bigg) \\&
- R''(x)
\left(
\chi_A'' \chi_A +\frac12 (\chi_A^2)' \frac{\zeta_B'}{\zeta_B}
\right), 
	\end{aligned}
	\end{equation}
	and
	\begin{equation}\label{eq:claimE3}
	\begin{aligned}
	\mathcal{E}_3(R(x),x)
= R(x)\bigg[ 
4\chi_A'' \chi_A -2(\chi_A' )^2+ 2\frac{\zeta_B'}{\zeta_B}(\chi_A^2)'  
\bigg]
+R'(x)(\chi_A^2)' .
	\end{aligned}
	\end{equation}

Finally, $P_R$, $\mathcal{E}_2$ and $\mathcal{E}_3$ satisfy the following inequalities
	\begin{equation}\label{falta}
	\begin{aligned}
	|P_R| \lesssim & ~{}  		B^{-1}\|R' \|_{L^{\infty}}  + B^{-1}\| R\|_{L^\infty},   \\
	|P_R'| \lesssim & ~{}  	B^{-1}\|R'' \|_{L^{\infty}}	+B^{-1}\|R' \|_{L^{\infty}}  + B^{-1}\| R\|_{L^\infty},   \\
	|\mathcal{E}_2| \lesssim & ~{} (AB)^{-1}\|R'' \|_{L^{\infty}}+(AB^2)^{-1}\|R' \|_{L^{\infty}}+(AB^3)^{-1}\|R \|_{L^{\infty}},\\
	|\mathcal{E}_3| \lesssim & ~{}  A^{-1}\|R' \|_{L^{\infty}}+(AB)^{-1}\|R \|_{L^{\infty}}.
	\end{aligned}
	\end{equation}
\end{claim}

For the proof of these results, see Appendix \ref{ap:vi_xx}.

\begin{rem}\label{rem:CZ_v2xx}
	For $R=1$,  we obtain
		\begin{equation}\label{eq:1_CZ_v2_xx}
	\begin{aligned}
	\int \chi_A^2\zeta_B^2 (\partial_x^2 v_i)^2
	=&
	\int (\partial_x^2 z_i)^2+\int \tilde{R}_1(x)z_i^2
	+\int P_1(x) (\partial_x z_i)^2 
	+
	\int    \left[ P_1'(x) \frac{\zeta_B'}{\zeta_B}+P_1(x) 
	\frac{\zeta_B''}{\zeta_B}\right] z_i^2 \\
	&+\int \mathcal{E}_2(1,x) \zeta_B^2v_i^2
	+\int  \mathcal{E}_1(P_1(x),x)\zeta_B^2 v_i^2
	+\int \mathcal{E}_3(1,x) \zeta_B^2(\partial_x v_i)^2,
	\end{aligned}
	\end{equation}
	where,
	\begin{equation}\label{eq:tildeR1_P1}	
		\tilde{R}_1(x)=	  
			-2\left[\frac{\zeta_B^{(4)}}{\zeta_B}+\frac{\zeta_B'''}{\zeta_B}\frac{\zeta_B'}{\zeta_B}\right]
	,\ \quad 
	P_1(x)=
	4  \frac{\zeta_B''}{ \zeta_B}  
	-2 \bigg(\frac{\zeta_B'}{\zeta_B} \bigg)^2,
		\end{equation}
	$\mathcal{E}_1$ is defined in \eqref{eq:claimE1}, 
		\begin{equation}\label{eq:claimE2_1}	
	\begin{aligned}
	\mathcal{E}_2(1,x)
=&
-
\bigg(
\chi_A^{(4)} \chi_A 
+4\chi_A''' \chi_A \frac{\zeta_B'}{\zeta_B^2}
+6 \chi_A''\chi_A\frac{\zeta_B''}{\zeta_B}
+2(\chi_A^2)' \frac{\zeta_B'''}{\zeta_B}  
\bigg),
	\end{aligned}
	\end{equation}
	and
	\begin{equation}\label{eq:claimE3_1}
	\begin{aligned}
	\mathcal{E}_3(1,x)
=
4\chi_A'' \chi_A -2(\chi_A' )^2+ 2\frac{\zeta_B'}{\zeta_B}(\chi_A^2)'  .
	\end{aligned}
	\end{equation}
	By Lemma \ref{lem:v_w}, we obtain the estimate:
	\[
	\begin{aligned}
	\| \chi_A\zeta_B \partial_x^2 v_i\| 
	\lesssim & ~{}  \| \partial_x^2 z_i\|_{L^2}^2+B^{-1}\| \partial_x z_i\|_{L^2}^2+B^{-1}\| z_i\|_{L^2}^2 +A^{-1}B^3 \| w_i\|_{L^2}^2.
	\end{aligned}
	\]
\end{rem}

\subsection{Start of proof of Proposition \ref{prop:ineq_dtM}}\label{ahorasi}
The proof of this result is based in the following computation:
\begin{lem}\label{lem:virial_M}
	Let $(v_1,v_2)\in H^1(\R)\times H^2(\R)$ a solution of \eqref{eq:syst_v}. Consider $\phi_{A,B}=\psi_{A,B}=\chi_A^2 \varphi_B $. Then
	\begin{equation}\label{dtM}
		\begin{aligned}
		\frac{d}{dt}\mathcal{M}
		=& 
		-\dfrac{1}{2}\int \left( (\partial_x z_1)^2+\left(V_0(x)-\frac{\varphi_B}{\zeta_B^2} \partial_x(f'(Q))\right) (\partial_x z_2)^2+3(\partial_x^2 z_2)^2	\right)\\
		&+ \frac12\int  \frac{\varphi_B}{\zeta_B^2} \frac{\zeta_B'}{\zeta_B} \partial_x^2(f'(Q))  z_2^2
		+\mathcal{R}_{z}(t)+\mathcal{R}_{v}(t)+\mathcal{DR}_{v}(t)\\
		&+\int  \phi_{A,B} \tilde{G}(x) \partial_x v_2
		+\int  \phi_{A,B} \tilde{H}(x) \partial_x v_1  ,
		\end{aligned}
	\end{equation}
	where $\mathcal{R}_{z}(t)$, $\mathcal{R}_{v}(t)$ and $\mathcal{DR}_{v}(t)$ are error terms that satisfy the following bounds
	\begin{equation}\label{eq:Rz+Rv+DRv_M}
		\begin{aligned}
		|\mathcal{R}_{z}(t)|+|\mathcal{R}_{v}(t)|+|\mathcal{DR}_{v}(t)| \lesssim & ~{}  
B^{-1} \big(\|  w_1\|^2_{L^2} + \|   \partial_x w_1\|_{L^2}^2+\| w_2\|^2_{L^2} \big)\\
&~{} + B^{-1}\big(\| z_1\|_{L^2}^2 +\| z_2\|_{L^2}^2+\| \partial_x z_2\|_{L^2}^2\big),
		\end{aligned}
	\end{equation}
valid for $B$ sufficiently large.	
	
\end{lem}
\begin{proof}
	First, we recall that $z_i=\chi_A \zeta_B v_i$, and
	by \eqref{eq:psi_deriv} and Claim \ref{claim:P_CZ_vi_x}
	\begin{equation}\label{eq:m1}
	\begin{aligned}
	\int  \phi_{A,B}'  \left[( \partial_x v_1)^2+(\partial_x v_2)^2\right]
	=&
	\int  (\partial_x z_1)^2+(\partial_x z_2)^2
	+\int  \dfrac{\zeta_B''}{\zeta_B}\big(z_1^2+z_2^2\big)\\
	&+ \int  \mathcal{E}_1(1,x)\zeta_B^2 \big(v_1^2+v_2^2\big)
	+\int  (\chi_A^2)'\varphi_B  \left[( \partial_x v_1)^2+(\partial_x v_2)^2\right],
	\end{aligned}
	\end{equation}
	where $ \mathcal{E}_1(1,x)$ is given by  \eqref{eq:E1_1}. Now, using Remark \ref{rem:CZ_v2xx} \eqref{eq:1_CZ_v2_xx}, we get
	\begin{equation}\label{eq:m2}
	\begin{aligned}
	\int  \phi_{A,B}'  (\partial_x^2 v_2)^2  
	=&
	\int (\partial_x^2 z_2)^2
	+\int P_1(x) (\partial_x z_2)^2
	+
	\int    \left[\tilde{R}_1(x)+ P_1'(x) \frac{\zeta_B'}{\zeta_B}+P_1(x) 
	\frac{\zeta_B''}{\zeta_B}\right] z_2^2 \\
	&+\int \mathcal{E}_2(1,x) \zeta_B^2v_2^2
	+\int  \mathcal{E}_1(P_1(x),x)\zeta_B^2 v_2^2
	+\int \mathcal{E}_3(1,x) \zeta_B^2(\partial_x v_2)^2\\
	&+\int  (\chi_A^2)'\varphi_B (\partial_x^2 v_2)^2 ,
		\end{aligned}
	\end{equation}
		where $\tilde{R}_1(x), P_1 (x), \mathcal{E}_2(1,x), \mathcal{E}_3(1,x) $ are given by  \eqref{eq:tildeR1_P1}, \eqref{eq:claimE2_1},  \eqref{eq:claimE3_1}, and $ \mathcal{E}_1$ is gyven by \eqref{eq:claimE1}.\\
	Now, continuing with the second integral in the RHS of \eqref{eq:M'}, we have
	\begin{equation*}
	\begin{aligned}
	\int  \phi_{A,B}''' (\partial_x v_2)^2 
	= \int   \frac{(\zeta_B^2)''}{\zeta_B^2} \chi_A^2 \zeta_B^2(\partial_x v_2)^2 + \int  \left[6(\chi_A^2)' \frac{\zeta_B'}{\zeta_B}+3(\chi_A^2)'' +(\chi_A^2)'''\frac{\varphi_B}{\zeta_B^2} \right]\zeta_B^2 (\partial_x v_2)^2, 
	\end{aligned}
	\end{equation*}
	and using Claim  \ref{claim:P_CZ_vi_x},
	\begin{equation}\label{eq:m3}
	\begin{aligned}
	\int  \phi_{A,B}''' (\partial_x v_2)^2 
	=&
	\int   \frac{(\zeta_B^2)''}{\zeta_B^2}  (\partial_x z_2)^2 
	+ \int   \frac{(\zeta_B^2)''}{\zeta_B^2}   \frac{\zeta_B''}{\zeta_B}  z_2^2
	+ \int  \left( \frac{(\zeta_B^2)''}{\zeta_B^2} \right)'  \frac{\zeta_B'}{\zeta_B}  z_2^2\\
	&+\frac12 \int   \frac{(\zeta_B^2)''}{\zeta_B^2}  \bigg[\chi_A''\chi_A +2(\chi_A^2)' \frac{\zeta_B'}{\zeta_B}\bigg] \zeta_B^2 v_2^2
	+\frac12 \int  \left( \frac{(\zeta_B^2)''}{\zeta_B^2} \right)' (\chi_A^2)'  \zeta_B^2   v_2^2\\
	&+ \int  \left[ 6(\chi_A^2)' \frac{\zeta_B'}{\zeta_B}+3(\chi_A^2)'' +(\chi_A^2)'''\frac{\varphi_B}{\zeta_B^2} \right] \zeta_B^2 (\partial_x v_2)^2. 
	\end{aligned}
	\end{equation}
	For the third integral in the RHS of \eqref{eq:M'},  integrating by parts
	\begin{equation*}
	\begin{aligned}
	\int \phi_{A,B} f'(Q)\partial_x ((\partial_x v_2)^2) 
	=&-\int   \left(f'(Q)+\frac{\varphi_B}{\zeta_B^2}
	\partial_x(f'(Q))\right) \chi_A^2 \zeta_B^2 (\partial_x v_2)^2\\
	&-\int  (\chi_A^2)'\varphi_B f'(Q)(\partial_x v_2)^2.
	\end{aligned}
	\end{equation*}
	By the extended version of Claim \ref{claim:P_CZ_vi_x} and expanding the derivates in terms of $z_2$, we have
	\begin{equation}\label{eq:m4}
	\begin{aligned}
		\int & \phi_{A,B}  f'(Q)\partial_x ((\partial_x v_2)^2) \\
		=&
	- \int  \bigg[\left(f'(Q)+\frac{\varphi_B}{\zeta_B^2} \partial_x(f'(Q))\right)  \frac{\zeta_B''}{\zeta_B}  
	- \left(2\partial_x(f'(Q))-\frac{\varphi_B (\zeta_B^2)'}{\zeta_B^4} \partial_x(f'(Q)) +\frac{\varphi_B}{\zeta_B^2} \partial_x^2(f'(Q))\right)   \frac{\zeta_B'}{\zeta_B} \bigg] z_2^2\\
	&-\int  \left(f'(Q)+\frac{\varphi_B}{\zeta_B^2} \partial_x(f'(Q))\right)  (\partial_x z_2)^2 
	-\frac12 \int \left(f'(Q)+\frac{\varphi_B}{\zeta_B^2} \partial_x(f'(Q))\right)  \bigg[\chi_A''\chi_A +2(\chi_A^2)' \frac{\zeta_B'}{\zeta_B}\bigg] \zeta_B^2 v_2^2\\
	&-\frac12 \int  \partial_x \left(f'(Q)+\frac{\varphi_B}{\zeta_B^2} \partial_x(f'(Q))\right)  (\chi_A^2)'  \zeta_B^2  v_2^2
	-\int  (\chi_A^2)'\varphi_B f'(Q)(\partial_x v_2)^2.
	\end{aligned}
	\end{equation}
Collecting \eqref{eq:m1}, \eqref{eq:m2},\eqref{eq:m3} and \eqref{eq:m4}, we obtain 
	\begin{equation*}
	\begin{aligned}
	\frac{d}{dt}\mathcal{M}
	=& 
	-\dfrac{1}{2}\int  (\partial_x z_1)^2+\left(V_0(x)-\frac{\varphi_B}{\zeta_B^2} \partial_x(f'(Q))\right) (\partial_x z_2)^2+3(\partial_x^2 z_2)^2	\\
		&+ \frac12\int  
		\left[\frac{\varphi_B}{\zeta_B^2} \frac{\zeta_B'}{\zeta_B} \partial_x^2(f'(Q))  +\frac{\varphi_B}{\zeta_B^2} \frac{\zeta_B''}{\zeta_B}  \partial_x(f'(Q))  \right]z_2^2
	+\mathcal{R}_{z}(t)+\mathcal{R}_{v}(t)+\mathcal{DR}_{v}(t)\\
	&+\int  \phi_{A,B} \tilde{G}(x) \partial_x v_2
	+\int  \phi_{A,B} \tilde{H}(x) \partial_x v_1   ,
	\end{aligned}
	\end{equation*}
	where the error terms are the following: associated to $(z_1,z_2)$ is
	\begin{equation}\label{R_z}
	\begin{aligned}
	\mathcal{R}_{z}(t)
	=&
	-\frac12\int  \dfrac{\zeta_B''}{\zeta_B}\big(z_1^2+z_2^2\big)
	-\frac32
	\int    \left[\tilde{R}_1(x)+ P_1'(x) \frac{\zeta_B'}{\zeta_B}+P_1(x) 
	\frac{\zeta_B''}{\zeta_B}\right] z_2^2 \\
	&
	+ \frac12 \int  
	\left[
	 \frac{(\zeta_B^2)''}{\zeta_B^2}   \frac{\zeta_B''}{\zeta_B}  
	+  \left( \frac{(\zeta_B^2)''}{\zeta_B^2} \right)'  \frac{\zeta_B'}{\zeta_B}  
	\right]
	z_2^2\\
	&+ \frac12 \int  
	\bigg[
	f'(Q)
	\frac{\zeta_B''}{\zeta_B}  
	+ \left(2\partial_x(f'(Q))-2\frac{\varphi_B}{\zeta_B^2} \frac{ \zeta_B'}{\zeta_B} \partial_x(f'(Q))  \right) \frac{\zeta_B'}{\zeta_B} \bigg] z_2^2\\
	& +	\frac12 \int   \left(\frac{(\zeta_B^2)''}{\zeta_B^2}-3 P_1(x) \right)  (\partial_x z_2)^2 ,		
	\end{aligned}
	\end{equation}
	 associated to $(v_1,v_2)$ is
	\begin{equation}\label{Rv}
	\begin{aligned}
	\mathcal{R}_{v}(t)=
	&
	 -\frac12\int  \mathcal{E}_1(1,x)\zeta_B^2 \big(v_1^2+v_2^2\big)
	 	-\frac32\int  \bigg[\mathcal{E}_2(1,x) 
	 + \mathcal{E}_1(P_1(x),x)\bigg]\zeta_B^2 v_2^2\\
	&+\frac14 \int   \frac{(\zeta_B^2)''}{\zeta_B^2}  \bigg[\chi_A''\chi_A +2(\chi_A^2)' \frac{\zeta_B'}{\zeta_B}\bigg] \zeta_B^2 v_2^2
	+\frac14 \int  \left( \frac{(\zeta_B^2)''}{\zeta_B^2} \right)' (\chi_A^2)'  \zeta_B^2   v_2^2\\
	&+\frac14 \int \left(f'(Q)+\frac{\varphi_B}{\zeta_B^2} \partial_x(f'(Q))\right)  \bigg[\chi_A''\chi_A +2(\chi_A^2)' \frac{\zeta_B'}{\zeta_B}\bigg] \zeta_B^2 v_2^2\\
	&+\frac14 \int  \partial_x \left(f'(Q)+\frac{\varphi_B}{\zeta_B^2} \partial_x(f'(Q))\right)  (\chi_A^2)'  \zeta_B^2  v_2^2,
	\end{aligned}
	\end{equation}
	and  associated to $(\partial_x v_1, \partial_x v_2)$ is
	\begin{equation}\label{DRv}
	\begin{aligned}
	\mathcal{DR}_v(t)=&
	-\frac12\int  (\chi_A^2)'\varphi_B  \left[( \partial_x v_1)^2+(\partial_x v_2)^2+3(\partial_x^2 v_2)^2\right]
	-\frac32\int \mathcal{E}_3(1,x) \zeta_B^2(\partial_x v_2)^2\\
	&+ \frac12\int  \big[3(\chi_A^2)' (\zeta_B^2)'+3(\chi_A^2)'' \zeta_B^2+(\chi_A^2)'''\varphi_B +(\chi_A^2)'\varphi_B f'(Q)\big] (\partial_x v_2)^2.
	\end{aligned}
	\end{equation}
	We have obtained the identity \eqref{dtM}.
To conclude the proof of Lemma \ref{lem:virial_M}, we must estimate the error terms.

\subsection{Controlling error terms}
We consider  the following decomposition  for $\mathcal{R}_{z}(t)$ from \eqref{R_z},
\[
\mathcal{R}_{z}(t) =\mathcal{R}^{1}_{z}(t) + \mathcal{R}^{2}_{z}(t) +\mathcal{R}^{3}_{z}(t),
\]
where
\begin{equation*}
\begin{aligned}
\mathcal{R}^{1}_{z}(t)
=&~{}-\dfrac{1}{2}\int  \dfrac{\zeta_B''}{\zeta_B}\big(z_1^2+z_2^2\big)
+ \dfrac{1}{2}\int  \bigg[ \frac{(\zeta_B^2)''}{\zeta_B^2}  \frac{\zeta_B''}{\zeta_B}  
+ \left( \frac{(\zeta_B^2)''}{\zeta_B^2} \right)'  \frac{\zeta_B'}{\zeta_B} -3\tilde{R}_1(x)\bigg] z_2^2,\\
\mathcal{R}^2_{z}(t)=&~{}-\frac32\int  \left[P_1'(x)  \frac{\zeta_B'}{\zeta_B} +P_1(x) \frac{\zeta_B''}{\zeta_B} \right] z_2^2
+\dfrac{1}{2}\int  \left( \frac{(\zeta_B^2)''}{\zeta_B^2} -3P_1(x) \right) (\partial_x z_2)^2, \\
\mathcal{R}^3_{z}(t)=&~{}\dfrac{1}{2} \int \bigg[ 
f'(Q)
  \frac{\zeta_B''}{\zeta_B} 
 +\left(2\partial_x(f'(Q))-2\frac{\varphi_B }{\zeta_B^2}\frac{\zeta_B'}{\zeta_B} \partial_x(f'(Q)) \right)   \frac{\zeta_B'}{\zeta_B}
\bigg] z_2^2.
\end{aligned}
\end{equation*}
For $\mathcal{R}^1_{z}(t)$, recalling estimate \eqref{eq:z'/z} and $\tilde{R}_1$ (see \eqref{eq:tildeR1_P1}), and we obtain
\begin{equation} \label{eq:R1z}
\begin{aligned}
|\mathcal{R}^1_{z}(t)|
\leq  B^{-1}\| z_1\|_{L^2}^2 +B^{-1}\| z_2\|_{L^2}^2.
\end{aligned}
\end{equation}
For $\mathcal{R}^2_{z}(t)$, we  recall the form of  $P_1$ (see  \eqref{eq:tildeR1_P1}) 
and 
by  \eqref{eq:z'/z}, we conclude
\begin{equation} \label{eq:R2z}
\begin{aligned}
|\mathcal{R}^2_{z}(t)|
\lesssim  B^{-1}\| z_2\|_{L^2}^2 +B^{-1}\| \partial_x z_2\|_{L^2}^2.
\end{aligned}
\end{equation}
For $\mathcal{R}^3_{z}(t)$, first we note
\[\begin{aligned}
\bigg|\frac{\varphi_B}{\zeta_B^2} \partial_x(f'(Q)) \bigg| \lesssim B, 
\end{aligned}
\]
and by \eqref{eq:C*z'/z}, we obtain
\begin{equation} \label{eq:R3z}
|\mathcal{R}^3_{z}(t)| \lesssim  B^{-1} \| z_2\|_{L^2}^2.
\end{equation}
Collecting \eqref{eq:R1z},\eqref{eq:R2z} and \eqref{eq:R3z}, we have
\begin{equation}\label{eq:Bound_Rz}
|\mathcal{R}_z(t)|\leq B^{-1}\bigg( \|z_1\|_{L^2}^2+  \|z_2\|_{L^2}^2+ \|\partial_x z_2\|_{L^2}^2\bigg).
\end{equation}
For $\mathcal{R}_{v}(t)$, given by \eqref{Rv}, we consider the following decomposition
\begin{equation*}
\begin{aligned}
\mathcal{R}^1_{v}(t)=
	&~{}
-\frac12\int  \mathcal{E}_1(1,x)\zeta_B^2 \big(v_1^2+v_2^2\big)
-\frac32\int  \bigg[\mathcal{E}_2(1,x) 
+ \mathcal{E}_1(P_1(x),x)\bigg]\zeta_B^2 v_2^2\\
&+\frac14 \int   \frac{(\zeta_B^2)''}{\zeta_B^2}  \bigg[\chi_A''\chi_A +2(\chi_A^2)' \frac{\zeta_B'}{\zeta_B}\bigg] \zeta_B^2 v_2^2
+\frac14 \int  \left( \frac{(\zeta_B^2)''}{\zeta_B^2} \right)' (\chi_A^2)'  \zeta_B^2   v_2^2,\\
\mathcal{R}^2_{v}(t)=&~{} \frac14 \int 
\partial_x \left(f'(Q)+\frac{\varphi_B}{\zeta_B^2} \partial_x(f'(Q))\right)  (\chi_A^2)'  \zeta_B^2 v_2^2 \\
&+ 
\frac14 \int  \left(f'(Q)+\frac{\varphi_B}{\zeta_B^2} \partial_x(f'(Q))\right)  \bigg[\chi_A''\chi_A+2(\chi_A^2)' \frac{\zeta_B'}{\zeta_B}\bigg]
 \zeta_B^2 v_2^2.
\end{aligned}
\end{equation*}
We note that the terms $\mathcal{E}_{1}(P_1(x),x)$ and $\mathcal{E}_{2}(1,x)$  (see  \eqref{eq:claimE1}, \eqref{eq:tildeR1_P1} and \eqref{eq:claimE2_1}), by \eqref{eq:C*z'/z}, are bounded and satisfy the following estimates:
\[
|\mathcal{E}_{2}(1,x)|\lesssim  (AB^3)^{-1},  \ \mbox{and} \ \
|\mathcal{E}_{1}(P_1(x),x)| \lesssim (AB^3)^{-1},
\]
and for $\mathcal{E}_1(1,x)$ in \eqref{eq:E1_1},
\[
|\mathcal{E}_1(1,x)|\lesssim (AB)^{-1}.
\]
Then, we have
\[
|\mathcal{R}^1_{v}(t)|\lesssim (AB)^{-1} \left( \| \zeta_B v_1\|_{L^2}^2+\| \zeta_B v_2\|_{L^2}^2 \right).
\]
For $\mathcal{R}^2_{v}(t)$,  expanding the derivative and using \eqref{eq:bound_ZCV_B}, we obtain
\[
\begin{aligned}
\left|\mathcal{R}^2_{v}(t)\right|
\lesssim &~{} A^{-1} \int 
\left|\frac{\varphi_B}{\zeta_B^2} \partial_x^2(f'(Q))+2\left(1-\frac{\zeta_B' }{\zeta_B}\frac{\varphi_B }{\zeta_B^2}\right) \partial_x(f'(Q))\right|  \zeta_B^2 v_2^2 \\
&+ 
A^{-1} \int \left|f'(Q)+\frac{\varphi_B}{\zeta_B^2} \partial_x(f'(Q))\right| \zeta_B^2
v_2^2 ~
\lesssim ~   
A^{-1}B  \| \zeta_B v_2\|^2_{L^2}.
\end{aligned}
\]
Then,
\begin{equation}\label{eq:Bound_Rv}
|\mathcal{R}_{v}(t)| \lesssim A^{-1}B \left( \| \zeta_B v_1\|^2_{L^2} +\|\zeta_B  v_2\|^2_{L^2}\right).
\end{equation}
For $\mathcal{DR}_v$, given by \eqref{DRv}, computing directly and using Remark \ref{rem:chiA_zetaA4}, we have  
\begin{equation}\label{eq:Bound_DRv}
\begin{aligned}
|\mathcal{DR}_v(t)|\lesssim 
& BA^{-1}\left(\| \zeta_A^2 \partial_x v_1\|_{L^2}^2+\|\zeta_A^2\partial_x v_2\|^2_{L^2}
+\|\zeta_A^2\partial_x^2 v_2\|_{L^2}^2 \right).
\end{aligned}
\end{equation}
And, by \eqref{eq:Bound_Rz}, \eqref{eq:Bound_Rv} and \eqref{eq:Bound_DRv}, we obtain
\[
\begin{aligned}
& |\mathcal{R}_z(t) |+|\mathcal{R}_v (t)|+|\mathcal{DR}_v(t)|\\
 &~{}\lesssim
BA^{-1} \big( \|  \zeta_A^2 \partial_x v_1\|_{L^2}^2+\| \zeta_A^2\partial_x v_2\|^2_{L^2} +\| \zeta_A^2 \partial_x^2 v_2\|_{L^2}^2
+\|  \zeta_B v_1\|^2_{L^2} +\| \zeta_B v_2\|^2_{L^2}
 \big)\\
&  \quad  +B^{-1}\left(\| z_1\|_{L^2}^2 +\| z_2\|_{L^2}^2+\| \partial_x z_2\|_{L^2}^2\right) .
 \end{aligned}
\]
Using Lemma \ref{lem:v_w}, we conclude
\begin{equation}\label{eq:bound_errorM}
\begin{aligned}
|\mathcal{R}_z(t)|+|\mathcal{R}_v(t)|+|\mathcal{DR}_v(t)|
\lesssim &~{}
A^{-1}B^9 \big(\|  w_1\|^2_{L^2} + \|   \partial_x w_1\|_{L^2}^2+\| w_2\|^2_{L^2} \big)\\
&~{} + B^{-1}\big(\| z_1\|_{L^2}^2 +\| z_2\|_{L^2}^2+\| \partial_x z_2\|_{L^2}^2\big).
\end{aligned}
\end{equation}
This ends the proof of Lemma \ref{lem:virial_M}.
\end{proof}

\subsection{Controlling nonlinear terms} Recall \eqref{dtM}. We set
\[
 \mathcal{M}_2 =  \int \phi_{A,B} \tilde{G}(x) \partial_x v_2,\qquad 
	 \mathcal{M}_1 =\int \phi_{A,B} \tilde{H}(x) \partial_x v_1.
\]
These are the two remaining terms in \eqref{dtM} to be controlled.

\subsubsection{Control of $\mathcal{M}_2$}

Recalling that  $\tilde{G}(x)=\partial_x G(x)$ and $G$ is given by \eqref{eq:G_H}, we have
\begin{equation*}
\begin{aligned}
\mathcal{M}_2
=&-\gamma \int ((\chi_A^2)'\zeta_B^2+\chi_A^2(\zeta_B^2)') \partial_x v_2   (1-\gamma\partial_x ^2)^{-1} \left[   \partial_x^2(f'(Q))\partial_x v_2+2 \partial_x(f'(Q))\partial_x^2 v_2 \right]\\
&-\gamma \int \chi_A^2\varphi_B \partial_x^2 v_2 (1-\gamma\partial_x ^2)^{-1} \left[   \partial_x^2(f'(Q))\partial_x v_2+2 \partial_x(f'(Q))\partial_x^2 v_2 \right]\\
=: &~{} M_{21}+M_{22}.
\end{aligned}
\end{equation*}
First, we focus on $M_{21}$. Using Remark \ref{rem:CZ_v2x} and Lemma \ref{lem:v_w}, we have
\begin{equation*}
\begin{aligned}
	\| ((\chi_A^2)'\zeta_B^2+\chi_A^2(\zeta_B^2)') \partial_x v_2 \|_{L^2}
	\lesssim & ~{}  A^{-1}\|  \zeta_B^2 \partial_x v_2 \|_{L^2}+B^{-1}\|   \chi_A\zeta_B \partial_x v_2 \|_{L^2}\\
	\lesssim & ~{}  A^{-1} B^{2}\| w_2 \|_{L^2}+B^{-1}\left[  \| \partial_x z_2\|_{L^2}+B^{-1}\| z_2\|_{L^2}
	+ (AB)^{-1/2}\| w_2\|_{L^2}\right]\\
		\lesssim & ~{}  B^{-1}\big[  \| \partial_x z_2\|_{L^2}+\| z_2\|_{L^2} \big]+B^{-5}\| w_2 \|_{L^2},
\end{aligned}
\end{equation*}
and by \eqref{eq:bound_f'x_v2x}, we conclude
\begin{equation}\label{eq:M21}
\begin{aligned}
|M_{21}|
\lesssim & ~{} 
B^{-3}\bigg[  \| \partial_x z_2\|_{L^2}^2+\| z_2\|_{L^2}^2
+ \| w_2\|_{L^2}^2\bigg].
\end{aligned}
\end{equation}
Secondly, for $M_{22}$.   Set $\rho(x)=\sech(x/10)$,  making the following separation
\[
\begin{aligned}
|M_{22}|
\lesssim & ~{}  \gamma  \|\chi_A^2\varphi_B \rho_K(x)\partial_x^2 v_2\|_{L^2}   \left\| (\rho_K(x))^{-1}(1-\gamma\partial_x ^2)^{-1} \bigg[   \partial_x^2(f'(Q))\partial_x v_2+2 \partial_x(f'(Q))\partial_x^2 v_2\bigg]\right\|_{L^2}\\
\lesssim& \gamma B  \| \chi_A \zeta_B \partial_x^2 v_2\|_{L^2}  
\underbrace{\left\| (\rho_K(x))^{-1}(1-\gamma\partial_x ^2)^{-1} \bigg[   \partial_x^2(f'(Q))\partial_x v_2+2 \partial_x(f'(Q))\partial_x^2 v_2\bigg]\right\|_{L^2}}_{M_{23}}.
\end{aligned}
\]
Using Lemma \ref{lem:cosh_kink} in $M_{23}$, we obtain
\begin{equation}\label{eq:rhoK_G}
\begin{aligned}
M_{23}=&
	\left\| (\rho_K(x))^{-1}(1-\gamma\partial_x ^2)^{-1} \bigg[\rho_K(x) 
	\left\{
	  \partial_x^2(f'(Q))(\rho_K(x))^{-1}\partial_x v_2+2 \partial_x(f'(Q))(\rho_K(x))^{-1}\partial_x^2 v_2
	  \right\}\bigg]\right\|_{L^2}\\
	\leq & 
	\left\|(1-\gamma\partial_x ^2)^{-1} \bigg[(  \partial_x^2(f'(Q))(\rho_K(x))^{-1}\partial_x v_2+2 \partial_x(f'(Q))(\rho_K(x))^{-1}\partial_x^2 v_2)\bigg]\right\|_{L^2}\\
		\lesssim &
	\|(  \partial_x^2(f'(Q))(\rho_K(x))^{-1}\partial_x v_2\|_{L^2}+\| \partial_x(f'(Q)) (\rho_K(x))^{-1}\partial_x^2 v_2)\|_{L^2}.
\end{aligned}
\end{equation}

Since $\partial_x^2(f'(Q))\rho_K^{-1}\sim e^{-4|x|/5}$ and making the following decomposition, we have
\[\begin{aligned}
e^{-8|x|/5}(\partial_x v_2)^2=& e^{-8|x|/5}\zeta_B^{-2}(\chi_A\zeta_B \partial_x v_2)^2+e^{-8|x|/5}\zeta_A^{-2}(1-\chi_A^2)(\zeta_A \partial_x v_2)^2.
\end{aligned}\]

Since $e^{-4|x|/5}\zeta_B^{-1}\leq 1$ and $e^{-8A/5}\zeta_A^{-2}(A)\sim e^{-2A/5}$, using Remark \ref{rem:CZ_v2x} and Lemma \ref{lem:v_w}, we conclude
\[\begin{aligned}
\|e^{-4|x|/5}\partial_x v_2\|_{L^2}\lesssim & ~{}  \left[\|\partial_x z_2\|_{L^2}+B^{-1} \|z_2\|_{L^2}+A^{-1} \|\zeta_B v_2\|_{L^2}\right] +e^{-A/5}\|\zeta_A \partial_x v_2\|_{L^2}\\
\lesssim & ~{}  \|\partial_x z_2\|_{L^2}+B^{-1} \|z_2\|_{L^2}+A^{-1} \|w_2\|_{L^2}.
\end{aligned}
\]

And, for the second term in the RHS of \eqref{eq:rhoK_G}. We note $\partial_x (f'(Q))\rho^{-1}\sim e^{-4|x|/5}$ and repeating the decomposition, we have
 \[
\begin{aligned}
e^{-8|x|/5}(\partial_x^2 v_2)^2
=& e^{-8|x|/5}\zeta_B^{-2} \chi_A^2\zeta_B^{2}(\partial_x^2 v_2)^2+ e^{-8|x|/5}\zeta_A^{-2} (1-\chi_A^2)\zeta_A^{2} (\partial_x^2 v_2)^2.
\end{aligned}
\]
By a similar argument as before, $e^{-4|x|/5}\zeta_B^{-1}\leq 1$ and $e^{-8A/5}\zeta_A^{-2}(A)\sim e^{-2A/5}$, applying Remark \ref{rem:CZ_v2xx} and Lemma \ref{lem:v_w}, we have
\[
\begin{aligned}
\|e^{-4|x|/5} \partial_x^2 v_2\|_{L^2}
\lesssim &~{}  \|\chi_A\zeta_B \partial_x^2 v_2\|_{L^2}+ e^{-A/5}\| \zeta_A \partial_x^2 v_2\|_{L^2}\\
\lesssim &~{} \| \partial_x^2 z_2\|_{L^2}+B^{-1}\| \partial_x z_2\|_{L^2}+B^{-2}\| z_2\|_{L^2} +(A^{-1}B^3)^{1/2} \| w_2\|_{L^2}.
\end{aligned}
\]
Finally, for $M_{22}$ we have
\begin{equation}\label{eq:M22}
\begin{aligned}
|M_{22}|\lesssim
B^{-3}\left(\| \partial_x^2 z_2\|_{L^2}^2+\| \partial_x z_2\|_{L^2}^2+B^{-2}\| z_2\|_{L^2}^2 +A^{-1}B^3 \| w_2\|_{L^2}^2\right).
\end{aligned}
\end{equation}

Collecting  \eqref{eq:M21} and \eqref{eq:M22}, we conclude
\begin{equation}\label{eq:boundM2}
\begin{aligned}
\mathcal{M}_2
\lesssim & ~{}   
B^{-3}\bigg[  \| \partial_x^2 z_2\|_{L^2}^2+\| \partial_x z_2\|_{L^2}^2+\| z_2\|_{L^2}^2
+ \| w_2\|_{L^2}^2\bigg].
\end{aligned}
\end{equation}

\subsubsection{Control of $\mathcal{M}_1$.}

Recalling that $\tilde{H}=\partial_x H$, $H$ is given by \eqref{eq:G_H} and using Lemma \ref{lem:estimates_IOp}, we obtain
\[
\begin{aligned}
|\mathcal{M}_1|
\lesssim & ~{} 
\| \chi_A\varphi_B \partial_x v_1\|_{L^2}
\| \chi_A(1-\partial_x^2)^{-1}\partial_x N^{\perp} \|_{L^2}\\
\lesssim & ~{} 
\| \chi_A\varphi_B \partial_x v_1\|_{L^2}
\| \zeta_A^2(1-\partial_x^2)^{-1}\partial_x N^{\perp} \|_{L^2}.
\end{aligned}
\]
Now, by a similar computation that \eqref{eq:LambaN}, we have
\begin{equation}\label{eq:LambdaP_N}
\|\zeta_A^2(1-\gamma\partial_x ^2)^{-1} \partial_x ( N^{\perp} )\|_{L^2}
\leq \gamma^{-1}(a_1^2+\|u_1\|_{L^\infty}\|w_1\|_{L^2}).
\end{equation}
Then, by \eqref{eq:estimates_psiB_v1}, Lemma \ref{lem:v_w}, \eqref{eq:N_0_bounded} and the above estimates, we conclude
\begin{equation}\label{eq:boundM1}
\begin{aligned}
| \mathcal{M}_1|
\lesssim & ~{}   \gamma^{-2}B\|w_1\|_{L^2} (a_1^2+\|u_1\|_{L^\infty}\|w_1\|_{L^2})^2\\
\lesssim & ~{}  B^{-1}\|w_1\|_{L^2}^2+\gamma^{-4}B^3 (a_1^4+\|u_1\|_{L^\infty}^4\|w_1\|_{L^2}^2).
\end{aligned}
\end{equation}

\subsection{End of proof Proposition \ref{prop:ineq_dtM}}
Using a similar computation that Lemma \ref{lem:bound_V}, we are able to estimates $ \frac{d} {dt} \mathcal{M} $.  
 Set $B=\delta^{-1/19}$, and considering \eqref{eq:bound_errorM}, \eqref{eq:boundM1}, and \eqref{eq:boundM2}, we obtain
\begin{equation*}
\begin{aligned}
	\frac{d}{dt}\mathcal{M}
\leq & ~{}  
-\dfrac{1}{2}\int \left( (\partial_x z_1)^2+\left(V_0(x)-\frac{\varphi_B}{\zeta_B^2} \partial_x(f'(Q))\right) (\partial_x z_2)^2+2(\partial_x^2 z_2)^2	\right) \\
&+ \frac12\int 
\bigg( 
\frac{\varphi_B}{\zeta_B^2} \frac{\zeta_B'}{\zeta_B} \partial_x^2(f'(Q)) 
+\frac{\varphi_B}{\zeta_B^2} \frac{\zeta_B''}{\zeta_B}  \partial_x(f'(Q))  
\bigg)
 z_2^2\\
&+ C \max\{{\color{red}B^9A^{-1}},B^{-1},\delta\} \left( \| \partial_x w_1\|_{L^2}^2+\| w_1\|_{L^2}^2
+\| w_2\|_{L^2}^2 \right)\\
&+C B^{-1}\bigg[ \| \partial_x z_2\|_{L^2}^2+\| z_2\|_{L^2}^2+\| z_1\|_{L^2}^2\bigg]
+
C |a_1|^3.
\end{aligned}
\end{equation*}
Since 
\[
\frac{\varphi_B}{\zeta_B^2} \partial_x(f'(Q))<0,\ \mbox{and }\ 
 \quad \left| \frac{\varphi_B}{\zeta_B^2} \frac{\zeta_B'}{\zeta_B} \partial_x^2(f'(Q)) 
+\frac{\varphi_B}{\zeta_B^2} \frac{\zeta_B''}{\zeta_B}  \partial_x(f'(Q))   \right|\lesssim  1,
\]
we conclude
\begin{equation*}
\begin{aligned}
	\frac{d}{dt}\mathcal{M}
 \leq & ~{}  
 -\dfrac{1}{2}\int (\partial_x z_1)^2+\left(V_0(x) - CB^{-1}\right) (\partial_x z_2)^2+2(\partial_x^2 z_2)^2	\\
 &+ C \| z_2\|_{L^2}^2
 +C B^{-1}\| z_1\|_{L^2}^2\\
 &+C\max\{B^9A^{-1},B^{-1},\delta\} \left( \| \partial_x w_1\|_{L^2}^2+\| w_1\|_{L^2}^2
 +\| w_2\|_{L^2}^2 \right)+
C |a_1|^3.
\end{aligned}
\end{equation*}
Calling $C_3=C$, and using \eqref{eq:scales}, the proof is completed.

\section{A second transfer estimate}\label{sec:5}

The variation of the virial $\M$ involve the terms $\partial_x z_1$ and $\partial_x^2 z_2$, these terms do not appear in the variation of the  virial related to the dual problem.  Hence, we need to find a way to transfer information between  the terms $\partial_x z_1$ to $\partial_x^2 z_2$. The  virial $\mathcal N$, defined  as
\begin{equation}\label{eq:N_def}
\begin{gathered}
\mathcal{N}=\int \rho_{A,B} \tv_1 v_2=\int \rho_{A,B} \partial_x v_1 v_2,
\end{gathered}
\end{equation}
where $\rho_{A,B}$ is a well-chosen localized weight depending on $A$ and $B$, its variation will give us that relation.
A similar quantity was considered in \cite{kink}. Note that the virial $\mathcal{N}$ considers the dynamics in \eqref{eq:syst_v} and \eqref{eq:syst_vx}.

\subsection{A virial identity for $\N$}
\begin{lem}\label{lem:identity_dtN}
	Let $(v_1,v_2)\in H^1(\R)\times H^2(\R)$ a solution of \eqref{eq:syst_v}. Consider $ \rho_{A,B} $ an even smooth bounded function to be a choose later. Then
	\begin{equation}
	\begin{aligned}\label{eq:N'}
	\frac{d}{dt}\mathcal{N}=&~{} 
	2\int \rho_{A,B}''   (\partial_x v_2)^2
	-\int \rho_{A,B} \big[ (\partial_x^2  v_2)^2+ V_0(x) (\partial_x v_2)^2\big]
	+\int \rho_{A,B} (\partial_x v_1)^2\\
	&+\frac12 \int \partial_x^2 [\rho_{A,B} V_0(x)]v_2^2-\frac12 \int \rho_{A,B}^{(4)} v_2^2     
	+ \int \rho_{A,B} v_2 \tilde{G}(x)
	+  \int \rho_{A,B}  \partial_x v_1 H(x)  .
	\end{aligned}
	\end{equation}
\end{lem}

\begin{proof}
	Computing the variation of $\mathcal{N}$, using \eqref{eq:syst_v} and \eqref{eq:syst_vx}, we obtain
	\[
	\begin{aligned}
	\frac{d}{dt}\mathcal{N}
	=&\int \rho_{A,B} v_2 \LL\partial_x  \tv_2
	+\int \rho_{A,B} \tv_1^2
	+ \int \rho_{A,B} v_2 (\tilde{G}(x)-\partial_x(f'(Q)) v_2)
	+  \int \rho_{A,B}  \tv_1 H(x).
	\end{aligned}
	\]
	Integrating by parts the first integral of the RHS $N_1$, we have
	\[
	\begin{aligned}
	\int \rho_{A,B} v_2 \LL\partial_x  \tv_2
	=&
	\int \rho_{A,B} v_2 (-\partial_x^3  \tv_2+V_0(x)
	\partial_x  \tv_2)\\
	=&
	-\int \partial_x^2(\rho_{A,B} v_2) \partial_x  \tv_2-\int 
	\left( \partial_x [\rho_{A,B} V_0(x)
	]v_2\tv_2  +\rho_{A,B} V_0(x)
	\tv_2^2  \right)  \\
	=&
	\int \rho_{A,B}''' v_2 \tv_2 +2\int \rho_{A,B}''   \tv_2^2-\int \rho_{A,B}  \left( (\partial_x  \tv_2)^2+V_0(x) \tv_2^2 \right)
	-\int \partial_x [\rho_{A,B} V_0(x)
	]v_2\tv_2.  
	\end{aligned}
	\]
	Then, we get
	\[
	\begin{aligned}
	\frac{d}{dt}\mathcal{N}
	=&~{} -\int \rho_{A,B}  \left( (\partial_x  \tv_2)^2+ V_0(x) \tv_2^2 \right)
	+\int \rho_{A,B} \tv_1^2+2\int \rho_{A,B}''   \tv_2^2
	-\int \partial_x [\rho_{A,B} V_0(x)]v_2\tv_2\\
	&+\int \rho_{A,B}''' v_2 \tv_2     
	+ \int \rho_{A,B} v_2 (\tilde{G}(x)-\partial_x(f'(Q)) v_2)
	+  \int \rho_{A,B}  \tv_1 H(x).
	\end{aligned}
	\]
	Rewriting the last expression in terms of $(v_1,v_2)$, we obtain
	\begin{equation*}
	\begin{aligned}
	\frac{d}{dt}\mathcal{N}=&~{} -\int \rho_{A,B}  \left( (\partial_x^2  v_2)^2+V_0(x) (\partial_x v_2)^2\right)
	+\int \rho_{A,B} (\partial_x v_1)^2+2\int \rho_{A,B}''   (\partial_x v_2)^2-\frac12 \int \rho_{A,B}^{(4)} v_2^2 \\
	&+\frac12 \int \partial_x^2 [\rho_{A,B} V_0(x)]v_2^2    
	+ \int \rho_{A,B} v_2( \tilde{G}(x)-\partial_x(f'(Q)) v_2)
	+  \int \rho_{A,B}  \partial_x v_1 H(x).
	\end{aligned}
	\end{equation*}
	The proof concludes from the fact that $\rho_{A,B} v_2^2$ is even and $\partial_x(f'(Q))$ is an odd function.
\end{proof}

Now we choose the weight function $\rho_{A,B}$. As in \cite{kink}, let
\begin{equation}\label{rho_AB}
\rho_{A,B}(x)=\chi_A^2 \zeta_B^2,
\end{equation}
with $\chi_A$ and $\zeta_B$ introduced in \eqref{eq:psi_chiA} and  \eqref{eq:zetaB}.

We will make the connection between \eqref{eq:N'} and the variables $(z_1,z_2)$, through the following result. Recall that from \eqref{primera} and Proposition \ref{prop:ineq_dtM}, $\gamma=B^{-4}$, $B=\delta^{-1/19}$.

\begin{prop}\label{prop:virial_N}
	Under \eqref{rho_AB}, the following holds. There exist $C_4$ and $\delta_4>0$ such that for $\gamma=B^{-4}$ and for any $0<\delta\leq\delta_4$, the following holds. Fix $B=\delta^{-1/5}$ holds. Assume that for all $t\geq 0$,  \eqref{eq:ineq_hip} holds. Then, for all $t\geq 0$, 	
	\begin{equation}\label{eq:dN_z}
	\begin{aligned}
	\frac{d}{dt}\N 
	 \ge &~{} 
	 \frac12	\int (\partial_x z_1)^2 -C_4 \int \left[ (\partial_x^2 z_2)^2  + (\partial_x z_2)^2 + z_2^2+z_1^2\right]
	 \\
	&~{} -C_4 \delta^{14/19} \bigg(\| w_1\|_{L^2}^2+\| w_2\|_{L^2}^2+|a_1|^3\bigg).
	\end{aligned}
	\end{equation}
\end{prop}

\subsection{Start of proof of Proposition \ref{prop:virial_N}} The proof of this result is based in the following result, which relates Lemma \ref{lem:identity_dtN} and the variables $z_i$.

\begin{lem}\label{lem:dtN}
	Let $(v_1,v_2)\in H^1(\R)\times H^2(\R)$ a solution of \eqref{eq:syst_v}. Consider $ \rho_{A,B}=\chi_A^2 \zeta_B^2 $, then
		\begin{equation}\label{dtN}
	\begin{aligned}
	\frac{d}{dt}\mathcal{N}=& 
	-\int \left[ (\partial_x^2 z_2)^2  + \left(1-f'(Q)\right) (\partial_x z_2)^2 +\frac12 	  \partial_x^2 (f'(Q)) z_2^2\right]
	+ 	\int (\partial_x z_1)^2 \\
	&+\mathcal{RZ}(t)+\mathcal{RV}(t)+\mathcal{RDV}(t)
	+ \int \rho_{A,B} v_2 \tilde{G}(x)
	+  \int \rho_{A,B}  \partial_x v_1 H(x),
	\end{aligned}
	\end{equation}
	where $\mathcal{RZ}(t)$,  $\mathcal{RV}(t)$ and $\mathcal{RDV}(t)$ are error term that satisfy the following estimates
		\begin{equation*}
	\begin{aligned}
	|\mathcal{RZ}(t)|\lesssim & ~{}  B^{-2}\| z_1\|_{L^2}^2+B^{-1}\| z_2\|_{L^2}^2+B^{-2}\| \partial_x z_2\|_{L^2}^2,\\
	|\mathcal{RV}(t)|
	\lesssim & ~{}  (AB)^{-1}\gamma^{-2 }\| w_1\|_{L^2}^2+ A^{-1}\| w_2\|_{L^2}^2,\\
	|\mathcal{RDV}(t)|
	\lesssim & ~{} (AB)^{-1} \gamma^{-1}\| w_2\|_{L^2}^2.
	\end{aligned}
	\end{equation*}
\end{lem}
\begin{proof}
	First, we consider the following decomposition from \eqref{eq:N'}:
	\begin{equation}\label{eq:decomp_N'}
	\begin{aligned}
	\frac{d}{dt}\mathcal{N}=&~{} 
	\int \rho_{A,B} (\partial_x v_1)^2
	-\int \rho_{A,B}  (\partial_x^2  v_2)^2-\int \rho_{A,B} V_0(x) (\partial_x v_2)^2
	+2\int \rho_{A,B}''   (\partial_x v_2)^2
	\\
	&+\frac12 \int \partial_x^2 [\rho_{A,B} V_0(x)]v_2^2-\frac12 \int \rho_{A,B}^{(4)} v_2^2     
	+  \int \rho_{A,B}  \partial_x v_1 H(x)+ \int \rho_{A,B} v_2 \tilde{G}(x)	\\
	=:&~{} (N_1+N_2+N_3+N_4)+(N_5+N_6)+(\mathcal{N}_1+\mathcal{N}_2).
	\end{aligned}
	\end{equation}
	Secondly, from the definition of $\rho_{A,B}$ \eqref{rho_AB}
	\begin{equation}\label{eq:rho_der}
	\begin{aligned}
			\rho_{A,B}'=&(\chi_A^2)' \zeta_B^2+\chi_A^2 (\zeta_B^2)',\\
		\rho_{A,B}''=&~{} (\chi_A^2)'' \zeta_B^2+2(\chi_A^2)' (\zeta_B^2)'+\chi_A^2 (\zeta_B^2)'',\\
		\rho_{A,B}'''=&~{}(\chi_A^2)''' \zeta_B^2+3(\chi_A^2) ''(\zeta_B^2)'+3(\chi_A^2)' (\zeta_B^2)''+\chi_A^2 (\zeta_B^2)''',\\
		\rho_{A,B}^{(4)}=&~{}
		(\chi_A^2)^{(4)} \zeta_B^2+4(\chi_A^2)'''(\zeta_B^2)'
		+6(\chi_A^2)'' (\zeta_B^2)''+4(\chi_A^2)' (\zeta_B^2)'''+\chi_A^2 (\zeta_B^2)^{(4)} .
	\end{aligned}
	\end{equation}
	For $N_3$, applying Claim \ref{claim:P_CZ_vi_x} with $i=2$ and $P(x)=V_0(x)$, we have
	\begin{equation}\label{eq:n3}
	\begin{aligned}
	-N_3=&
	\int V_0(x) (\partial_x z_2)^2 +
	\int   \left[ V_{0}(x)
	\frac{\zeta_B''}{\zeta_B} +\partial_x(V_0(x)) \frac{\zeta_B'}{\zeta_B}\right] z_2^2
	+\int \mathcal{E}_1( V_0(x),x)\zeta_B^2 v_2^2,
	\end{aligned}
	\end{equation}
	where $\mathcal{E}_1$ is given by \eqref{eq:claimE1}
	%
	For $N_4$, by \eqref{eq:rho_der}, we have
		\begin{equation}
	\begin{aligned}
	\frac12 N_4
	=& \int [(\chi_A^2)'' \zeta_B^2+(\chi_A^2)'(\zeta_B^2)']   (\partial_x v_2)^2
	+\int  \frac{(\zeta_B^2)''}{\zeta_B^2}   \chi_A^2 \zeta_B^2(\partial_x v_2)^2,
	\end{aligned}
	\end{equation}
	and using  Claim \ref{claim:P_CZ_vi_x}, with $i=2$ and $P(x)=(\zeta_B^2)''/\zeta_B^2$, we get
		\begin{equation}\label{eq:n4}
	\begin{aligned}
	\frac12 N_4
=& \int \frac{(\zeta_B^2)''}{\zeta_B^2}   (\partial_x z_2)^2 +
\int   \left[  \left(\frac{(\zeta_B^2)''}{\zeta_B^2}  \right)' \frac{\zeta_B'}{\zeta_B}+
\frac{(\zeta_B^2)''}{\zeta_B^2}   
\frac{\zeta_B''}{\zeta_B}\right] z_2^2
+\int \mathcal{E}_1\left( \frac{(\zeta_B^2)''}{\zeta_B^2}  ,x\right)\zeta_B^2 v_2^2\\
&+\int [(\chi_A^2)'' +2(\chi_A^2)' \frac{\zeta_B'}{\zeta_B} ]  \zeta_B^2(\partial_x v_2)^2.
	\end{aligned}
	\end{equation}
	Now, for $N_5$,  expanding the derivative, replacing  \eqref{eq:rho_der} and using definition of $z_2$, we have
	\begin{equation}\label{eq:n5}
	\begin{aligned}
	N_5
	=& -	\int  \left[ -\partial_x^2(V_{0}(x))   - \frac{(\zeta_B^2)''}{\zeta_B^2}  V_{0}(x)
	-2\frac{(\zeta_B^2)'}{\zeta_B^2} \partial_x(V_{0}(x)) \right] z_2^2\\
	&
	 +\int \bigg(\left[ 2(\chi_A^2)' \frac{\zeta_B'}{\zeta_B}+(\chi_A^2)''\right] V_{0}(x)+2(\chi_A^2)'  \partial_x(V_{0}(x))  \bigg) \zeta_B^2 v_2^2. 
	\end{aligned}
	\end{equation}
	%
	Finally, for $N_6$, reeplacing \eqref{eq:rho_der}, we have
	\begin{equation}\label{eq:n6}
	\begin{aligned}
	N_6
	=&
	\int \frac{(\zeta_B^2)^{(4)}}{\zeta_B^2} z_2^2
	+\int \left[ 4(\chi_A^2)'\frac{(\zeta_B^2)'''}{\zeta_B^2}+ 6(\chi_A^2)''\frac{(\zeta_B^2)''}{\zeta_B^2}+ 8(\chi_A^2)''' \frac{\zeta_B'}{\zeta_B}+ (\chi_A^2)^{(4)}  \right] \zeta_B^2v_2^2   . 
	\end{aligned}
	\end{equation}
	Therefore, collecting  \eqref{eq:1_CZ_v2_xx}, \eqref{eq:1CZ_B_vix}, \eqref{eq:n3}, \eqref{eq:n4}, \eqref{eq:n5} and \eqref{eq:n6} (and also for $N_1$ and $N_2$ we use the relations in Remarks \ref{rem:CZ_v2x}  and \ref{rem:CZ_v2xx}), we conclude
	\begin{equation*}
	\begin{aligned}
	\frac{d}{dt}\mathcal{N}=& 
	-\int \left[ (\partial_x^2 z_2)^2  + V_{0}(x) (\partial_x z_2)^2 +\frac12 	  \partial_x^2 (f'(Q)) z_2^2\right]
	+ 	\int (\partial_x z_1)^2 \\
	&+\mathcal{RZ}(t)+\mathcal{RV}(t)+\mathcal{RDV}(t)
	+ \int \rho_{A,B} v_2 \tilde{G}(x)
	+  \int \rho_{A,B}  \partial_x v_1 H(x),
	\end{aligned}
	\end{equation*}
	where the error term related to $z=(z_1,z_2)$ is
	\begin{equation*}
	\begin{aligned}
	\mathcal{RZ}(t)=&
		\int   \frac{\zeta_B''}{\zeta_B} z_1^2-
	\int   \left[ V_0(x)
	\frac{\zeta_B''}{\zeta_B} +\partial_x(V_{0}(x)) \frac{\zeta_B'}{\zeta_B}\right] z_2^2\\
		&-	\frac12\int  \left[   - \frac{(\zeta_B^2)''}{\zeta_B^2} V_{0}(x)
	-2\frac{(\zeta_B^2)'}{\zeta_B^2} \partial_x(V_{0}(x)) + \frac{(\zeta_B^2)^{(4)}}{\zeta_B^2} \right] z_2^2\\
	&
		-
		\int   \left[\tilde{R}_1(x)+ P_1'(x) \frac{\zeta_B'}{\zeta_B}+P_1(x) 
		\frac{\zeta_B''}{\zeta_B}\right] z_2^2 \\
	&+
	2\int   \left[  \left(\frac{(\zeta_B^2)''}{\zeta_B^2}  \right)' \frac{\zeta_B'}{\zeta_B}+
	\frac{(\zeta_B^2)''}{\zeta_B^2}   
	\frac{\zeta_B''}{\zeta_B}\right] z_2^2
	+\int \bigg[2\frac{(\zeta_B^2)''}{\zeta_B^2} -P_1(x)  \bigg](\partial_x z_2)^2,
			\end{aligned}
	\end{equation*}
	the related to $(v_1,v_2)$ is
	\begin{equation*}
	\begin{aligned}
	\mathcal{RV}(t)
	=&
	\int \mathcal{E}_1(1,x)\zeta_B^2 v_1^2\\
	&
	-\int \bigg[\mathcal{E}_1( V_{0}(x),x)+\mathcal{E}_2(1,x)+ \mathcal{E}_1(P_1(x),x)-2\mathcal{E}_1\left( \frac{(\zeta_B^2)''}{\zeta_B^2}  ,x\right)\bigg]\zeta_B^2 v_2^2\\
	& +\frac12\int 
	\bigg[\left( 2(\chi_A^2)' \frac{\zeta_B'}{\zeta_B}+(\chi_A^2)''\right) V_{0}(x)
	+2 (\chi_A^2)'  \partial_x(V_{0}(x)) 
	 \bigg]\zeta_B^2 v_2^2\\
	&-\frac12\int \left[ 4(\chi_A^2)'\frac{(\zeta_B^2)'''}{\zeta_B^2}+ 6(\chi_A^2)''\frac{(\zeta_B^2)''}{\zeta_B^2}+ 8(\chi_A^2)''' \frac{\zeta_B'}{\zeta_B}+ (\chi_A^2)^{(4)}  \right] \zeta_B^2v_2^2,  
			\end{aligned}
	\end{equation*}
	and the related to $\partial_x v_2$ is
	\begin{equation*}
	\begin{aligned}
	\mathcal{RDV}(t)=&~{} 2\int \bigg[(\chi_A^2)'' +(\chi_A^2)'\frac{\zeta_B'}{\zeta_B} - \frac12  \mathcal{E}_3(1,x)\bigg]  \zeta_B^2 (\partial_x v_2)^2.
	\end{aligned}
	\end{equation*}
	 It is clear, from \eqref{eq:z'/z}, that the error terms satisfies the following estimates
	\begin{equation}\label{eq:bound_RZ}
	\begin{aligned}
	|\mathcal{RZ}(t)|\lesssim & ~{}  B^{-1}\bigg(\| z_1\|_{L^2}^2+\| z_2\|_{L^2}^2+\| \partial_x z_2\|_{L^2}^2\bigg),
	\end{aligned}
	\end{equation}
	and
	\begin{equation*}
	\begin{aligned}
	|\mathcal{RV}(t)|\lesssim & ~{}  (AB)^{-1}\| \zeta_B v_1\|_{L^2}^2+ A^{-1}\| \zeta_B v_2\|_{L^2}^2,\\
	|\mathcal{RDV}(t)|\lesssim & ~{} (AB)^{-1} \| \zeta_B \partial_x v_2\|_{L^2}^2.
	\end{aligned}
	\end{equation*}
	Recalling that $\gamma=B^{-4}$ and applying Lemma \ref{lem:v_w} , we conclude
	\begin{equation}\label{eq:bound_RV_RDV}
	\begin{aligned}
	|\mathcal{RV}(t)|
	\lesssim & ~{}  A^{-1}B^{7}\| w_1\|_{L^2}^2+ A^{-1}\| w_2\|_{L^2}^2,\\
	|\mathcal{RDV}(t)|
	\lesssim & ~{} A^{-1} B^{7}\| w_2\|_{L^2}^2.
	\end{aligned}
	\end{equation}
	This concludes the proof of the Lemma \ref{lem:dtN}.
\end{proof}

\subsection{Control of nonlinear terms} The nonlinear terms in \eqref{dtN} are denoted
\[
\mathcal{N}_2 = \int \rho_{A,B} v_2 \tilde{G}(x), \qquad 
	\mathcal{N}_1=  \int \rho_{A,B}  \partial_x v_1 H(x).
\]

\noindent
{\bf Control of $\mathcal{N}_2$.}
Recalling  that  $\tilde{G}=\partial_x G$ and  $G$ is given by \eqref{eq:G_H}, using definition of $z_2$ and \eqref{rho_AB}, we have
\[
\begin{aligned}
|\mathcal{N}_2|
=& \left|\int ((\chi_A \zeta_B)' z_2+\chi_A \zeta_B \partial_x z_2)  G(x)\right| \lesssim   \gamma(\|z_2 \|_{L^2}+\|\partial_x z_2 \|_{L^2}) 
\|  G\|_{L^2}.
\end{aligned}
\]
By Cauchy-Schwarz inequality, \eqref{eq:bound_f'x_v2x} and Lemma \ref{lem:v_w},  we conclude
\begin{equation}\label{eq:N2}
\begin{aligned}
|\mathcal{N}_2|
\lesssim & ~{} \gamma^{1/2}(\|z_2 \|_{L^2}+\|\partial_x z_2 \|_{L^2})\bigg(\| \partial_x z_2\|_{L^2}
+ \|z_2\|_{L^2} +e^{-(p-1)A}  (  \| \zeta_B v_2\|_{L^2} + \|\zeta_B \partial_x v_2\|_{L^2}) \bigg) \\
\lesssim & ~{} \gamma^{1/2} \bigg(\| \partial_x z_2\|_{L^2}^2
+ \|z_2\|_{L^2}^2 +e^{-2(p-1)A}  (  \| \zeta_B v_2\|_{L^2}^2 + \|\zeta_B \partial_x v_2\|_{L^2}^2) \bigg) \\
\lesssim & ~{} 
\gamma^{1/2} \bigg(\| \partial_x z_2\|_{L^2}^2
+ \|z_2\|_{L^2}^2 +e^{-2(p-1)A}\gamma^{-1}  \| w_2\|_{L^2}^2 \bigg).
\end{aligned}
\end{equation}

\medskip

\noindent
{\bf Control of $\mathcal{N}_1$.} We observe that
\[
	(\chi_A \zeta_B)^2 \partial_x v_1=	\chi_A \zeta_B\partial_x z_1 -(\chi_A \zeta_B)' z_1,
\]
then, we have
\[
\begin{aligned}
\mathcal{N}_1=&  \int [\chi_A \zeta_B\partial_x z_1 -(\chi_A \zeta_B)' z_1] H(x).
\end{aligned}
\]
Recalling that $H(x)$ is given by \eqref{eq:G_H}. Moreover,  using \eqref{eq:LambaN}, \eqref{eq:N_0_bounded} and \eqref{eq:ineq_hip} we have
\begin{equation}\label{eq:N1}
\begin{aligned}
|\mathcal{N}_1|
 \lesssim & ~{}  B^{-1}\left( \|  \partial_x z_1\|_{L^2}^2 +\| z_1\|_{L^2}^2\right)
+ B^5\left(a_1^4+\| u_1\|_{L^\infty}^2\| w_1\|_{L^2}^2\right) \\
 \lesssim & ~{}  B^{-1}\left( \|  \partial_x z_1\|_{L^2}^2 +\| z_1\|_{L^2}^2\right)
+ B^5 \delta \left(|a_1|^3+\| w_1\|_{L^2}^2\right).
\end{aligned}
\end{equation}
\subsection{End of proof of  Proposition \ref{prop:virial_N}}
Since $\gamma=B^{-4}$ and $B=\delta^{-1/19}$, collecting \eqref{eq:bound_RZ}, \eqref{eq:bound_RV_RDV}, \eqref{eq:N2} and \eqref{eq:N1}, we obtain for some $C_4>0$ fixed in \eqref{dtN}
	\begin{equation*}
	\begin{aligned}
	\frac{d}{dt}\mathcal{N} (t) \geq &~{} 	\int (\partial_x z_1)^2 -\int \left[ (\partial_x^2 z_2)^2  + \left(1-f'(Q)\right) (\partial_x z_2)^2 +\frac12 	  \partial_x^2 (f'(Q)) z_2^2\right]
	 \\
	&- |\mathcal{RZ}(t)|-|\mathcal{RV}(t)|-|\mathcal{RDV}(t)| -|\mathcal{N}_1| -|\mathcal{N}_2|\\
 \ge & ~{}
 \frac12	\int (\partial_x z_1)^2 -C_4\int \left[ (\partial_x^2 z_2)^2  + (\partial_x z_2)^2 + z_2^2+z_1^2\right] \\
& ~{}-C_4\max\{A^{-1} B^7,\delta^{14/19}\}\bigg(\| w_1\|_{L^2}^2+\| w_2\|_{L^2}^2 + |a_1|^3\bigg).
	\end{aligned}
	\end{equation*}
	Using \eqref{eq:scales}, $A^{-1} B^7\ll B^{-3} \ll \delta^{14/19}$. This ends the proof.

\section{Proof of Theorem \ref{thm1}}\label{sec:6}

Before starting the proof of Theorem \ref{thm1}, we need a coercivity result to deal with the term
\[
\int \sech(x)w_1^2
\]
that appears in the virial estimates of $\I$ (see \eqref{eq:dI_w}). We will decompose this term in terms of the variables $(w_1,w_2)$ and $(z_1,z_2)$. The last ones involve the variables $(v_1,v_2)$; then we should be able to reconstruct the operator $\LL$ from our computations.

\subsection{Coercivity} 

We shall prove a coercivity result adapted to the orthogonality conditions  $\Jap{u}{Q'}=\Jap{u}{\LL(\phi_0)}=0$ in \eqref{eq:orthogonal_condition}, where $\phi_0$ was introduced in \eqref{eq:phi0}. The idea is to follow the strategy used in \cite{weinstein_modulation} and \cite{CMPS_zakharov2d}. 
Recently, in \cite{KMM} the operator $\LL$ was appeared in a similar setting. It has a unique negative single eigenvalue  $\tau_0=-(p+1)(p+3)/4$,  associated to an $L^2$ eigenfunction denoted $Y_0$.

\medskip

Our first result is a coercivity property for $\LL$ whenever the first eigenfunction $Y_0$ is changed by 
$\LL(\phi_0)$.

\begin{lem}[Coercivity lemma]\label{lem:coercivity}
	Consider the bilinear form
	\[
	H(u,v)=\Jap{\LL (u)}{v}=\int (\partial_x u \partial_x v+uv-f'(Q)uv) .
	\]
	Then, there exists $\lambda>0$ such that
	\begin{equation}\label{eq:coercivity_V}
H(v,v)\geq \lambda \|v\|_{H^1}^2,
	\end{equation}
	for all $v \in H^1(\R)$ satisfying $\Jap{v}{Q'}=\Jap{v}{\LL(\phi_0)}=0$.
\end{lem}

\begin{proof}
See Appendix \ref{ap:proof_coercivity}.
\end{proof}

We will need a weighted version of the previous result. See e.g. C\^ote-Mu\~noz-Pilod-Simpson \cite{CMPS_zakharov2d} for a very similar proof of this result. 

\begin{lem}[Coercivity with weight function]\label{lem:coerc_weight}
	Consider the bilinear form
	\[
	H_{\phi_\ell} (u,v)=\Jap{\sqrt{\phi_\ell} \LL(u)}{\sqrt{\phi_\ell} v}=\int \phi_\ell (\partial_x u \partial_x v+uv-f'(Q)uv) .
	\]
	for $\phi_\ell$ smooth and bounded and such that $0<\phi_\ell'\leq C\ell \phi_\ell$, where $C$ is independent from $\ell$.
	Then, there exists $\lambda>0$ independent of $\ell$ small such that
	\[
	H_{\phi_\ell} (v,v)\geq \lambda \int \phi_\ell ((\partial_x v)^2+v^2) ,
	\]
	for all $v \in H^1(\R)$ satisfying $\Jap{v}{Q'}=\Jap{v}{\LL(\phi_0)}=0$, and provided $\ell$ is taken small enough.
\end{lem}

The key element of the proof of Theorem \ref{thm1} is the following transfer estimate.

\begin{lem}\label{cor:sech_u1_proof2}
	Let $u_1$ be even and satisfying \eqref{eq:orthogonal_condition2}, $(w_1,w_2)$ be as in \eqref{eq:wi}, and $(z_1,z_2)$ as in \eqref{eq:def_psiB}. Then, for any $B$ large enough, it holds
	\begin{equation}\label{eq:w_sech}
	\begin{aligned}
	\int \sech ( x)  u_1^2
\lesssim & ~{}  B^{-1/2} \left( \| w_1\|_{L^2}^2 +\|\partial_x w_1\|_{L^2}^2 \right)
+B^{1/2} \|z_1\|_{L^2}^2 +\gamma\| \partial_x z_1\|_{L^2}^2.
	\end{aligned}
	\end{equation}
\end{lem}
\begin{proof}
	Set $\frac2B<\ell<\min\{\frac12, \frac14\sqrt{\lambda}\} \leq \frac12$. We note that
	\[
	\int  \sech(x) u_1^2\lesssim \int  \sech^2\left(\ell x\right)u_1^2.
	\]
	Now, we focus on the term on the RHS of the last equation. 
	 Applying Lemma \ref{lem:coerc_weight} for $\phi=\sech^2\left(\ell x \right)$, since $|\phi'|\leq C\ell \phi$.   We obtain
	\begin{equation*}
	\begin{aligned}
	\int  \sech^2\left(\ell x \right)  u_1^2
	\leq&  \int  \sech^2\left(\ell x \right) \left[ u_1^2 +(\partial_x  u_1)^2\right]\\
	\leq& ~{}\frac1\lambda \int  \sech^2\left(\ell x \right) \left[ u_1^2 +(\partial_x  u_1)^2-f'(Q) u_1^2\right].
	\end{aligned}
	\end{equation*}
	Now, integrating by parts
	\begin{equation*}
	\begin{aligned}
		\int  \sech ( \ell x) (\partial_x  u_1)^2
					=& -\int  \sech^2 ( \ell x)u_1\partial_x^2 u_1+\frac12\int ( \sech^2 ( \ell x))'' u_1^2,
	\end{aligned}
	\end{equation*}
	and by
	\[
	|( \sech^2 ( \ell x))''|\leq \ell^2 \sech^2 ( \ell x). 
	\]
	Choosing $\ell $ small enough ($0< \ell \leq \frac{\sqrt{\lambda}}{4}$), we obtain
		\begin{equation*}
		\begin{aligned}
			\int  \sech^2 ( \ell x)u_1^2
			\lesssim & ~{}  
			 \int  \sech^2 \left( \ell x\right)\LL( u_1) u_1.
		\end{aligned}
	\end{equation*}
	Now, using definition of $v_1$, we obtain	
	\begin{equation}\label{eq:ineq_Pv1}
	\begin{aligned}
	\int  \sech^2 \left( \ell x\right)\LL( u_1) u_1
	\lesssim & ~{}  
	\int  \sech^2 \left( \ell x\right) u_1 v_1-\gamma\int  \sech^2 \left( \ell x\right) u_1 \partial_x^2 v_1.
	\end{aligned}
	\end{equation}
	For the first integral in RHS of  \eqref{eq:ineq_Pv1}, using definition of $z_1$ and $w_1$, one can see that
	\begin{equation}\label{eq:CA3_sec}
	\begin{aligned}
			\int \sech^2 \left( \ell x\right)  u_1 v_1
=&\int \chi_A^3 \sech^2 \left( \ell x\right)  u_1 v_1+	\int (1-\chi_A^3) \sech^2 \left( \ell x\right)  u_1 v_1\\
=&\int \chi_A^2 \sech^2 \left( \ell x\right)(\zeta_A \zeta_B)^{-1}  w_1 z_1
+	\int (1-\chi_A^3) \sech^2 \left( \ell x\right)\zeta_A^{-2}  w_1 (\zeta_A v_1)\\
\lesssim&~{} \max_{|x|<2A}\{\sech^2 \left( \ell x\right)(\zeta_A \zeta_B)^{-1}\}  \|w_1\|_{L^2} \|z_1\|_{L^2}
+	\max_{|x|>A}\{ \sech^2 \left( \ell x\right)\zeta_A^{-2} \} \|w_1\|_{L^2} \| \zeta_A v_1\|_{L^2}\\
\lesssim& ~{}\max_{|x|<2A}\{\sech^2 \left( \ell x\right)(\zeta_A \zeta_B)^{-1}\}  \|w_1\|_{L^2} \|z_1\|_{L^2}
+	\gamma^{-1} \max_{|x|>A}\{ \sech^2 \left( \ell x\right)\zeta_A^{-2} \} \|w_1\|_{L^2}^2\\
\lesssim& ~{} \epsilon \|w_1\|_{L^2}^2+\epsilon^{-1} \|z_1\|_{L^2}^2
+	\gamma^{-1} e^{-\frac{A}{4B}} \|w_1\|_{L^2}^2.
	\end{aligned}
	\end{equation}
	Note that the last inequality holds if $2B^{-1}<\ell$.\\
	Now,  for  the second integral on the RHS of  \eqref{eq:ineq_Pv1}, integrating by parts we obtain the following expression 
	\begin{equation}\label{eq:decomposition_part_v1}
	\begin{aligned}
	\int \partial_x \big[ \sech^2 &\left( \ell x\right)  u_1\big] \partial_x v_1\\
	=& ~
	\int \big[ 
	(\sech^2 \left( \ell x\right))'  u_1
	+ \sech^2 \left( \ell x\right)  \partial_x u_1
	\big] \partial_x v_1\\
	=& ~
	\int  (\sech^2 \left( \ell x\right))' \chi_A^2 u_1 \partial_x v_1+\int (1-\chi_A^2) (\sech^2 \left( \ell x\right))'  u_1 \partial_x v_1\\
	&+ \int \sech^2 \left( \ell x\right) \chi_A^2 \partial_x u_1 \partial_x v_1+\int  (1- \chi_A^2)\sech^2 \left( \ell x\right) \partial_x u_1 \partial_x v_1.
	\end{aligned}
	\end{equation}
	Using the following decomposition and by H\"older inequality, we get
	\[
	\begin{aligned}
	\left| \int  (\sech^2 \left( \ell x\right))' u_1 \partial_x v_1\right|
	\lesssim&~{} \left| \int  (\sech^2 \left( \ell x\right))' \chi_A^3 u_1 \partial_x v_1 \right| + \left| \int  (\sech^2 \left( \ell x\right))' (1-\chi_A^3) u_1 \partial_x v_1 \right|\\
	\lesssim&~{} \left| \int  (\sech^2 \left( \ell x\right))' \chi_A^3 u_1 \partial_x v_1 \right| + \left| \int  (\sech^2 \left( \ell x\right))'\zeta_A^{-2} (1-\chi_A^3) w_1  (\zeta_A\partial_x v_1) \right|\\
	\lesssim&~{}  \ell \|\chi_A u_1\|_{L^2} \|  \zeta_B \chi_A^2 \partial_x v_1 \|_{L^2} + \ell \max_{|x|>A}\{\sech^2(\ell x)\zeta_A^{-2}\}
	\|  w_1\|_{L^2}  \|\zeta_A \partial_x v_1\|_{L^2}.
	\end{aligned}
	\]
	Furthermore, by the definition of $z_1$, we can check
	\begin{equation}\label{eq:C^3Zv1_x}
	\begin{aligned}
			\chi_A^2 \zeta_B \partial_x v_1= \chi_A\partial_x z_1-\chi_A\frac{\zeta_B'}{\zeta_B} z_1- \chi_A' z_1;
	\end{aligned}
	\end{equation}
	and by Lemma \ref{lem:v_w} and Remark \ref{rem:chiA_zetaA4}, we obtain
		\begin{equation}\label{eq:coer_1}
	\begin{aligned}
	\left| \int  (\sech^2 \left( \ell x\right))' u_1 \partial_x v_1\right|
	\lesssim&~{}  \ell \| w_1\|_{L^2} (\|  \partial_x z_1 \|_{L^2}
	+B^{-1}\| z_1 \|_{L^2}) \\
	&+ \ell\gamma^{-1} \max_{|x|>A}\{\sech^2(\ell x)\zeta_A^{-2}\}
	(\|  \partial_x w_1\|_{L^2}^2+\|  w_1\|_{L^2}^2).
	\end{aligned}
	\end{equation}
	In similar way, we obtain
	\[
	\begin{aligned}
	\left| \int  \sech^2 \left( \ell x\right) \partial_x u_1 \partial_x v_1\right|
	\lesssim&~{} \left| \int  \sech^2 \left( \ell x\right) \chi_A^3 \partial_x u_1 \partial_x v_1 \right| + \left| \int  \sech^2 \left( \ell x\right) (1-\chi_A^3) \partial_x u_1 \partial_x v_1 \right|\\
	\lesssim&~{}   \|\chi_A \partial_x u_1\|_{L^2} \|  \zeta_B \chi_A^2 \partial_x v_1 \|_{L^2} 
	+  \max_{|x|>A}\{\sech^2(\ell x)\zeta_A^{-2}\}
	\|  \zeta_A \partial_x u_1\|_{L^2}  \|\zeta_A \partial_x v_1\|_{L^2}.
	\end{aligned}
	\]
	By \eqref{eq:C^3Zv1_x}, Lemma \ref{lem:v_w} and Remark \ref{rem:chiA_zetaA4}, we get
		\[
	\begin{aligned}
	\left| \int  \sech^2 \left( \ell x\right) \partial_x u_1 \partial_x v_1\right|
	\lesssim&~{}  
	 \|\zeta_A \partial_x u_1\|_{L^2} (\| \partial_x z_1 \|_{L^2} +\| z_1 \|_{L^2})
	+ \gamma^{-1} \max_{|x|>A}\{\sech^2(\ell x)\zeta_A^{-2}\}
	\|  \zeta_A \partial_x u_1\|_{L^2}^2  .
	\end{aligned}
	\]
	We conclude using \eqref{eq:Zk_u1x} with $K=A$, we have
	\begin{equation}\label{eq:coer_2}
	\begin{aligned}
	\left| \int  \sech^2 \left( \ell x\right) \partial_x u_1 \partial_x v_1\right|
	\lesssim&~{}  
	( \| w_1\|_{L^2}+\|\partial_x w_1\|_{L^2} ) (\| \partial_x z_1 \|_{L^2} +\| z_1 \|_{L^2})\\
	&+ \gamma^{-1} \max_{|x|>A}\{\sech^2(\ell x)\zeta_A^{-2}\}
	(\| w_1\|_{L^2}^2+\|\partial_x w_1\|_{L^2}^2 ) .
	\end{aligned}
	\end{equation}
	Collecting \eqref{eq:CA3_sec}, \eqref{eq:coer_1}, \eqref{eq:coer_2} and by Cauchy-Schwarz inequality, we obtain
	\[
	\begin{aligned}
	\int \sech ( x)  u_1^2
	\lesssim & ~{}  \epsilon \| w_1\|_{L^2}+\frac1\epsilon \| z_1\|_{L^2}\\
	&+\gamma \left(\| w_1\|_{L^2}^2+\|\partial_x w_1\|_{L^2}^2 \right)+\gamma\left(\| \partial_x z_1\|_{L^2}^2+B^{-2}\| z_1\|_{L^2}^2\right)\\
	\lesssim & ~{} \max\{ \epsilon,A^{-1},\gamma\} \| w_1\|_{L^2}
	+\gamma \|\partial_x w_1\|_{L^2}^2 \\
	&+\max \{\epsilon^{-1},B^{-2} \} \|z_1\|_{L^2}^2 
	+\gamma\| \partial_x z_1\|_{L^2}^2.
	\end{aligned}
	\]
	Finally, choosing $\epsilon=B^{-1/2}$, we conclude
	\[
	\begin{aligned}
	\int \sech ( x)  u_1^2
		\lesssim & ~{}  B^{-1/2} \left( \| w_1\|_{L^2} +\|\partial_x w_1\|_{L^2}^2 \right)+B^{1/2} \|z_1\|_{L^2}^2 +\gamma\| \partial_x z_1\|_{L^2}^2.
	\end{aligned}
	\]
	This ends the proof of Lemma \ref{cor:sech_u1_proof2}.
\end{proof}

We will need a third coercivity estimate, related to the function $z_2$ in \eqref{eq:def_psiB}. 

\begin{lem}\label{coercividad_final}
Recall $\mathcal L = -\partial_x^2 + V_0(x)$, with $V_0$ defined in \eqref{eq:LL}. Assume that $\int Q \phi_0 \neq 0$. Then there exists $m_0>0$ fixed such that 
\[
\langle \mathcal L (u),u\rangle \geq m_0 \|u\|_{H^1}^2 - \frac1{m_0} \left|\langle u, \partial_x^{-1}\phi_0\rangle  \right|^2,
\]
for any $u\in H^1(\mathbb R)$ odd.
\end{lem}

\begin{proof}
Since $u$ is odd, one clearly has $\langle \mathcal L (u),u\rangle \geq 0$. Since $\ker \mathcal L = \hbox{span}\{Q'\}$, we only need to check that 
\begin{equation}\label{lindo}
\langle \mathcal L (u),u\rangle \geq m_0 \|u\|_{L^2}^2,
\end{equation}
for any $u\in H^1(\mathbb R)$ odd, and provided $\langle u, \partial_x^{-1}\phi_0\rangle =0$. First of all, it is not difficult to check that for some $m_0>0$,
\[
\langle \mathcal L (u),u\rangle \geq m_0 \left( \|u\|_{L^2}^2- \|Q'\|_{L^2}^{-2}\langle u, Q'\rangle^2\right).
\]
Assume that $\|u\|_{L^2}=1$. The term on the right hand side is zero only if $u$ is parallel to $Q'$, which is not possible since $\int Q \phi_0 \neq 0$. Therefore, after rescaling, \eqref{lindo} is proved.
\end{proof}

\begin{rem}\label{ojala_porfin}
Lemma \ref{coercividad_final} will be used in the following way: from \eqref{eq:orthogonal_condition} we have $\langle u_2, \partial_x^{-1}\phi_0\rangle =0$, and from \eqref{eq:change_variable}, we have
\[
\langle v_2, (1-\gamma \partial_x^2)\partial_x^{-1}\phi_0 \rangle =0.
\]
Using \eqref{eq:def_psiB} and \eqref{eq:wi}, and the exponential decay of $\partial_x^{-1} \phi_0$ we obtain
\[
\begin{aligned}
|\langle z_2  , \partial_x^{-1}\phi_0 \rangle| \leq &~{} |\langle z_2  , (1-\gamma \partial_x^2)\partial_x^{-1}\phi_0 \rangle| +\gamma |\langle z_2  ,  \partial_x\phi_0 \rangle|\\
\lesssim &~{} |\langle v_2\chi_A \zeta_B  , (1-\gamma \partial_x^2)\partial_x^{-1}\phi_0 \rangle| + \gamma \|z_2\|_{L^2}\\
 \lesssim &~{} |\langle v_2  ,(1-\chi_A \zeta_B) (1-\gamma \partial_x^2)\partial_x^{-1}\phi_0 \rangle| + \gamma \|z_2\|_{L^2}\\
  \lesssim &~{} |\langle u_2 \zeta_A  ,  \zeta_A^{-1}(1-\gamma \partial_x^2)^{-1}(1-\chi_A \zeta_B) (1-\gamma \partial_x^2)\partial_x^{-1}\phi_0 \rangle| + \gamma \|z_2\|_{L^2}\\
 \lesssim &~{} e^{-\frac12 B}\|w_2\|_{L^2} + \gamma \|z_2\|_{L^2}.
\end{aligned}
\]
\end{rem}
Finally, we prove that
\begin{lem}\label{no_soy_orto}
$\int Q \phi_0 \neq 0$.
\end{lem}
\begin{proof}
If $\int Q \phi_0 = 0$, from \eqref{eq:phi0} one has $\langle \partial_x^2 \mathcal L Q, Q\rangle \leq 0.$ However
\[
0 \geq \langle \partial_x^2 \mathcal L Q, Q\rangle= -(p-1)\langle  Q^p, Q''\rangle =  -(p-1)\langle  Q^p, Q -Q^p\rangle = -(p-1)\int_\R (Q^{p+1} -Q^{2p}) . 
\]
Finally, from the equation $Q''=Q-Q^p$ and multiplying by $Q^p$ and integrating by parts, we get
\[
-p\int Q^{p-1} Q'^2 = \int Q^{p+1} - \int Q^{2p}. 
\]
Finally, using that $Q'^2 = Q^2 - \frac{2}{p+1} Q^{p+1}$, we get $\int Q^{p+1} = \frac{3p+1}{(p+1)^2} \int Q^{2p}$, and replacing,
\[
0 \geq \langle \partial_x^2 \mathcal L Q, Q\rangle= -(p-1)\int_\R (Q^{p+1} -Q^{2p}) = \frac{p(p-1)^2}{(p+1)^2} \int_\R Q^{2p}>0, 
\]
a contradiction.
\end{proof}

Now we are ready to conclude the proof of Theorem \ref{thm1}. 

\subsection{Proof of Theorem \ref{thm1}}

Recalling that the constants $\gamma$, $C_i$ and $\delta_i>0$ for $i=1,\ldots, 4$ were defined in Propositions \ref{prop:virial_I} \ref{prop:virial_J}, \ref{prop:ineq_dtM}, \ref{prop:virial_N}.

\begin{prop}\label{prop:virial_I+J+M+N}
	There exist $C_5$ and $0<\delta_5 \leq \min\{\delta_1,\delta_2,\delta_3,\delta_4\}$ such that for any $0<\delta\leq \delta_5$, the following holds. Fix $A=\delta^{-1}$, $B=\delta^{-1/19}$ and $\gamma=B^{-4}$. Assume that for all $t\geq0$, \eqref{eq:ineq_hip} holds.
	Let
	\begin{equation*}
	\mathcal{H}=\mathcal{J}+ 16 C_2 B^{-1} \mathcal{I}+B^{-1} \mathcal{M}-16 B^{-5} C_1C_2\mathcal N.
	\end{equation*}
	Then, for all $t\geq 0$,
	\begin{equation}\label{eq:dtH_inequality}
	\begin{aligned}
	\dfrac{d}{dt}\mathcal{H}(t)
			\leq		
	&-C_2 B^{-1} \left(\|w_1\|_{L^2}^2+\|\partial_x w_1\|_{L^2}^2 
	+\| w_2\|_{L^2}^2  \right)+ C_5 |a_1|^3.
	\end{aligned}
	\end{equation}
\end{prop}
\begin{proof}
	First,  from \eqref{eq:dI_w} and  \eqref{eq:w_sech} we obtain for some $C_1>0$ fixed, 
	\[\begin{aligned}
	\dfrac{d}{dt}\I
	\leq&-\dfrac{1}{2}  \left[ \|w_2\|_{L^2}^2 +2\|\partial_x w_1\|_{L^2}^2
	+\frac12 \|w_1^2\|_{L^2} \right]   \\
	&+C_1 a_1^4 
	+  C_1B^{-1/2} \left( \| w_1\|_{L^2}^2 +\|\partial_x w_1\|_{L^2}^2 \right)
	+C_1 B^{1/2} \|z_1\|_{L^2}^2 +C_1\gamma\| \partial_x z_1\|_{L^2}^2.
		\end{aligned}
	\]
Using \eqref{eq:dN_z} and $\gamma=B^{-4}$, we get
	\[\begin{aligned}
	\dfrac{d}{dt}\I	\leq&-\dfrac{1}{4} \left( \|w_2 \|_{L^2}^2+2\|\partial_x w_1 \|_{L^2}^2+\frac12\|w_1 \|_{L^2}^2\right)
	+C_1|a_1|^3 	+B^{-4} C_1	\frac{d}{dt}\mathcal{N}(t) \\
	&
	+B^{1/2} C_1 \|z_1\|_{L^2}^2 + B^{-4} C_1 \left[ \|\partial_x^2 z_2\|_{L^2}^2  + \|\partial_x z_2\|_{L^2}^2 + \|z_2\|_{L^2}^2+\|z_1\|_{L^2}^2\right].
	\end{aligned}
	\]
	Secondly, for $\frac{d}{dt}\mathcal J$, using \eqref{eq:dJ},  Lemma \ref{coercividad_final}, and Remark \ref{ojala_porfin},
	\begin{equation*}
	\begin{aligned}
	\dfrac{d}{dt}\J
		&\leq - \frac14 m_0 \bigg(\|z_1\|_{L^2}^2+ \|\partial_x z_2\|_{L^2}^2+  \|z_2\|_{L^2}^2\bigg) 
	+C_2 B^{-1} \bigg(\| w_1\|_{L^2}^2
	+\| w_2\|_{L^2}^2 \bigg)
	+ C_2 |a_1|^3.
	\end{aligned}
	\end{equation*}
	We conclude that 
	\[
	\begin{aligned}
	\dfrac{d}{dt}\J + 16 C_2 B^{-1} \dfrac{d}{dt}\I \leq &~{} - \frac18 m_0 \bigg(\|z_1\|_{L^2}^2+ \|\partial_x z_2\|_{L^2}^2+  \|z_2\|_{L^2}^2\bigg) \\
	 &~{} - 4 C_2 B^{-1} \left( \|w_2 \|_{L^2}^2+2\|\partial_x w_1 \|_{L^2}^2+\frac12\|w_1 \|_{L^2}^2\right) \\
	 &~{} + 16 B^{-5} C_1 C_2 \|\partial_x^2 z_2\|_{L^2}^2 +2C_2 |a_1|^3 +16 B^{-5} C_1C_2	\frac{d}{dt}\mathcal{N}(t)  .
	\end{aligned}
	\]
	Thirdly,  using \eqref{eq:dM_z} for $\frac{d}{dt} \mathcal M$,
	\begin{equation*}
	\begin{aligned}
	\frac{d}{dt}\M
	\lesssim & ~{}  
	-\dfrac{1}{2}\left( \| \partial_x z_1\|_{L^2}^2+ \| \partial_x^2 z_2\|_{L^2}^2 \right) +C_3 	\|\partial_x z_2 \|_{L^2}^2
	+ \frac12C_3 \| z_2\|_{L^2}^2
	+C_3 B^{-1}\| z_1\|_{L^2}^2\\
	&+C_3  B^{-1} \left( \| \partial_x w_1\|_{L^2}^2+\| w_1\|_{L^2}^2
	+\| w_2\|_{L^2}^2 \right)+
	C_3 |a_1|^3.
	\end{aligned}
	\end{equation*}
	Therefore,
	\[
	\begin{aligned}
	&~{} \dfrac{d}{dt} \left( \J + 16 C_2 B^{-1} \I  +B^{-1} \M -16 B^{-5} C_1C_2	\mathcal{N}(t) \right) \\
	&~{} \quad \leq - \frac1{16} m_0 \bigg(\|z_1\|_{L^2}^2+ \|\partial_x z_2\|_{L^2}^2+  \|z_2\|_{L^2}^2\bigg)-\dfrac{1}{4}B^{-1}\left( \| \partial_x z_1\|_{L^2}^2+ \| \partial_x^2 z_2\|_{L^2}^2 \right)\\
	&~{} \qquad  -  C_2 B^{-1} \left( \|w_2 \|_{L^2}^2+ \|\partial_x w_1 \|_{L^2}^2 + \|w_1 \|_{L^2}^2 \right)+ 3C_2 |a_1|^3 .
	\end{aligned}
	\]
	Setting $\mathcal{H}=\mathcal{J}+ 16 C_2 B^{-1} \mathcal{I}+B^{-1} \mathcal{M}-16 B^{-5} C_1C_2\mathcal N$, and $C_5=3C_2$, we obtain the desired property.
\end{proof}
We define now
\[
\B=b_{+}^2-b_{-}^2,
\]
where $b_{+},b_{-}$ are given in \eqref{eq:b}.

\begin{lem}\label{lem:b}
There exist $C_6$ and $\delta_6>0$, such that for any $0<\delta\leq \delta_6$, the following holds. Assume that for all $t\geq0$ \eqref{eq:ineq_hip} holds. Then, for all $t\geq 0$,
\begin{equation}\label{eq:dtb}
	|\dot{b}_+ -\nu_0 b_+|+|\dot{b}_- +\nu_0 b_-|\leq C_6 \left(b_+^2+b^2_- +
\left\|\sech^{1/2}(x/2) w_1 \right\|_{L^2}^2\right),
\end{equation}
and
\begin{equation}\label{eq:dtb^2}
	\left|\frac{d}{dt}b_+^2 -2\nu_0 b_+^2\right|+\left|\frac{d}{dt}b_-^2 + 2\nu_0 b_-^2\right|\leq C_6 \left(b_+^2+b^2_- +\left\|\sech^{1/2}(x/2)w_1\right\|_{L^2}^2\right)^{3/2}.
\end{equation}
In particular,
\begin{equation}\label{eq:dtB}
	\frac{d}{dt}\B \geq 
	\frac{\nu_0}{2} (a_1^2+ a_2^2)- C_6 \left\|\sech^{1/2}(x/2)  w_1 \right\|_{L^2}^2.
\end{equation}
\end{lem}
\begin{proof}
From \eqref{eq:N_0_bounded} and \eqref{eq:b}, it holds
\[
|N_0|\lesssim b_+^2+b_-^2+\int \sech\left( \frac{x}{2}\right) w_1^2.
\]
From \eqref{eq:motion} we conclude the estimates \eqref{eq:dtb} and \eqref{eq:dtb^2}. 
Finally, \eqref{eq:dtB} follows directly from \eqref{eq:dtb^2} and taking $\delta_6>0$ small enough.
\end{proof}

Combining \eqref{eq:dtH_inequality} and \eqref{eq:dtB}, it holds
\begin{equation}\label{eq:dtB-dtH}
\frac{d}{dt} \left(  \mathcal{B}- 2B\frac{C_6}{C_2} \mathcal{H} \right) 
\geq \frac{\nu_0}{4}(a_1^2+a_2^2)+C_6 \left(\| w_2\|_{L^2}^2  +\|\partial_x w_1\|_{L^2}^2 
+\|w_1\|_{L^2}^2\right) .
\end{equation}
By the choice of $A=\delta^{-1}$, the bound $|\varphi_A|\lesssim A$, \eqref{eq:I} and \eqref{eq:ineq_hip}, we have for all $t\geq 0$
\[
|\I| \lesssim A \| u_1\|_{L^2} \| u_2\|_{L^2}\lesssim \delta.
\]
Analogously, using Lemma \ref{lem:v_u}, we have
\[
\begin{aligned}
|\J| \lesssim &~{} B \| v_1\|_{L^2} \| v_2\|_{L^2}  \lesssim B \gamma^{-1}\| u_1\|_{L^2} \| u_2\|_{L^2}\\
\lesssim &~{} B^{5}\| u_1\|_{L^2} \| u_2\|_{L^2} \lesssim B^5\delta^{2}\lesssim \delta,
\end{aligned}
\]
\[
\begin{aligned}
|\M| \lesssim &~{} B \| \partial_x v_1\|_{L^2} \| \partial_x v_2\|_{L^2}  \lesssim B\gamma^{-3/2} \| u_1\|_{L^2} \| u_2\|_{L^2}\\
\lesssim &~{} B^{7} \| u_1\|_{L^2} \| u_2\|_{L^2}
 \lesssim B^{7}\delta^{2}\lesssim \delta,
 \end{aligned}
\]
and
\[
|\N| \lesssim  \| \partial_x v_1\|_{L^2} \| v_2\|_{L^2}  \lesssim \gamma^{-1} \| u_1\|_{H^1} \| u_2\|_{L^2} \lesssim B^{4} \| u_1\|_{H^1} \| u_2\|_{L^2} \lesssim B^{4}\delta^{2}\lesssim \delta.
\]
Then, we have
\[
|\mathcal{H}|\leq \delta.
\]
Estimate $|\mathcal{B}|\leq \delta^2$ is also clear from \eqref{eq:ineq_hip}.
Therefore, integrating estimates \eqref{eq:dtB-dtH} on $[0,t]$ and passing the limit as $t\to \infty$, we have
\begin{equation*}
\int_{0}^\infty \left[a_1^2+a_2^2+\| w_2\|_{L^2}^2  +\|\partial_x w_1\|_{L^2}^2 
+\|w_1\|_{L^2}^2  \right]dt\lesssim \delta.
\end{equation*}
By Lemma \ref{lem:sech_u_w}
 one can see
\begin{equation}\label{eq:int_a1_a2_norm}
\int_{0}^\infty \left(a_1^2+a_2^2
+ \int (u_1^2+(\partial_x u_1)^2+u_2^2 )\sech(x) \right) dt\leq \delta.
\end{equation}
Using the above equation, we will conclude the proof of Theorem \ref{thm1}.

Let
\begin{equation*}
\begin{gathered}
\mathcal{K}(t)=\int \sech(x) u_1^2+\int \sech(x) (\IOpg \partial_x u_2)^2=:\mathcal{K}_1(t)+\mathcal{K}_2(t).
\end{gathered}
\end{equation*}
For $\mathcal{K}_1$, using \eqref{eq:u_linear} and integrating by parts, we have
\[
\begin{aligned}
\frac{d \mathcal{K}_1}{dt}
=&~{} 2\int \sech(x)(u_1 \dot{u_1})
= 2\int \sech(x)(u_1 \partial_x u_2)
= -2\int (\sech'(x)u_1+\sech(x)\partial_x u_1) u_2.
\end{aligned}
\]
Then,
\[
\begin{aligned}
\left|\frac{d}{dt}\mathcal{K}_1(t)\right|
\leq& \int \sech(x)(u_1^2+(\partial_x u_1)^2+ u_2^2).
\end{aligned}
\]
For $\mathcal{K}_2$, passing to the variables $(v_1,v_2)$ (see \eqref{eq:change_variable}) 
\begin{equation*}
\begin{gathered}
\mathcal{K}_2=\int \sech(x) ( \partial_x v_2)^2,
\end{gathered}
\end{equation*}
 and using \eqref{eq:syst_v}, we get
\[
\frac{d}{dt}{\mathcal{K}}_2
= 2\int \sech(x)\partial_x v_2   \partial_x^2  v_1+2\int \sech(x)\partial_x v_2  \partial_x H(x)
=: K_{21}+K_{22}.
\]
Integrating by parts in $K_{21}$, we have
\[
\begin{aligned}
K_{21}
= -2\int (\sech'(x) \partial_x v_2+\sech(x)  \partial_x^2 v_2 ) \partial_x v_1,
\end{aligned}
\]
besides using Cauchy-Schwarz inequality and Lemma \ref{lem:v_u}, we obtain
\[
\begin{aligned}
|K_{21}|
\lesssim & ~{}\int  \sech (x)( (\partial_x v_2)^2+(\partial_x^2 v_2)^2+ ( \partial_x v_1)^2)\\
\lesssim &~{} \gamma^{-2} \int  \sech (x)( u_2^2+ ( \partial_x u_1)^2).
\end{aligned}
\]
For $K_{22}$, we use Cauchy-Schwartz inequality, Corollary \ref{cor:estimates_Sech_Iop_partial} and a similar computation of  \eqref{eq:LambdaP_N}, then
\[
\begin{aligned}
|K_{22}| \lesssim & \int \sech(x)[ (\IOpg \partial_x u_2)^2 + (\IOpg \partial_x N^{\perp})^2]\\
 \lesssim & \int \sech(x)[\gamma^{-1} u_2^2 + (\IOpg \partial_x N^{\perp})^2]
 \lesssim_{\gamma}   a_1^2+  \int \sech(x)[ u_2^2 + u_1^2].
\end{aligned}
\]
Then, we conclude
\[
\begin{aligned}
\left|\frac{d}{dt}{\mathcal{K}}_2(t)\right|
\lesssim_{\gamma}& ~{} a_1^2+ \int  \sech(x)(u_1^2+(\partial_x u_1)^2+ u_2^2).
\end{aligned}
\]
By \eqref{eq:int_a1_a2_norm}, there exists and increasing sequence $t_n\to \infty$ such that
\[
\lim_{n\to \infty} \left[a_1^2(t_n)+a_2^2(t_n)+\mathcal{K}_1(t_n)+\mathcal{K}_2(t_n)\right]=0.
\]
For $t\geq 0$, integrating on $[t,t_n]$, and passing to the limit as $n\to \infty$, we obtain
\[
\mathcal{K}(t)\lesssim \int_{t}^\infty \left[a_1^2+\int \sech(x)(u_1^2+(\partial_x u_1)^2+ u_2^2)\right]dt.
\]
By \eqref{eq:int_a1_a2_norm}, we deduce 
\[
\lim_{t\to\infty} \mathcal{K}(t)=0.
\]
Finally, by \eqref{eq:motion}  and \eqref{eq:N_0_bounded}, we get
\[
\left|\frac{d}{dt}(a_1^2)\right|+\left|\frac{d}{dt}(a_2^2)\right|\lesssim a_1^2+a_2^2+\int \sech(x) u_1^2 .
\]
In a similar way  as above, integrating on $[t,t_n]$ and taking $n\to\infty$, we conclude
\[
a_1^2(t)+a_2^2(t)\lesssim \int_{t}^{\infty} \left[a_1^2+a_2^2+\int u_1^2 \sech(x)\right]dt,
\]
which proves $\lim_{t\to\infty} |a_1(t)|+|a_2(t)|=0$.
By the decomposition of solution \eqref{eq:decomposition} this implies \eqref{eq:local_stability}. This ends the proof of Theorem \ref{thm1}.

\begin{rem}
We have not being able to describe the asymptotic behavior of $(\partial_x u_1)^2$ and $u_2^2$, due to the fact that we are working in the energy space, and any variation of the virial that involves these terms is not well-defined.
In fact, the regularity considered for the variation of $\mathcal{K}_1$ and $\mathcal{K}_2$ is sharp, in the sense that we do not have a gap where to include terms with higher-order derivatives. For example, 
%
		for
		\begin{equation*}
		\begin{gathered}
		\mathcal{K}_3=\int \sech(x) u_2^2,
		\end{gathered}
		\end{equation*}
		its variation  is 
		\[
		\frac{d}{dt}{\mathcal{K}}_3
		= 2\int \sech(x)u_2 (\partial_x \LL u_1+N^{\perp}).
		\]
		One can see that $\LL u_1\in H^{-1}$ and $u_2\in L^2$. Then, the last estimate may not be well-defined.
\end{rem}

\section{Proof of Theorem \ref{thm2}}\label{proof_TH2}

Now we construct initial data for which Theorem \ref{thm1} remains valid. We follow the ideas in \cite{KMM}, with some particular differences in some estimates.

\subsection{Conservation of Energy}
Using \eqref{eq:energy}, \eqref{eq:decomposition},  \eqref{eq:LL}, and by the orthogonality condition \eqref{eq:orthogonal_condition}, we have
\[
\begin{aligned}
	2\big[E(u,v)&-E(Q,0) \big]
	= \int \big[v^2 +u^2 +(\partial_x u)^2-2F(u)\big]- 2E(Q,0)\\
	=&~{} a_2^2 \int \nu_0^2 (\partial_{x}^{-1}\phi_{0})^2
	+a_1^2\int ((\partial_x \phi_{0})^2+V_{0}(x)\phi^2_0)+ \int f'(Q)(a_1\phi_0+u_1)^2\\
	&+\int ((\partial_x u_1)^2+V_{0}(x)u_1^2+u_{2}^2)
	-2\int \big( F(u)-F(Q)-f(Q)(a_1\phi_0+u_1)  \big) \\
	=&~{} a_2^2  \nu_0^2 \|\partial_{x}^{-1}\phi_{0}\|_{L^2}^2
	+a_1^2 \Jap{\LL(\phi_0)}{\phi_0}
	+\Jap{\LL(u_1)}{u_1} +\|u_2\|_{L^2}^2\\
	&-2\int \Big( F(u)-F(Q)-f(Q)(a_1\phi_0+u_1)-f'(Q)\frac{(a_1\phi_0+u_1)^2}{2} \Big).
\end{aligned}
\]
Using \eqref{eq:phi0}, we get
\[
\Jap{\LL(\phi_0)}{\phi_0}=\Jap{\nu_0^2 \partial_x^{-2} \phi_0}{\phi_0}=-\nu_0^2\Jap{\partial_x^{-1}\phi_0}{\partial_x^{-1}\phi_0}=-\nu_0^2,
\]
and, by \eqref{eq:b}, we obtain the identity
\begin{equation}\label{eq:E(u,v)-E(Q,0)}
\begin{aligned}
	2\big[E(u,v)-E(Q,0) \big]
	=& -4\nu_0^2b_{+} b_{-}
	+\Jap{\LL(u_1)}{u_1} +\|u_2\|_{L^2}^2\\
	&-2\int \left( F(u)-F(Q)-f(Q)(a_1\phi_0+u_1)-f'(Q)\frac{(a_1\phi_0+u_1)^2}{2} \right).
\end{aligned}
\end{equation}
Let $\delta_0$ be defined  by
\[
\delta_0^2= b_+^2(0)+b_-^2(0)+\|u_1(0)\|_{H^1}^2+\| u_2(0)\|_{L^2}^2.
\]
Considering  \eqref{eq:E(u,v)-E(Q,0)} at  $t=0$ follows $|2\big[E(u,v)-E(Q,0) \big]|\lesssim \delta_0^2$. Besides,  by the conservation of energy, estimate \eqref{eq:E(u,v)-E(Q,0)} at some $t>0$ gives
\[
| -4\nu_0^2b_{+} b_{-}
+\Jap{\LL(u_1)}{u_1} +\|u_2\|_{L^2}^2\\
-
O(|b_{+}|^3+|b_{-}|^3+\|u_1\|_{H^1}^3)
|
\lesssim \delta_0^2.
\]
Considering the orthogonality condition $\Jap{u_1}{Q'}=\Jap{u_1}{\LL(\phi_0)}=0$,  the parity of $u_1$, and using the Lemma \ref{lem:coercivity}, it follows that for some $\lambda \in (0,1)$, 
\[
\Jap{\LL(u_1)}{u_1}\geq \lambda \|u_1\|_{H^1}^2.
\]
Due to $\| u_1\|_{H^1}+\|u_2\|_{L^2}+|b_{+}|+|b_{-}|\lesssim \delta_0 $, the following estimate holds
\begin{equation}\label{eq:energy_estimate}
\| u_1\|_{H^1}^2+\| u_2\|_{L^2}^2\lesssim |b_+|^2+|b_-|^2+\delta_0^2.
\end{equation}

\subsection{Construction of the graph}

We will construct initial data that directs to global solutions close to the ground state $Q$. To accomplish this objective, we use the energy estimate \eqref{eq:energy_estimate}, Lemma \ref{lem:b} and a standard contradiction argument. 
 
Let $\boldsymbol{\epsilon}=(\epsilon_1, \epsilon_2)\in \mathcal{A}_0$. Let $\boldsymbol{Z_+}$ be as in \eqref{eq:Y_Z}. Then, the condition $\Jap{\boldsymbol{\epsilon} }{\boldsymbol{Z}_{+}}=0$ rewrites
\[
\Jap{\epsilon_1}{\partial_x^{-2}\phi_0}+\Jap{\epsilon_2}{ \nu_0^{-1}\partial_x^{-1}\phi_0}=0.
\]
Define $b_{-}(0)$ and $(u_1(0),u_2(0))$ such that
\[
b_{-}(0)=-\Jap{\epsilon_1}{\partial_x^{-2}\phi_0}=\Jap{\epsilon_2}{ \nu_0^{-1}\partial_x^{-1}\phi_0},
\]
and
\[
\epsilon_1=b_{-}(0) \phi_0+u_1(0)
, \ \ \ 
\epsilon_2=-b_{-}(0) \nu_0 \partial_x^{-1}\phi_0+u_2(0).
\]
Then, it holds
\[
\Jap{u_1(0)}{\partial_x^{-2}\phi_0}
=\Jap{u_2(0)}{\partial_x^{-1}\phi_0}=0.
\]
From \eqref{eq:M_variedad} and \eqref{eq:A0_variedad}, we observe that the initial condition in Theorem \ref{thm2} holds the following decomposition:
\[
\boldsymbol{\phi}_0=\boldsymbol{\phi}(0)
=(Q,0)+(u_1,u_2)(0)+b_{-}(0)\boldsymbol{Y}_{-}+h(\boldsymbol{\epsilon})\boldsymbol{Y}_{+}.
\]
We will prove that there is a function $h(\boldsymbol{\epsilon})$ such that the corresponding solution $\boldsymbol{\phi}$ is global and satisfies \eqref{eq:condition_close}. We show that at least $h(\boldsymbol{\epsilon})=b_{+}(0)$ satisfies this statement.

Let $\delta_0>0$ small enough and $K>1$ large enough to be chosen. Following the scheme of \cite{KMM}, we introduce the following bootstrap estimates
\begin{gather}
\| u_1\|_{H^1}\leq K^2 \delta_0 \quad \mbox{ and }\quad  \|u_2\|_{L^2}\leq K^2 \delta_0, \label{eq:boots_u1_u2}\\
|b_{-}|\leq K\delta_0,\hfill\label{eq:boots_b-}\\
|b_{+}|\leq K^5\delta_0^2.\label{eq:boots_b+}
\end{gather}
Given any $(u_1(0),u_2(0))$ and $b_{-}(0)$ such that
\begin{equation}\label{eq:bound_initial_data}
\| u_1(0)\|_{H^1}\leq \delta_0,\quad  \| u_2(0)\|_{L^2}\leq \delta_0,\quad  |b_{-}(0)|\leq \delta_0  
\end{equation}
and $b_{+}(0)$ satisfying
\[
|b_{+}(0)|\leq K^5 \delta_0.
\]
Let
\[
T=\sup\{t\geq 0 \quad \mbox{ such that }\quad \eqref{eq:boots_u1_u2},\eqref{eq:boots_b-},\eqref{eq:boots_b+} \quad \mbox{hold on } [0,t]\}.
\]
Since $K>1$ follow that $T$ is well-defined in $[0,+\infty]$. Our aim is to prove that there exists at least a value of $b_+(0)\in [-K^5\delta^2, K^5\delta_0^2]$ such that $T=\infty$. To prove this we argue by contradiction: we assume that for all values of $b_{+}(0)\in [-K^5\delta^2, K^5\delta_0^2]$, one has $T<\infty$.

\medskip

The first step is improve the estimates  \eqref{eq:boots_u1_u2}. By \eqref{eq:boots_u1_u2}, we have
\begin{equation}\label{eq:to_improve}
\|u_1\|_{H^1}^2+\|u_2\|_{L^2}^2\leq 2K^4\delta_0^2
\end{equation}
Otherwise,  using the energy estimates \eqref{eq:energy_estimate}  it holds
\[
\|u_1\|_{H^1}^2+\|u_2\|_{L^2}^2\leq C_8 (K^2\delta_0^2+K^{10}\delta^4+\delta_0^2),
\]
for some constant $C_8>0$.  Thus, using the smallness of $\delta_0$ and largeness of $K$, 
it holds
\begin{equation}\label{eq:const_1}
C_8 \leq \frac14 K^4 , \quad C_8 K^{10}\delta_0^2\leq \frac14 K^4, \quad 1\leq \frac14 K^4,
\end{equation}
and we obtain 
\[
\|u_1\|_{H^1}^2+\|u_2\|_{L^2}^2\leq \frac34 K^4\delta_0^2,
\]
that it is a clear improve of the inequality \eqref{eq:to_improve}.

The second step is control $b_{-}$. Using \eqref{eq:dtb^2}, \eqref{eq:boots_u1_u2}, \eqref{eq:boots_b-} and \eqref{eq:boots_b+}, we have
\[
\left| \frac{d}{dt}(e^{2\nu_0t} b^2_{-})\right|
\leq C_9(K^{15}\delta_0^6+K^6\delta_0^3)e^{2\nu_0 t } ,
\]
for some constant $C_9>0$.  Therefore, by integration on  $[0,t]$ and using \eqref{eq:bound_initial_data}, we obtain
\[
b_{-}^2\leq \frac{C_9}{2\nu_0} (K^{15}\delta_0^6+K^6\delta_0^3)+\delta_0^2.
\]
Under the constraints 
\begin{equation}\label{eq:const_2}
\frac{C_9}{2\nu_0} K^{15}\delta_0^4\leq \frac14 K^4,\quad \frac{C_9}{2\nu_0} K^6\delta_0\leq \frac14 K^4,\quad 1\leq \frac14 K^4,
\end{equation}
we get
\[
b_{-}^2\leq \frac34 K^2 \delta_0^2,
\]
 that is an improvement of \eqref{eq:boots_b-}.
 
By the improved estimates \eqref{eq:boots_u1_u2} and \eqref{eq:boots_b-}, and a continuity argument, we observe that if $T<+\infty$, then $|b_{+}(T)|=K^5\delta_0^2$.

The third step is to  analyze the growth of $b_{+}$.  If $t\in[0,T]$ is such that $|b_{+}(t)|=K^5\delta^2_0$, then follows from \eqref{eq:dtb} that
\[
\begin{aligned}
\dfrac{d}{dt}b_{+}^2
\geq & ~{} 2\nu_0 b_{+}^2
-2C_4 |b_{+}|(b_{+}^2+b_{-}^2+\| u_1\|_{H^1}^2)\\
\geq & ~{}
2\nu_0 b_{+}^2
-2C_4 |b_{+} | (b_{+}^2+K^2\delta_0^2+K^4\delta_0^2)\\
\geq & ~{}
2\nu_0 K^{10} \delta_0^4
-C_{10} ( K^{15} \delta_0^6+K^9\delta_0^4),
\end{aligned}
\]
for some constant $C_{10}>0.$ Under the constraints 
\begin{equation}\label{eq:const_3}
C_{10}  K^{15} \delta_0^2\leq \frac12 \nu_0 K^{10}
,\quad
C_{10} K^9 \leq \frac12\nu_0 K^{10},
\end{equation}
the following  inequality holds
\[
\dfrac{d}{dt} b_{+}^2\geq \nu_0 K^{10}\delta_0^4>0.
\]
By standard arguments, such transversality condition implies that $T$ is the first time for which $|b_{+}(t)|=K^5 \delta_0^2$ and moreover that $T$ is continuous in the variable $b_{+}(0)$.
The image of the continuous map
\[
b_{+}(0)\in [-K^5 \delta_0^2,K^5 \delta_0^2] \longmapsto  b_{+}(T)\in \{-K^5 \delta_0^2,K^5 \delta_0^2\}
\]
is exactly $\{-K^5 \delta_0^2,K^5 \delta_0^2\}$ which is a contradiction. 
We conclude that there exists at least one value of $b_{+}(0)\in (-K^5 \delta_0^2,K^5 \delta_0^2)$ such that $T=\infty$, when constraints in \eqref{eq:const_1}, \eqref{eq:const_2}, \eqref{eq:const_3} are fulfilled.
Finally,  to satisfy the conditions \eqref{eq:const_1}, \eqref{eq:const_2}, \eqref{eq:const_3} it is sufficient first to fix $K>0$ large enough, depending only on $C_8$, $C_9$, $C_{10}$, and then to choose $\delta_0>0$ small enough.

 \subsection{Uniqueness and Lipschitz regularity}
 
To finish the proof of Theorem 2, we will prove the following proposition that implies the uniqueness of the choice of $h(\boldsymbol{\epsilon})=b_{+}(0)$, for a given $\boldsymbol{\epsilon} \in \mathcal{A}_0$, as well the Lipschitz regularity of the graph $\mathcal{M}$ (see \eqref{eq:M_variedad})
 
 \begin{prop}
 There exist $C,\delta>0$ such if  $\boldsymbol{\phi}$ and $\widetilde{\boldsymbol{\phi}}$ are two even-odd solution of \eqref{eq:gGB} satisfying 
 \begin{equation}
 \mbox{for all } t\geq 0, \|\boldsymbol{\phi}(t)-(Q,0) \|_{H^1\times L^2}<\delta,  \quad \|\widetilde{\boldsymbol{\phi}}(t)-(Q,0) \|_{H^1\times L^2}<\delta
 \end{equation}
 then, decomposing
 \begin{equation}
 \boldsymbol{\phi}(0)=(Q,0)+\boldsymbol{\epsilon}+b_{+}(0)\boldsymbol{Y}_{+},
 \quad 
 \widetilde{ \boldsymbol{\phi}}(0)=(Q,0)+\tilde{\boldsymbol{\epsilon}}+\tilde{b_{+}}(0)\boldsymbol{Y}_{+}
 \end{equation}
 with $\Jap{\boldsymbol{\epsilon}}{\boldsymbol{Z}_{+}}=\Jap{\tilde{\boldsymbol{\epsilon}}}{\boldsymbol{Z}_{+}}=0$, it holds
 \begin{equation}\label{eq:lip_b+}
 |b_{+}(0)-\tilde{b}_{+}(0)|\leq C\delta^{1/2} \| \epsilon-\tilde{\epsilon} \|_{H^1\times L^2}.
 \end{equation}
 \end{prop}
 \begin{proof}
 Let  $\boldsymbol{\phi}$ and $\tilde{\boldsymbol{\phi}}$ solutions of \eqref{eq:gGB} likes in the Subsection \ref{sub:2.1}, i.e.,  satisfies the decomposition \eqref{eq:decomposition} and the smallest condition \eqref{eq:ineq_hip}. Then,
 \begin{equation}\label{eq:intial_lips}
 \| u_1\|_{H^1}+ \| \tilde{u}_1\|_{H^1}+
  \| u_2\|_{L^2}+ \| \tilde{u}_2\|_{L^2}+
   | b_{\pm}|+ | \tilde{b}_{\pm}|
   \leq C_0 \delta.
 \end{equation}
Let
 	\begin{equation}\label{eq:lips}
	\begin{aligned}
 	\lips{a}{1},\quad 
 	\lips{a}{2},\quad
 	\lips{b}{+},\quad
 	\lips{b}{-},\quad
 	\lips{u}{1},\\
 	\lips{u}{2},\quad
 	\lips{N}{},\quad
	\ch{N}^{\perp}=N^{\perp}-\widetilde{N}^{\perp},\quad 
 	\lips{N}{0}.\quad \quad 
 	\end{aligned}
	\end{equation}
 Then, by \eqref{eq:motion} and \eqref{eq:u_linear}, $(\check{u}_1,\check{u}_2)$ and $(\check{b}_+,\check{b}_-)$ satisfy the following equations:
 \begin{equation}\label{eq:syst_chu}
 \left\{ \begin{aligned}                                                            
 \dot{\ch{u}}_1&= \partial_x \ch  u_2 \\
 \dot{\ch  u}_2&=\partial_x \LL(\ch  u_1)+\ch N^{\perp}
 \end{aligned}\right.
\quad  \mbox{ and }\quad
 \left\{
 \begin{aligned}
 \dot{\ch{b}}_{+}&= \nu_0 \ch b_{+}+\dfrac{\ch N_0}{2\nu_0} \\
 \dot{\ch b}_{-}&=-\nu_0 \ch b_{-}-\dfrac{\ch  N_0}{2\nu_0}.
 \end{aligned}\right.
 \end{equation}
 Furthermore, let
  \[
 \beta_{+}=~{}\ch{b}_{+}^2,\quad
  \beta_{-}=~{}\ch{b}_{-}^2,\quad
   \beta_{c}=\Jap{\LL \ch u_1}{\ch u_1}+\Jap{\ch u_2}{\ch u_2}.
 \]
Computing the variation of $\beta_c$, we obtain
\[
\begin{aligned}
  \dot{ \beta_{c}}
    =&~{} 2~{} \Jap{\ch N^{\perp}}{\ch u_2}.
\end{aligned}
\] 
 Now, recalling \eqref{eq:Nperp} and \eqref{eq:N_0_bounded}, we get
 \[
 \begin{aligned}
\check{N}
=&\ (Q'-\tilde{a}_1 \partial_x \phi_0-\partial_x \tilde{u}_1)  \bigg[f'(Q)+f''(Q)(a_1\phi_0+u_1)-f'(Q+a_1\phi_0+u_1)\\
& \qquad \qquad \qquad \qquad\qquad  - \Big( f'(Q)+f''(Q)(\tilde{a}_1\phi_0+\tilde{u}_1)-f'(Q+\tilde{a}_1\phi_0+\tilde{u}_1)\Big)\bigg]\\
&+(\ch{a}_1 \partial_x \phi_0+\partial_x \ch{u}_1)(f'(Q)-f'(Q+a_1\phi_0+u_1))\\
&+(\ch{a}_1\phi_0+\ch{u}_1)f''(Q)(\tilde{a}_1 \partial_x \phi_0+\partial_x \tilde{u}_1) .
 \end{aligned}
 \]
  By Taylor expansion, for any $v, \tilde{v}$, it holds 
 \[
 \begin{gathered}
 \big|f'(Q+v)-f'(Q)-f''(Q)v-(f'(Q+\tilde{v})-f'(Q)-f''(Q)\tilde{v})\big|\\
 \lesssim |v-\tilde{v}|(|v|+|\tilde{v}|)(Q^{p-3}+|v|^{p-3}+|\tilde{v}|^{p-3})
 \lesssim |v-\tilde{v}|(|v|+|\tilde{v}|)
 \end{gathered}
 \]
 Then,
  \[
 \begin{aligned}
|\check{N}|
\lesssim &~{} |\ch{a}_1\phi_0+\ch{u}_1|
|Q'-\tilde{a}_1 \partial_x \phi_0-\partial_x \tilde{u}_1|  
\big(
|\tilde{a}_1\phi_0+\tilde{u}_1|+|a_1\phi_0+u_1|
\big)\\%
&+f''(Q)|\ch{a}_1 \partial_x \phi_0+\partial_x \ch{u}_1|
|a_1\phi_0+u_1|
+f''(Q) |\ch{a}_1\phi_0+\ch{u}_1| 
|\tilde{a}_1 \partial_x \phi_0+\partial_x \tilde{u}_1| .
 \end{aligned}
 \]
%
Then, using Sobolev embbeding,  $L^2$- norm of $\ch{N}$ is bounded by
  \begin{equation}\label{eq:bound_ch_N}
 \begin{aligned}
\|\check{N}\|_{L^2}
\lesssim &~{} \|\ch{a}_1\phi_0+\ch{u}_1\|_{L^{\infty}}
\|Q'-\tilde{a}_1 \partial_x \phi_0-\partial_x \tilde{u}_1\|_{L^2}  
\big(
\|\tilde{a}_1\phi_0+\tilde{u}_1\|_{L^{\infty}}+\| a_1\phi_0+u_1\|_{L^{\infty}}
\big)\\%
&+\|\ch{a}_1 \partial_x \phi_0+\partial_x \ch{u}_1\|_{L^2}
\|a_1\phi_0+u_1\|_{L^{\infty}}
+ \|\ch{a}_1\phi_0+\ch{u}_1\|_{L^{\infty}} 
\|\tilde{a}_1 \partial_x \phi_0+\partial_x \tilde{u}_1\|_{L^2} \\
\lesssim &~{} (|\ch{a}_1|+\|\ch{u}_1\|_{H^{1}})
\|Q'-\tilde{a}_1 \partial_x \phi_0-\partial_x \tilde{u}_1\|_{L^2}  
\big(
|\tilde{a}_1|+\|\tilde{u}_1\|_{H^{1}}+| a_1|+\|u_1\|_{H^{1}}
\big)\\%
&+(|\ch{a}_1|+\|\ch{u}_1\|_{H^1})
(|a_1|+\|u_1\|_{H^{1}})
+ (|\ch{a}_1|+\|\ch{u}_1\|_{H^{1}} )
(|\tilde{a}_1| +\|\tilde{u}_1\|_{H^1} )\\
\lesssim &~{} (|\ch{a}_1|+\|\ch{u}_1\|_{H^{1}})
\big[
|a_1|+|\tilde{a}_1|+\|u_1\|_{H^{1}} +\|\tilde{u}_1\|_{H^1} \big].
 \end{aligned}
 \end{equation}
Then, by \eqref{eq:syst_chu}, \eqref{eq:bound_ch_N}, and using
$
|\ch{N}_0|
\lesssim \|\ch{N} \|_{L^2} \|\partial_x^{-1} \phi_0 \|_{L^2},
$
 we get
\begin{equation}\label{eq:b_variation}
|\dot{\beta}_{c}|+|\dot{\beta}_{+}-2\nu_0\beta_{+}|+|\dot{\beta}_{-}+2\nu_0\beta_{-}|
\leq K\delta (\beta_c+\beta_{+}+\beta_{-}) \mbox{ for some } K>0.
\end{equation}
In order to obtain a contradiction, assume that the following holds
\begin{equation}\label{eq:initial_data_lipsc}
0<K\delta  (\beta_{c}(0)+\beta_{+}(0)+\beta_{-}(0))<\frac{\nu_0}{10}\beta_{+}(0).
\end{equation}
Now, we consider the following bootstrap estimate
\begin{equation}\label{eq:boots_lipsc}
K\delta (\beta_{c}+\beta_{+}+\beta_{-})
\leq \nu_0 \beta_{+}.
\end{equation}
and let
\[
T=\sup\left\{ t>0 \mbox{ such that } \eqref{eq:boots_lipsc} \mbox{ holds } \right\}>0.
\]
From \eqref{eq:b_variation} and \eqref{eq:boots_lipsc}, it holds
\begin{equation}\label{eq:boun_bellow_dtbeta+}
\nu_0  \beta_{+}\leq 2\nu_0 \beta_{+}-K\delta (\beta_c+\beta_{+}+\beta_{-})\leq \dot{\beta}_{+}, \  \mbox{for } t \in [0,T].
\end{equation}
Then,  $\beta_{+}$ is positive and increasing function on $[0,T]$.\\
Now, by \eqref{eq:b_variation} and \eqref{eq:boots_lipsc}, we get
\[
\dot{ \beta_{c}}\leq \nu_0 \beta_{+}\leq \dot{\beta}_{+}
\]
integrating and using that $\beta_+(0)>0$ , we obtain
\[
\beta_{c}(t)\leq \beta_{c}(0)+\beta_{+}(t)-\beta_{+}(0)\leq \beta_{c}(0)+\beta_{+}(t).
\]
Furtheremore, by \eqref{eq:initial_data_lipsc} and for $\delta$ small enough, we get
\[
K\delta \beta_{c}(t)
\leq K\delta (\beta_c(0)+\beta_{+}(t))\leq \frac{\nu_0}{10}\beta_{+}(0)+K\delta \beta_{+}(t)
\leq \frac{\nu_0}{5}\beta_{+}(t).
\] 
For $\beta_{-}$, using \eqref{eq:b_variation} and \eqref{eq:boots_lipsc}, we get
\[
\dot{\beta_{-}}\leq -2\nu_0
\beta_{-}+\nu_0\beta_{+},
\]
integrating and using  \eqref{eq:initial_data_lipsc}, we have
\[
\beta_{-}(t)\leq e^{-2\nu_0 t}\beta_{-}(0)+\nu_0 \beta_{+}e^{-2\nu_0 t}\int_{0}^{t} e^{2\nu_0 s}ds
\leq \beta_{-}(0)+\frac{1}{2}\beta_{+}(t).
\]
For $\delta$ small enough, we get
\[
K\delta \beta_{-}(t)
\leq K\delta (\beta_{-}(0)+\beta_{+}(t))
\leq \frac{\nu_0}{10}\beta_{+}(0)+ K\delta \beta_{+}(t)
\leq \frac{\nu_0}{5}\beta_{+}(t).
\]
For $\beta_{+}$, it is clear that holds $K\delta \leq \frac{\nu_0}{5}$ for $\delta$ small enough.\\
We have proved that, for all $t\in [0,T]$,
\[
K\delta (\beta_{c}(t)+\beta_{+}(t)+\beta_{-}(t))\leq \frac{3}{5}\nu_0 \beta_{+}(t).
\]
By a continuity argument, we get that $T=\infty$. However, by the exponential growth \eqref{eq:boun_bellow_dtbeta+} and $\beta_{+}(0)>0$, we obtain a contradiction with \eqref{eq:intial_lips} on $|b_{+}|$.

 Since it holds
\[
\boldsymbol{\epsilon}=\boldsymbol{u}(0)+ b_{-}(0) \boldsymbol{Y}_{-},\qquad 
\tilde{\boldsymbol{\epsilon}}=\tilde{\boldsymbol{u}}(0)+ \tilde{b}_{-}(0) \boldsymbol{Y}_{-},
\]
with $\Jap{\boldsymbol{u}(0)}{\boldsymbol{Y}_{-}}=\Jap{\tilde{\boldsymbol{u}}(0)}{\boldsymbol{Y}_{-}}=0$, and estimates \eqref{eq:initial_data_lipsc} is contradicted, we have proved \eqref{eq:lip_b+}.
 \end{proof}

\appendix

\section{Linear spectral theory for $-\partial_x^2\mathcal L$}\label{A}

In this section we describe the spectral properties of the operator $-\partial_x^2\mathcal L$, where $\mathcal L$ is introduced in \eqref{eq:LL}. Notice that this last operator has been widely studied (see \cite{Martel-Merle1,Martel-Merle2}). For the study of the operator $-\partial_x^2\mathcal L$ we shall start with the following result.

\begin{lem}\label{lem:LL}
Let $p>1$. The operator $\LL$ defined in \eqref{eq:LL} satisfies the following properties.
\begin{enumerate}
	\item \label{CS_LL} The continuum spectrum of $\LL$ is $[1,\infty)$.
	\item The kernel of $\LL$ is only spanned by the function $Q'$.
	\item \label{ker_LL} The generalized kernel of $\LL$ is given by $\hbox{span}\left\{Q'(x), Q'(x)\displaystyle \int_{\epsilon}^{x} (Q'(r))^{-2}dr \right\}$, for any $x\geq \epsilon >0$ or $x\leq \epsilon<0$.
\end{enumerate}
\end{lem}
In what follows, and with a slight abuse of notation, we will write 
\[
 \int_0^{x} (Q'(r))^{-2}dr
\]
instead of $ \int_{\epsilon}^{x} (Q'(r))^{-2}dr$; but it is understood that the zero limit of integration corresponds to any $\varepsilon$ sufficiently close to zero.

\medskip
 
An important remark is the following:
\begin{rem}
	Note that
	\begin{equation}
	\LL(fg)=g\LL(f)-2f'g'-fg''. \label{eq:LL(fg)}
	\end{equation}
	This property will be useful in the following computations.
\end{rem}
%

Now, we study the properties of the operator $-\PL$.

\begin{rem}\label{DD}
	A direct analysis shows that the null space of $\partial_x^{2}\LL_0=\partial^4_x-\partial_{x}^2$ is spanned  by functions of the type
	\[
	 e^{ x},~e^{-x}, ~1,~ x, \ \ \ \mbox{as } \ \ x\to \infty.
	\]
	Note that this set is linearly independent and among these four functions there is only one $L^2$ integrable in the semi-infinite line $[0,\infty)$. Therefore, since $\PL$ is a compact perturbation of the scalar operator $\partial_x^{2}\LL_0$, the null space of $\PL|_{H^4(\R)}$ is spanned by at most one $L^2$-function.
\end{rem}

\begin{lem}\label{lem:PL}
Let $p>1$. The operators $-\PL$ satisfy the following properties.
\begin{enumerate}
\item \label{CS_PL} The continuum spectrum of $-\PL$ is $[0,\infty)$.
\item \label{ker_PL} The generalized kernel  of $-\PL$ is spanned by
\begin{equation}\label{PL_kernel}
\begin{aligned}
&~{} \hbox{span} \left\{\left[-\dfrac{2}{p-1}+-\dfrac{p+1}{p-1}Q'(x) \displaystyle\int_{0}^x Q^{-1}(r)dr \right],  \right. \\
&~{} \quad \qquad \left. Q'(x)\displaystyle
\int_{0}^x\dfrac{\left(sQ(s)-\int_{-\infty}^{s}Q(y)dy\right)}{(Q')^{2}(s)}ds, ~ Q'(x)~ , ~ Q'(x)\displaystyle \int_{0}^{x} (Q'(r))^{-2}dr\right\}.
\end{aligned}
\end{equation}
\end{enumerate}
\end{lem}

\begin{proof}
The proof of \eqref{CS_PL} follows directly from the form of the operator.\\

\noindent
Proof of \eqref{ker_PL}. Clearly 
\[
u_1(x):= Q'(x) , \qquad u_2(x):= Q'(x)\displaystyle \int_{0}^{x} (Q'(r))^{-2}dr,
\]
are solutions to $-\PL (u)=0$. Notice that if $-\PL(u)=0$ is equivalent to $\LL(u)=ax+b$ with $a,b\in \R$. Then we should solve this equation.
First, we consider the case $a=0$. Without loss of generality, we consider $b=1$. One has  $\LL(1)=1-pQ^{p-1}$.
Computing,
	\[
		\begin{aligned}
		 \LL\left(Q'\int_{0}^x Q^n\right)
		 &=\LL(Q')\int_{0}^x Q^n-Q^n(2Q'')-nQ^{n-1}(Q')^2\\
		 &=-2Q^n(Q-Q^p)-nQ^{n-1}\left(Q^2-\dfrac{2}{p+1}Q^{p+1}\right)\\
		 &=Q^{n+1}(-2-n)+\left( \dfrac{2n+2p+2}{p+1}\right)Q^{p+n}.
		\end{aligned}
	\]
If $n=-1$, we have
	\[
	\LL\left(Q'\int_{0}^x Q^{-1}\right)=-1+\dfrac{2p}{p+1}Q^{p-1}.
	\]
Set $u_3(x)=-\dfrac{2}{p-1}\left[1+\dfrac{p+1}{2}Q'(x) \displaystyle\int_{0}^x Q^{-1}(r)dr \right]$.
We observe that $\LL\left(u_3(x)\right)=1.$ Therefore, up to the generalized kernel of $\LL$, $u_3$ solves the equation $\LL(u_3)=1$.

\medskip

Now, without loss of generality, we consider $a=1$ and $b=0$, then we must solve $\LL(u_4)=x$. Using the method of reduction of order with an unknown function $\psi$, consider $u_4=Q'\psi$. Using \eqref{eq:LL(fg)}, we have
\begin{equation*}
\LL(Q' \psi)=-2Q''\psi'-Q'\psi''=x.
\end{equation*}
We obtain that the solution of this equation is
\[\begin{aligned}
u_4(x)
&=Q'(x)\displaystyle
\int_{0}^x(Q')^{-2}(s)\left(\int _{0}^sQ'(y)y dy\right)ds-{\color{red}\left(\int_{-\infty}^{0}Q(y)dy\right)Q'(x)\int_{0}^x (Q')^{-2}(s)ds}\\
&=Q'(x)\displaystyle
\int_{0}^x(Q')^{-2}(s)\left(\int _{0}^s yQ'(y) dy-\int_{-\infty}^{0}Q(y)dy\right)ds\\
&=Q'(x)\displaystyle
\int_{0}^x(Q')^{-2}(s)\left(sQ(s)-\int_{-\infty}^{s}Q(y)dy\right)ds.
\end{aligned}\]
We finally conclude that the fundamental set of solutions for $\PL(u)=0$ is given by
\[
\left\{u_1(x), u_2(x),u_{3}(x),u_4(x) \right\}.
\]
This ends the proof.
\end{proof}

\begin{cor}
	There is, up to constant, only one   solution of $-\PL(u)=0$ in $L^2(\R)$.
\end{cor}


Now, we focus on describing the eigenfunctions and negative eigenvalues of operator $-\PL$. This analysis will be the main ingredient to describe the stability of the soliton.  Our first result establishes the parity of eigenfunctions associated to nonzero eigenvalues.

\begin{lem}\label{lem:WD}
	If $\phi_0\in H^4(\R)$ is an eigenfunction associated to an eigenvalue $\lambda_0\neq 0$ of the operator $-\PL$, then $\partial^{-1}_{x}\phi_0\in H^3$ and and $\partial_x^{-2} \phi_0 \in H^2$, i.e., are well-defined. Furthermore, if $\phi_0$ is an even function then $\partial_x^{-1}\phi_0$ is an odd function and
	\[
	\int_{0}^{\infty} \phi_{0}(y)dy=0.
	\]

\end{lem}
\begin{proof}
	We have
	\[
	-\PL \phi_0=\lambda_0\phi_0,\ \mbox{ with }\ \lambda_0\neq 0,
	\]
	this is equivalent to
	\begin{equation}
	    \partial_{x}^4\phi_0-\partial_x^2\phi_0+\partial^2_x(pQ^{p-1}\phi_0)=\lambda_0 \phi_0. \label{eq:LP=lambda}
	\end{equation}
	Applying Fourier transform, we have
	\[
	\xi^4 \widehat{\phi_0}+\xi^2 \widehat{\phi_0}-\xi^2 p(\widehat{Q^{p-1} \phi_0})=\lambda_0 \widehat{\phi_0}.
	\]
    From this identity, and the fact that $\phi_0\in H^4(\R)$, we observe that
    \[
    \begin{aligned}
    \lim_{\xi\to 0} \widehat{\phi_0}(\xi)
      &=\lambda_{0}^{-1} \lim_{\xi\to0} \left(\xi^4 \widehat{\phi_0}+\xi^2 \widehat{\phi_0}-\xi^2 \widehat{pQ^{p-1} \phi_0}\right)=0.
    \end{aligned}
    \]
    Also
    \[
    \begin{aligned}
    \lim_{\xi\to 0} \xi^{-1} \widehat{\phi_0}(\xi)
      &=\lambda_{0}^{-1}\lim_{\xi\to0} \left(\xi^3 \widehat{\phi_0}+\xi \widehat{\phi_0}-\xi \widehat{pQ^{p-1} \phi_0}\right)=0,\\
    \lim_{\xi\to 0} \xi^{-2} \widehat{\phi_0}(\xi)
      &=\lambda_{0}^{-1} \lim_{\xi\to0} \left(\xi^2 \widehat{\phi_0}+ \widehat{\phi_0}- \widehat{pQ^{p-1} \phi_0}\right)
      =-p\lambda_{0}^{-1}\widehat{Q^{p-1}\phi_0}(0).
    \end{aligned}
    \]
	Then, we obtain that
	\[
	\int \phi_0(x) =0,\ \ \ \int \int_{-\infty}^{x}\phi_0(s)ds =0.
	\]
	Also, we know that $\widehat{Q^{p-1}\phi_0}$ is well defined (the Fourier transform is an homeomorphism from $L^2$ into $L^2$). Then $\partial_x^{-1} \phi_0$ and $\partial_x^{-2} \phi_0$ are well-defined, and exponentially decreasing, provided $\phi_0$ and its derivatives are also exponentially decreasing.\\

	Now, suppose that $\phi_0$ is an \textbf{even} function. Integrating between 0 and $x$ in \eqref{eq:LP=lambda}, we obtain
	\[
	(\partial_{x}^3\phi_0-\partial_x\phi_0+\partial_x(pQ^{p-1}\phi_0))(x)-
	[\partial_{x}^3\phi_0-\partial_x\phi_0+\partial_x(pQ^{p-1}\phi_0)]\vert_{x=0}
	=\lambda_0 \int_{0}^x\phi_0.
	\]
	Since  $Q^{p-1}$ is an even function and $\partial_x^3\phi_0$, $\partial_x\phi_0$ and $\partial_x(Q^{p-1}\phi_0)$ are odd functions, satisfying $\partial_x^3\phi_0(0)=\partial_x\phi_0 (0)=\partial_x(Q^{p-1}\phi_0)(0)=0$, we conclude
		\[
		\partial_{x}\LL(\phi_0)(x)=(\partial_{x}^3\phi_0-\partial_x\phi_0+\partial_x(pQ^{p-1}\phi_0))(x)=\lambda_0 \int_{0}^x\phi_0(y)dy.
		\]
	Now, given that $\phi_0\in H^4(\R)$, one has $\partial_x^3\phi_0(x), \partial_x\phi_0(x), \partial_x(Q^{p-1}\phi_0)(x)\to 0$ as $x\to \pm \infty$. We conclude
	\[
	\int_{0}^{x} \phi_0(y)dy=-\int_{0}^{-x}\phi_{0}(y)dy\ \  \mbox{  and  }\ \  \int_{0}^{\infty} \phi_0(y)dy=0.
	\]
	This proves the oddness of $\partial_x^{-1}\phi_0$ and concludes the proof.
\end{proof}
We  observe that $-\PL$ is not a self-adjoint operator. In fact, if $\varphi, \psi \in H^4(\R)$, 
\[
\left\langle -\PL\varphi,\psi\right\rangle 
=\Jap{\varphi}{-\PL(\psi)}+\Jap{\varphi}{f'(Q)\partial_x^{2}\psi-\partial_x^{2}(f'(Q)\psi)},
\]
since the operators $\partial^2_x$ and $\LL$ do not commute. For this reason, we need to consider this operator in an appropriate sense. A way to face this problem is to consider the following result.

\begin{lem}
The operator $-\PL$ has only real eigenvalues.
\end{lem}
\begin{proof}
Given $\varphi_0\in H^{4}(\R)$ eigenfunction of the operator $-\PL$ with eigenvalue $\lambda_0\in \C$, we consider $\varphi_0=\partial_x \psi_0$ or $\psi_0=\partial_x^{-1}\varphi_0$. We know that this function is well defined by Lemma \ref{lem:WD}. Now, we have
	\[
	-\PL(\partial_{x}\psi_0)=-\PL(\varphi_0)=\lambda_0 \varphi_0=\lambda_0 \partial_x \psi_0.
	\]
Integrating, we obtain
	\[
	-\partial_x\LL(\partial_{x}\psi_0)=\lambda_0 \psi_0.
	\]
We can easily check that the operator $-\partial_x\LL\partial_x$ is self-adjoint with eigenvalue $\lambda_0$ and eigenfunction $\psi_0$. We conclude that $\lambda_0$ is real, hence the eigenvalues of $-\PL$ are real.
\end{proof}

Therefore, the operator $-\PL$ has a similar structure of a self-adjoint operator. This fact allows to follow the strategy of Greenberg and Maddocks-Sachs \cite{Greenberg,MS} for counting the negatives eigenvalues of this operator. \\

The most important property about $-\PL$ is that it possesses only one negative eigenvalue.

\begin{thm}
The operator $-\PL$ has a unique negative eigenvalue $-\nu_0^2<0 $ of multiplicity one. The associated eigenfunction $\phi_0$ satisfies the exponential decay in \eqref{eq:phi0}, along with its derivatives.
\end{thm}

This is just a consequence of the fact that the only solution of $-\PL(u)=0$ converging to zero at $-\infty$ is $Q'(x)$, see Claim \ref{END}. This function has a unique zero. The exponential decay is just consequence of Remark \ref{DD}.

\begin{cor}\label{cor:even_eigenfunction}
    Given $\phi_0$ eigenfunction associated to  the unique negative eigenvalue $-\nu_0^2$, then $\phi_0$ is an even function and $\partial_x^{-1}\phi_{0}$ is an odd function.
\end{cor}
\begin{proof}
Consider the function $\psi(x)=\phi_0(-x)$, we have
\[
-\PL(\psi)=\partial_{x}^4(\psi)-\partial_{x}^2 \psi +\partial_{x}^2(pQ^{p-1}\psi).
\]
Notice that $Q^{p-1}, \partial_{x}^2(Q^{p-1})$ are even functions and $\partial_{x}(Q^{p-1})$ is an odd function, also \smallskip
\begin{equation*}
    \begin{aligned}
        \partial_{x}^2(pQ^{p-1}\psi)(x)
        = \partial_{x}^2(pQ^{p-1}(x)\phi_0(-x))
        &= \partial_{x}^2(pQ^{p-1}\phi_0)(-x).
    \end{aligned}
    \smallskip
\end{equation*}
Then, we observe that
\begin{equation*}
    -\PL(\psi)
    =-\PL(\phi_0)(-x)=-\nu_0^2 \phi_{0}(-x)=-\nu_0^2 \psi(x).
\end{equation*}
Finally, since $\lambda_0$ is the unique negative eigenvalue of multiplicity one, we conclude that $\phi_0(x)=\psi(x)=\phi_0(-x)$, i.e., $\phi_0$ is an even function. Finally, by Lemma \ref{lem:WD} we  know $\partial^{-1}_x\phi_0$ is an odd function.
\end{proof}

\subsection{Asymptotic behavior of fundamental solutions of $-\PL(u)=0$}\label{Ap:limits}

The following computations are direct, but we include them by the sake of completeness. They are just simple applications of L'{}H\^opital's rule. 

\begin{claim}\label{END}
The functions $u_1, u_2, u_3$ and $u_4$ found in Lemma \ref{lem:LL} and \ref{lem:PL} satisfy
\[
\lim_{x\to-\infty} u_1 (x) =0, \quad  \lim_{x\to-\infty} u_2(x) = +\infty, \quad  \lim_{x\to-\infty} u_3(x) =1, \quad \lim_{x\to-\infty} u_4(x) =-\infty.
\]
\end{claim}
\begin{proof} One has
\begin{enumerate}
	\item 
		\begin{eqnarray*}
  		 \lim_{x\to-\infty} u_1 (x)
  		 &=& \lim_{x\to-\infty} Q'(x)=0.
		\end{eqnarray*}

	\item Second,
		\begin{eqnarray*}
		 \lim_{x\to-\infty} u_2(x)
		 &=& \lim_{x\to-\infty} Q'(x)\int_{0}^x (Q'(r))^{-2}dr = \lim_{x\to-\infty} \dfrac{\int_{0}^x (Q'(r))^{-2}dr}{(Q'(x))^{-1}}\\
		 &=& \lim_{x\to-\infty} \dfrac{ (Q'(x))^{-2}}{-(Q'(x))^{-2}Q''(x)}= \lim_{x\to-\infty} \dfrac{1}{-Q''(x)}= \lim_{x\to-\infty} \dfrac{1}{Q(x)(Q^{p-1}(x)-1)}
		 = +\infty.
		\end{eqnarray*}

	\item Third,
	\[
		\begin{aligned}
		 & \lim_{x\to-\infty} u_3(x) = \lim_{x\to-\infty} -\dfrac{2}{p-1}\left[1+\dfrac{p+1}{2}Q' \displaystyle\int_{0}^x Q^{-1}(r)dr \right]\\
		 &\quad = -\dfrac{2}{p-1}- \dfrac{p+1}{p-1} \lim_{x\to -\infty} Q'(x) \displaystyle\int_{0}^x Q^{-1}(r)dr = -\dfrac{2}{p-1}- \dfrac{p+1}{p-1} \lim_{x\to -\infty}  \dfrac{\displaystyle\int_{0}^x Q^{-1}(r)dr}{(Q'(x))^{-1}}\\
		 &\quad = -\dfrac{2}{p-1}- \dfrac{p+1}{p-1} \lim_{x\to -\infty}  \dfrac{ Q^{-1}(x)}{-(Q'(x))^{-2}Q''(x)}= -\dfrac{2}{p-1}+ \dfrac{p+1}{p-1} \lim_{x\to -\infty}  \dfrac{ Q^{-1}(x)(Q^2-\frac{2}{p+1}Q^{p+1})}{Q-Q^{p}}\\
		 &\quad = -\dfrac{2}{p-1}+ \dfrac{p+1}{p-1} \lim_{x\to -\infty}  \dfrac{ (1-\frac{2}{p+1}Q^{p-1})}{1-Q^{p-1}}
		 = 1.
		\end{aligned}
		\]

	\item Finally,
		\begin{eqnarray*}
		 & & \lim_{x\to-\infty} u_4(x)= \lim_{x\to-\infty}  Q'(x)\displaystyle
\int_{0}^x(Q')^{-2}\left(sQ(s)-\int_{-\infty}^{s}Q\right)ds\\
		&=&  \lim_{x\to-\infty}  \dfrac{\displaystyle
		\int_{0}^x(Q')^{-2}\left(sQ(s)-\int_{-\infty}^{s}Q\right)ds}{(Q'(x))^{-1}}= \lim_{x\to-\infty}  \dfrac{\displaystyle(Q'(x))^{-2}\left(xQ(x)-\int_{-\infty}^{x}Q\right)}{-(Q'(x))^{-2}Q''(x)}\\
		&=&\lim_{x\to-\infty}  \dfrac{\displaystyle\left(xQ(x)-\int_{-\infty}^{x}Q\right)}{-Q''(x)} =  \lim_{x\to-\infty}  \dfrac{\displaystyle\left(xQ(x)-\int_{-\infty}^{x}Q\right)}{Q(x)(Q^{p-1}(x)-1)}= \lim_{x\to-\infty}  \dfrac{\left(xQ'(x)+Q(x)-Q(x)\right)}{Q'(x)(pQ^{p-1}(x)-1)}\\
		&=& \lim_{x\to-\infty}  \dfrac{x}{pQ^{p-1}(x)-1}=-\infty.
		\end{eqnarray*}
\end{enumerate}
\end{proof}

	\section{Proof of Claims \ref{claim:P_CZ_vi_x} and \ref{claim:R_CZ_vi_xx}}
	\subsection{Relation between $\partial_x z_i$  and $\partial_x v_i$} \label{ap:vi_x} We prove Claim \ref{claim:P_CZ_vi_x}. First, recall that $z_i=\chi_A \zeta_B v_i$ and
	\begin{equation}
	\begin{aligned}
	\partial_x z_i = (\chi_A \zeta_B)' v_i+\chi_A \zeta_B \partial_x v_i.
	\end{aligned}
	\end{equation}
	Then
	\begin{equation}\label{eq:AA}
	\begin{aligned}
	(\partial_x z_i)^2 &= ((\chi_A \zeta_B)' v_i)^2+2(\chi_A \zeta_B)'\chi_A \zeta_B v_i \partial_x v_i+(\chi_A \zeta_B \partial_x v_i)^2.
	\end{aligned}
	\end{equation}
	For a function $P(x)\in C^{1}(\R)$, we consider
	\[
		\int P(x)\chi_A^2\zeta_B^2(\partial_x v_i)^2 .
	\]
	Using \eqref{eq:AA}, we obtain
	\begin{equation}\label{eq:R_chiBzetaB vi_x}
	\begin{aligned}
	\int P(x)\chi_A^2\zeta_B^2(\partial_x v_i)^2 
	=& 
	\int P(x) (\partial_x z_i)^2 
	- \int P(x) [(\chi_A \zeta_B)']^2 v_i^2
	-\frac12 \int P(x) ((\chi_A \zeta_B)^2)' \partial_x (v_i^2)\\
	=& 
	\int P(x) (\partial_x z_i)^2 
	- \int P(x) [(\chi_A \zeta_B)']^2 v_i^2
	+\frac12 \int [P(x) ((\chi_A \zeta_B)^2)']' v_i^2.
	\end{aligned}
	\end{equation}
	Now
	\[
	\begin{aligned}
	P(x) [(\chi_A \zeta_B)']^2
		=& 	P(x)\zeta_B^2 \bigg[(\chi_A' )^2+(\chi_A^2)'\frac{\zeta_B'}{\zeta_B}\bigg]+P(x)(\chi_A \zeta_B)^2\left( \frac{\zeta_B'}{\zeta_B}\right)^2.
	\end{aligned}
	\]
	Then, we have
	\[
	\begin{aligned}
	\int P(x) [(\chi_A \zeta_B)']^2 v_i^2
	=\int P(x)\left(\frac{\zeta_B'}{\zeta_B}\right)^2 z_i^2 
	+\int P(x)  \bigg[(\chi_A' )^2+(\chi_A^2)' \frac{\zeta_B'}{\zeta_B}\bigg] \zeta_B^2 v_i^2.
	\end{aligned}
	\]
	As for the third integral in the RHS of \eqref{eq:R_chiBzetaB vi_x}, we have
	\[
	\begin{aligned}
	\ [P(x) (\chi_A^2 \zeta_B^2)']'
	=&~{}P'(x) (\chi_A^2 \zeta_B^2)'+P(x) (\chi_A^2 \zeta_B^2)''\\
	=&~{}P'(x)\zeta_B^2\left [(\chi_A^2)'+2\chi_A^2\left( \frac{\zeta_B'}{\zeta_B}\right)\right]\\
	&+P(x)  \zeta_B^2 \left[ (\chi_A^2)''+4(\chi_A^2)' \frac{\zeta_B'}{\zeta_B}+2\chi_A^2
	\left[ \left(\frac{\zeta_B'}{\zeta_B}\right)^2+\frac{\zeta_B''}{\zeta_B}\right]\right]\\
		=&~{}  
	 2\chi_A^2\zeta_B^2 \left[ P'(x) \frac{\zeta_B'}{\zeta_B}+P(x) 
	\left[ \left(\frac{\zeta_B'}{\zeta_B}\right)^2+\frac{\zeta_B''}{\zeta_B}\right]\right]\\
	&~{} +P(x)  \zeta_B^2 \left[ (\chi_A^2)''+4(\chi_A^2)' \frac{\zeta_B'}{\zeta_B}\right]
	+P'(x)\zeta_B^2(\chi_A^2)'.
	\end{aligned}
	\]
	Then, we obtain
	\[
	\begin{aligned}
	\int [P(x) ((\chi_A \zeta_B)^2)']' v_i^2
	=&~{} 2\int   \left[ P'(x) \frac{\zeta_B'}{\zeta_B}+P(x) 
	\left[ \left(\frac{\zeta_B'}{\zeta_B}\right)^2+\frac{\zeta_B''}{\zeta_B}\right]\right]z_i^2\\
	&+\int P(x)   \left[ (\chi_A^2)''+4(\chi_A^2)' \frac{\zeta_B'}{\zeta_B}\right] \zeta_B^2 v_i^2
	+\int P'(x)(\chi_A^2)' \zeta_B^2 v_i^2.
	\end{aligned}
	\]
	We conclude in \eqref{eq:R_chiBzetaB vi_x}:
		\begin{equation}\label{eq:P_CZ_B_vix_final}
	\begin{aligned}
	\int P(x)\chi_A^2\zeta_B^2(\partial_x v_i)^2 =&~{} 
	\int P(x) (\partial_x z_i)^2 
	- \int P(x) [(\chi_A \zeta_B)']^2 v_i^2
	+\frac12 \int [P(x) ((\chi_A \zeta_B)^2)']' v_i^2\\
	=&~{} \int P(x) (\partial_x z_i)^2 -\int P(x)\left(\frac{\zeta_B'}{\zeta_B}\right)^2 z_i^2 
	- \int P(x)  \bigg[(\chi_A' )^2+(\chi_A^2)' \frac{\zeta_B'}{\zeta_B}\bigg] \zeta_B^2 v_i^2  \\
	&~{}  + \int   \left[ P'(x) \frac{\zeta_B'}{\zeta_B}+P(x) 
	\left[ \left(\frac{\zeta_B'}{\zeta_B}\right)^2+\frac{\zeta_B''}{\zeta_B}\right]\right]z_i^2\\
	&+ \frac12 \int P(x)   \left[ (\chi_A^2)''+4(\chi_A^2)' \frac{\zeta_B'}{\zeta_B}\right] \zeta_B^2 v_i^2
	+ \frac12 \int P'(x)(\chi_A^2)' \zeta_B^2 v_i^2\\
		=&~{}	\int P(x) (\partial_x z_i)^2 +
		\int   \left[ P'(x) \frac{\zeta_B'}{\zeta_B}+P(x) 
		\frac{\zeta_B''}{\zeta_B}\right] z_i^2\\
		&~{} +\int \mathcal{E}_1(P(x),x)\zeta_B^2 v_i^2,
	\end{aligned}
	\end{equation}
	where
	\begin{equation}\label{eq:E1}
	\mathcal{E}_1(P(x),x)
	= P(x)   \left[ \chi_A''\chi_A+(\chi_A^2)' \frac{\zeta_B'}{\zeta_B}\right] 
	+ \frac12 P'(x)(\chi_A^2)' .
	\end{equation}
	Finally, \eqref{cota_final} follows directly from the definition of $\mathcal{E}_1(P(x),x)$ and Remark \ref{acotados} replacing $B$ by $A$. This ends the proof of Claim \ref{claim:P_CZ_vi_x}.	
	
	\subsection{Relation between $\partial_x^2 z_i$  and $\partial_x^2 v_i$} \label{ap:vi_xx}
	
	Now we prove  Claim \ref{claim:R_CZ_vi_xx}. The following relation is obtained from $z_i$ in \eqref{eq:def_psiB}:
					\[
					\begin{aligned}
					\partial_x^2 z_i =&~{} (\chi_A \zeta_B)'' v_i + 2(\chi_A \zeta_B)' \partial_x v_i + \chi_A \zeta_B \partial_x^2v_i,\\
					(\partial_x^2 z_i)^2=&~{}
					[(\chi_A \zeta_B)'' v_i]^2 + 4[(\chi_A \zeta_B)' \partial_x v_i ]^2+ [\chi_A \zeta_B \partial_x^2v_i]^2\\
					&+4 (\chi_A \zeta_B)'' v_i (\chi_A \zeta_B)' \partial_x v_i +4(\chi_A \zeta_B)' \partial_x v_i \chi_A \zeta_B \partial_x^2v_i+2 (\chi_A \zeta_B)'' v_i  \chi_A \zeta_B \partial_x^2v_i\\
					= &~{} [(\chi_A \zeta_B)'' v_i]^2 + 4[(\chi_A \zeta_B)' \partial_x v_i ]^2+ [\chi_A \zeta_B \partial_x^2v_i]^2\\
					&+2 (\chi_A \zeta_B)''  (\chi_A \zeta_B)' \partial_x( v_i^2) +[\chi_A^2 \zeta_B ^2]'  \partial_x[(\partial_x v_i)^2]+2 (\chi_A \zeta_B)''  \chi_A \zeta_B v_i  \partial_x^2v_i.
					\end{aligned}
					\]
					Then,
					\begin{equation}
	\begin{aligned}
	(\chi_A \zeta_B \partial_x^2 v_i)^2
	=&~{} (\partial_x^2 z_i)^2
	 - ((\chi_A \zeta_B)'' v_i)^2
	 - 4((\chi_A \zeta_B)' \partial_x v_i)^2\\
	&-2(\chi_A \zeta_B)'' (\chi_A \zeta_B)' \partial_x (v_i^2)
	-2(\chi_A \zeta_B)'' \chi_A \zeta_B v_i  \partial_x^2 v_i\\
	&-[\chi_A^2 \zeta_B ^2]'  \partial_x[(\partial_x v_i)^2]. 
	\end{aligned}
	\end{equation}
	{
			\color{black}
	Now,
	\begin{equation}
	\begin{aligned}
	& \int  R(x)(\chi_A \zeta_B \partial_x^2 v_i)^2\\
	&= \int R(x)(\partial_x^2 z_i)^2 +\int R(x) \left( -[(\chi_A \zeta_B)'' ]^2 v_i^2 - 4[(\chi_A \zeta_B)']^2 (\partial_x v_i )^2 \right)\\
					&\quad +\int R(x)  \left[ {}~ -( [(\chi_A \zeta_B)']^2)' \partial_x( v_i^2) -[(\chi_A \zeta_B)^2]'  \partial_x[(\partial_x v_i)^2]-2 (\chi_A \zeta_B)''  \chi_A \zeta_B v_i  \partial_x^2v_i \right]\\
	&= \int R(x)(\partial_x^2 z_i)^2 +\int R(x) ( -[(\chi_A \zeta_B)'' ]^2 v_i^2 - 4[(\chi_A \zeta_B)']^2 (\partial_x v_i )^2)\\
					&\quad +\int   \partial_x [ R(x)( [(\chi_A \zeta_B)']^2)' ] v_i^2 
					+\int \partial_x[ R(x)[(\chi_A \zeta_B)^2]']  (\partial_x v_i)^2
					-\int 2 R(x)(\chi_A \zeta_B)''  \chi_A \zeta_B v_i  \partial_x^2v_i.
	\end{aligned}
	\end{equation}
	Since
	\[
	\begin{aligned}
	-\int 2 R(x)(\chi_A \zeta_B)''  \chi_A \zeta_B v_i  \partial_x^2v_i
	=& - \int  \partial_x^2[R(x)(\chi_A \zeta_B)''  \chi_A \zeta_B ]  v_i^2+  2 \int R(x)(\chi_A \zeta_B)''  \chi_A \zeta_B  (\partial_x v_i)^2,
	\end{aligned}
	\]
	we get
	\begin{equation}
	\begin{aligned}
	& \int  R(x)(\chi_A \zeta_B \partial_x^2 v_i)^2\\
	&= \int R(x)(\partial_x^2 z_i)^2 +\int R(x) ( -[(\chi_A \zeta_B)'' ]^2 v_i^2 - 4[(\chi_A \zeta_B)']^2 (\partial_x v_i )^2)\\
					& \quad +\int   \partial_x [ R(x)( [(\chi_A \zeta_B)']^2)' ] v_i^2 
					+\int \partial_x[ R(x)[(\chi_A \zeta_B)^2]']  (\partial_x v_i)^2\\
					& \quad - \int  \partial_x^2[R(x)(\chi_A \zeta_B)''  \chi_A \zeta_B ]  v_i^2+  2 \int R(x)(\chi_A \zeta_B)''  \chi_A \zeta_B  (\partial_x v_i)^2\\
	&= \int R(x)(\partial_x^2 z_i)^2
	+\int [-\partial_x^2[R(x)(\chi_A \zeta_B)''  \chi_A \zeta_B ] +\partial_x [ R(x)( [(\chi_A \zeta_B)']^2)' ] -R(x)[(\chi_A \zeta_B)'' ]^2] v_i^2\\
	&\quad +\int \Big[ 2 R(x)(\chi_A \zeta_B)''  \chi_A \zeta_B+ \partial_x[R(x)[(\chi_A \zeta_B)^2]']  - 4R(x) [(\chi_A \zeta_B)']^2 \Big] (\partial_x v_i)^2.
	\end{aligned}
	\end{equation}
		}
	Now we perform the following splitting: 
		\begin{equation}
	\begin{aligned}
	\int  R(x)(\chi_A \zeta_B \partial_x^2 v_i)^2
	=& \int R(x)(\partial_x^2 z_i)^2\\
	& +\int \bigg[\partial_x [R(x)([(\chi_A \zeta_B)']^2)' ] -  R(x)((\chi_A \zeta_B)'')^2 
	- \partial_x^2[R(x)(\chi_A \zeta_B)'' \chi_A \zeta_B]\bigg] v_i^2\\
	&+\int \bigg[\partial_x [R(x)(\chi_A^2 \zeta_B^2)'] 
	+2  R(x)(\chi_A \zeta_B)'' \chi_A \zeta_B   
	-4 R(x)((\chi_A \zeta_B)' )^2\bigg](\partial_x v_i)^2\\
	=: &~{} R_1+R_2+R_3.
	\end{aligned}
	\end{equation}
	Firstly, we will focus on $R_2$. The term that accompanies to $v_i^2$, holds the following decomposition
	\begin{equation}\label{eq:R2_term}	
	\begin{aligned}
	\partial_x& [R(x)([(\chi_A \zeta_B)']^2)' ] -  R(x)((\chi_A \zeta_B)'')^2 
	- \partial_x^2[R(x)(\chi_A \zeta_B)'' \chi_A \zeta_B]\\
		=&	 \chi_A^2\zeta_B^2 \tilde{R}(x)+\mathcal{E}_2  (R(x),x)\zeta_B^2,
\end{aligned}
\end{equation}
where
	\begin{equation}\label{eq:tildeR}	
\begin{aligned}
\tilde{R}(x)=&	  
-2R(x)\left[\frac{\zeta_B^{(4)}}{\zeta_B}+\frac{\zeta_B'''}{\zeta_B}\frac{\zeta_B'}{\zeta_B}\right]-2	R'(x)\frac{\zeta_B'''}{\zeta_B}
- R''(x)\frac{\zeta_B''}{\zeta_B},
\end{aligned}
\end{equation}
and
	\begin{equation}\label{eq:E2}	
\begin{aligned}
\mathcal{E}_2(R(x),x)
=&
-R(x)
\bigg(
\chi_A^{(4)} \chi_A 
+4\chi_A''' \chi_A \frac{\zeta_B'}{\zeta_B^2}
+6\chi_A'' \chi_A   \frac{\zeta_B''}{\zeta_B}
+2(\chi_A^2)' \frac{\zeta_B'''}{\zeta_B}  
\bigg)\\
&
-R'(x)
\bigg(
2\chi_A''' \chi_A 
+6\chi_A'' \chi_A  \frac{\zeta_B'}{\zeta_B}
+6
\chi_A' \chi_A \frac{\zeta_B''}{\zeta_B}
\bigg) \\
&- R''(x)
\left(
\chi_A'' \chi_A +\frac12 [\chi_A^2]' \frac{\zeta_B'}{\zeta_B}
\right) .
\end{aligned}
\end{equation}
Rewriting $R_2$, we obtain
	\begin{equation}\label{eq:R2}
	\begin{aligned}
	R_2=&\int \tilde{R}(x)z_i^2
	+\int \mathcal{E}_2(R(x),x) \zeta_B^2 v_i^2.
	\end{aligned}
	\end{equation}
	Secondly, for $R_3$, the term that accompanies to $(\partial_x v_i)^2$ satisfies the following decomposition
		\begin{equation}\label{eq:R3}
	\begin{aligned}
	\partial_x  [R(x)(\chi_A^2 \zeta_B^2)'] 
	+2&  R(x)(\chi_A \zeta_B)'' \chi_A \zeta_B 
	-4 R(x)((\chi_A \zeta_B)' )^2\\
	=& P_R(x) \chi_A^2\zeta_B^2 + \mathcal{E}_3(R(x),x)\zeta_B^2,
	\end{aligned}
	\end{equation}
	where
	\begin{equation}\label{eq:PR}
P_R(x)=
R(x) 
\bigg[
4  \frac{\zeta_B''}{ \zeta_B}  
-2 \bigg(\frac{\zeta_B'}{\zeta_B} \bigg)^2\bigg]
+2R'(x) \frac{\zeta_B'}{\zeta_B},
			\end{equation}
and
	\begin{equation}\label{eq:E3}
	\begin{aligned}
		\mathcal{E}_3(R(x),x)
=&R(x)\bigg[ 
 4\chi_A'' \chi_A -2(\chi_A' )^2+ 2\frac{\zeta_B'}{\zeta_B}(\chi_A^2)'  
  \bigg]
+R'(x)(\chi_A^2)' .
		\end{aligned}
	\end{equation}
	Now, by Claim \ref{claim:P_CZ_vi_x}, we have
\begin{equation}\label{eq:R3b}
\begin{aligned}
\int   P_R(x) \chi_A^2 \zeta_B^2(\partial_x v_i)^2
=& 
\int P_R(x) (\partial_x z_i)^2 +
\int   \left[ P_R'(x) \frac{\zeta_B'}{\zeta_B}+P_R(x) 
\frac{\zeta_B''}{\zeta_B}\right] z_i^2\\
&+\frac12\int  \mathcal{E}_1(P_R(x),x)\zeta_B^2 v_i^2,
\end{aligned}
\end{equation}
where $\mathcal{E}_1$ is given by \eqref{eq:E1}.
Finally, we obtain that $R_3$ has the following decomposition
\begin{equation}\label{eq:R3a}
\begin{aligned}
R_3
=&
\int P_R(x) (\partial_x z_i)^2 
+\int   \left[ P_R'(x) \frac{\zeta_B'}{\zeta_B}+P_R(x) \frac{\zeta_B''}{\zeta_B}\right] z_i^2\\
&+\frac12\int \mathcal{E}_1(P_R(x),x)\zeta_B^2 v_i^2
+\int \mathcal{E}_3(R(x),x) \zeta_B^2(\partial_x v_i)^2.
\end{aligned}
\end{equation}
Collecting $R_1$, \eqref{eq:R2} and \eqref{eq:R3a}, we obtain
\begin{equation}
\begin{aligned}
\int R(x) \chi_A^2\zeta_B^2 (\partial_x^2 v_i)^2
=&
 \int R(x)(\partial_x^2 z_i)^2+\int \tilde{R}(x)z_i^2
+\int \mathcal{E}_2(R(x),x) \zeta_B^2  v_i^2\\
&+\int P_R(x) (\partial_x z_i)^2 +
\int   \left[ P_R'(x) \frac{\zeta_B'}{\zeta_B}+P_R(x) 
\frac{\zeta_B''}{\zeta_B}\right] z_i^2\\
&+\frac12\int \mathcal{E}_1(P_R(x),x)\zeta_B^2 v_i^2
+\int \mathcal{E}_3(R(x),x) \zeta_B^2(\partial_x v_i)^2,
\end{aligned}
\end{equation}
	where $\mathcal{E}_1,\mathcal{E}_2,\mathcal{E}_3$ and $P_R$ are given in  \eqref{eq:E1}, \eqref{eq:E2}, \eqref{eq:E3} and  \eqref{eq:PR}, respectively. Finally, the proof of \eqref{falta} is direct. This concludes the proof of the Claim \ref{claim:R_CZ_vi_xx}.
	
	\bigskip

\section{Proof of Lemma \ref{lem:coercivity}}\label{ap:proof_coercivity}

\begin{proof}
	We claim that for all $v\in H^1(\R)$ that satisfies $\Jap{\LL \phi_0}{v}=0$, one has
	\[
	\Jap{\LL v}{v}\geq 0.
	\]
	Then the conclusion is evident since $\Jap{Q'}{v}=0$. Suppose that for some nonzero $u\in H^1(\R)$ with $\Jap{\LL \phi_0}{u}=0$, we have $\Jap{\LL u }{u}<0$.
	Then, since $\phi_0$ satisfies \eqref{eq:phi0}, 
	\[
	\begin{aligned}
	\Jap{\LL \phi_0}{\phi_0}&
	=\nu_0^{2}\Jap{\partial_x^{-2} \phi_0}{\phi_0}
	=\nu_0^{2}\Jap{\partial_x^{-2}\phi_0}{\partial_x \partial_x^{-1}\phi_0}
	=-\nu_0^{2}\|\partial_x^{-1}\phi_0\|_{L^2}^2<0.
	\end{aligned}
	\]
	Then we observe that the quadratic form  $(\LL \cdot,\cdot )$ is negative definite in $\hbox{span}(\phi_0, u)$. Since $\langle \phi_0,Q\rangle \neq 0$ (see Lemma \ref{no_soy_orto}), there exists $u_0\in span(\phi_0, u)$ such that $u_0\perp Q$ and $\Jap{\LL u_0}{u_0}<0$.
	This is a contradiction with the result
	\[
	\inf_{\Jap{v}{Q}=0} \Jap{\LL v}{v}=0.
	\]
	(See Proposition 2.9 in \cite{weinstein_modulation} for more details.)
\end{proof}


\end{document}